\documentclass[12pt]{amsart}
\usepackage{}
\usepackage{amsmath}
\usepackage{amsfonts}
\usepackage{amssymb}
\usepackage[all,cmtip]{xy}           
\usepackage{bbm}
\usepackage{bbding}
\usepackage{txfonts}
\usepackage[shortlabels]{enumitem}
\usepackage{ifpdf}
\ifpdf
\usepackage[colorlinks,final,backref=page,hyperindex]{hyperref}
\else
\usepackage[colorlinks,final,backref=page,hyperindex,hypertex]{hyperref}
\fi
\usepackage{amscd}
\usepackage{tikz-cd}
\usepackage[active]{srcltx}
\usepackage{tikz}
\usetikzlibrary{calc}
\usetikzlibrary{arrows,shapes,chains}

\makeatletter

\newtheorem{thm}{Theorem}[section]
\newtheorem{prop}[thm]{Proposition}
\newtheorem{lem}[thm]{Lemma}

\newtheorem{prop-def}{Proposition-Definition}[section]

\theoremstyle{definition}
\newtheorem{defn}[thm]{Definition}

\newtheorem{remark}[thm]{Remark}
\newtheorem{exam}[thm]{Example}

\newcommand{\nc}{\newcommand}

\nc{\delete}[1]{{}}
\nc{\mmargin}[1]{}
\nc{\Alg}{\mathrm{Alg}}
\nc{\AvO}{\mathrm{AvO}}
\nc{\AvA}{\mathrm{AvA}}
\nc{\rmH}{\mathrm{H}}
\nc{\DT}{\mathrm{DT}}
\nc{\C}{\mathrm{C}}

\nc{\mlabel}[1]{\label{#1}}  
\nc{\mcite}[1]{\cite{#1}}  
\nc{\mref}[1]{\ref{#1}}  
\nc{\mbibitem}[1]{\bibitem{#1}} 

\delete{
	\nc{\mlabel}[1]{\label{#1}  
		{\hfill \hspace{1cm}{\bf{{\ }\hfill(#1)}}}}
	\nc{\mcite}[1]{\cite{#1}{{\bf{{\ }(#1)}}}}  
	\nc{\mref}[1]{\ref{#1}{{\bf{{\ }(#1)}}}}  
	\nc{\mbibitem}[1]{\bibitem[\bf #1]{#1}} 
	
}

 \font\cyrs=wncyr7

\nc{\vep}{\varepsilon}
\nc{\bin}[2]{ (_{\stackrel{\scs{#1}}{\scs{#2}}})}  
\nc{\binc}[2]{(\!\! \begin{array}{c} \scs{#1}\\
		\scs{#2} \end{array}\!\!)}  
\nc{\bincc}[2]{  ( {\scs{#1} \atop
		\vspace{-1cm}\scs{#2}} )}  
\nc{\oline}[1]{\overline{#1}}
\nc{\mapm}[1]{\lfloor\!|{#1}|\!\rfloor}
\nc{\bs}{\bar{S}}
\nc{\la}{\longrightarrow}
\nc{\ot}{\otimes}
\nc{\rar}{\rightarrow}
\nc{\lon }{\,\rightarrow\,}
\nc{\dar}{\downarrow}
\nc{\dap}[1]{\downarrow \rlap{$\scriptstyle{#1}$}}
\nc{\defeq}{\stackrel{\rm def}{=}}
\nc{\dis}[1]{\displaystyle{#1}}
\nc{\dotcup}{\ \displaystyle{\bigcup^\bullet}\ }
\nc{\hcm}{\ \hat{,}\ }
\nc{\hts}{\hat{\otimes}}
\nc{\hcirc}{\hat{\circ}}
\nc{\lleft}{[}
\nc{\lright}{]}
\nc{\curlyl}{\left \{ \begin{array}{c} {} \\ {} \end{array}
	\right .  \!\!\!\!\!\!\!}
\nc{\curlyr}{ \!\!\!\!\!\!\!
	\left . \begin{array}{c} {} \\ {} \end{array}
	\right \} }
\nc{\longmid}{\left | \begin{array}{c} {} \\ {} \end{array}
	\right . \!\!\!\!\!\!\!}
\nc{\ora}[1]{\stackrel{#1}{\rar}}
\nc{\ola}[1]{\stackrel{#1}{\la}}
\nc{\scs}[1]{\scriptstyle{#1}} \nc{\mrm}[1]{{\rm #1}}
\nc{\dirlim}{\displaystyle{\lim_{\longrightarrow}}\,}
\nc{\invlim}{\displaystyle{\lim_{\longleftarrow}}\,}
\nc{\dislim}[1]{\displaystyle{\lim_{#1}}} \nc{\colim}{\mrm{colim}}
\nc{\mvp}{\vspace{0.3cm}} \nc{\tk}{^{(k)}} \nc{\tp}{^\prime}
\nc{\ttp}{^{\prime\prime}} \nc{\svp}{\vspace{2cm}}
\nc{\vp}{\vspace{8cm}}
\nc{\modg}[1]{\!<\!\!{#1}\!\!>}
\nc{\intg}[1]{F_C(#1)}
\nc{\lmodg}{\!<\!\!}
\nc{\rmodg}{\!\!>\!}
\nc{\cpi}{\widehat{\Pi}}
\nc{\ssha}{{\mbox{\cyrs X}}} 
\nc{\tsha}{{\mbox{\cyrt X}}}
\nc{\shpr}{\diamond}    
\nc{\labs}{\mid\!}
\nc{\rabs}{\!\mid}

\nc{\RBO}{{\mathrm{RBO}_\lambda}}
\nc{\Sh}{\mathrm{Sh}}
\nc{\RBA}{{\mathrm{RBA}_\lambda}}
\nc{\sgn}{\mathrm{sgn}}

\nc{\ad}{\mrm{ad}}
\nc{\ann}{\mrm{ann}}
\nc{\Aut}{\mrm{Aut}}
\nc{\bim}{\mbox{-}\mathsf{Bimod}}
\nc{\br}{\mrm{bre}}
\nc{\can}{\mrm{can}}
\nc{\Cont}{\mrm{Cont}}
\nc{\rchar}{\mrm{char}}
\nc{\cok}{\mrm{coker}}
\nc{\de}{\mrm{dep}}
\nc{\dtf}{{R-{\rm tf}}}
\nc{\dtor}{{R-{\rm tor}}}

\nc{\Div}{{\mrm Div}}
\nc{\Diff}{\mrm{DA}}
\nc{\Diffl}{\mathsf{DA}_\lambda}
\nc{\diffo}{{\mathsf{DO}_\lambda}}
\nc{\alg}{\mathsf{Alg}}
\nc{\End}{\mrm{End}}
\nc{\Ext}{\mrm{Ext}}
\nc{\Fil}{\mrm{Fil}}
\nc{\Fr}{\mrm{Fr}}
\nc{\Frob}{\mrm{Frob}}
\nc{\Gal}{\mrm{Gal}}
\nc{\GL}{\mrm{GL}}
\nc{\Hom}{\mrm{Hom}}
\nc{\Hoch}{\mrm{Hoch}}

\nc{\hsr}{\mrm{H}}
\nc{\hpol}{\mrm{HP}}
\nc{\id}{\mrm{Id}}
\nc{\im}{\mrm{im}}
\nc{\Id}{\mrm{Id}}
\nc{\ID}{\mrm{ID}}
\nc{\Irr}{\mrm{Irr}}
\nc{\incl}{\mrm{incl}}
\nc{\length}{\mrm{length}}
\nc{\NLSW}{\mrm{NLSW}}
\nc{\Lie}{\mrm{Lie}}
\nc{\mchar}{\rm char}
\nc{\mpart}{\mrm{part}}
\nc{\ql}{{\QQ_\ell}}
\nc{\qp}{{\QQ_p}}
\nc{\rank}{\mrm{rank}}
\nc{\rcot}{\mrm{cot}}
\nc{\rdef}{\mrm{def}}
\nc{\rdiv}{{\rm div}}
\nc{\rtf}{{\rm tf}}
\nc{\rtor}{{\rm tor}}
\nc{\res}{\mrm{res}}
\nc{\SL}{\mrm{SL}}
\nc{\Spec}{\mrm{Spec}}
\nc{\tor}{\mrm{tor}}
\nc{\Tr}{\mrm{Tr}}
\nc{\tr}{\mrm{tr}}
\nc{\wt}{\mrm{wt}}
\def\ot{\otimes}

\nc{\bfk}{{\bf k}}
\nc{\bfone}{{\bf 1}}
\nc{\bfzero}{{\bf 0}}
\nc{\detail}{\marginpar{\bf More detail}
	\noindent{\bf Need more detail!}
	\svp}
\nc{\gap}{\marginpar{\bf Incomplete}\noindent{\bf Incomplete!!}
	\svp}
\nc{\FMod}{\mathbf{FMod}}
\nc{\Int}{\mathbf{Int}}
\nc{\Mon}{\mathbf{Mon}}
 \nc{\sproof}{\noindent{  \textit{Sketch of Proof:} }}
\nc{\remarks}{\noindent{\bf Remarks: }}
\nc{\Rep}{\mathbf{Rep}}
\nc{\Rings}{\mathbf{Rings}}
\nc{\Sets}{\mathbf{Sets}}
\nc{\ob}{\mathsf{Ob}}

\nc{\BA}{{\mathbb A}}   \nc{\CC}{{\mathbb C}}
\nc{\DD}{{\mathbb D}}   \nc{\EE}{{\mathbb E}}
\nc{\FF}{{\mathbb F}}   \nc{\GG}{{\mathbb G}}
    \nc{\LL}{{\mathbb L}}
\nc{\NN}{{\mathbb N}}   \nc{\PP}{{\mathbb P}}
\nc{\QQ}{{\mathbb Q}}   \nc{\RR}{{\mathbb R}}
\nc{\TT}{{\mathbb T}}   \nc{\VV}{{\mathbb V}}
\nc{\ZZ}{{\mathbb Z}}   \nc{\TP}{\widetilde{P}}
\nc{\m}{{\mathbbm m}}

\nc{\cala}{{\mathcal A}}    \nc{\calc}{{\mathcal C}}
\nc{\cald}{\mathcal{D}}     \nc{\cale}{{\mathcal E}}
\nc{\calf}{{\mathcal F}}    \nc{\calg}{{\mathcal G}}
\nc{\calh}{{\mathcal H}}    \nc{\cali}{{\mathcal I}}
\nc{\call}{{\mathcal L}}    \nc{\calm}{{\mathcal M}}
\nc{\caln}{{\mathcal N}}    \nc{\calo}{{\mathcal O}}
\nc{\calp}{{\mathcal P}}    \nc{\calr}{{\mathcal R}}
\nc{\cals}{{\mathcal S}}    \nc{\calt}{{\Omega}}
\nc{\calv}{{\mathcal V}}    \nc{\calw}{{\mathcal W}}
\nc{\calx}{{\mathcal X}}

\nc{\fraka}{{\mathfrak a}}
\nc{\frakb}{\mathfrak{b}}
\nc{\frakg}{{\frak g}}
\nc{\frakl}{{\frak l}}
\nc{\fraks}{{\frak s}}
\nc{\frakB}{{\frak B}}
\nc{\frakm}{{\frak m}}
\nc{\frakM}{{\frak M}}
\nc{\frakp}{{\frak p}}
\nc{\frakW}{{\frak W}}
\nc{\frakX}{{\frak X}}
\nc{\frakS}{{\frak S}}
\nc{\frakA}{{\frak A}}
\nc{\frakx}{{\frakx}}
\nc{\red}{\color{red}}
\nc{\RB}{\mathfrak{RB}}

\begin{document}

\title[Homotopy Rota-Baxter algebras]{Deformations and homotopy theory of Rota-Baxter algebras of any weight}

\author{Kai Wang and Guodong Zhou}
\address{  School of Mathematical Sciences, Shanghai Key Laboratory of PMMP,
  East China Normal University,
 Shanghai 200241,
   China}
\email{wangkai@math.ecnu.edu.cn }

\email{gdzhou@math.ecnu.edu.cn}

\date{\today}

\begin{abstract} This paper studies the formal deformations and homotopy of Rota-Baxter algebras of any given weight. We define an $L_\infty$-algebra that controls simultaneous the deformations of the associative
product and the Rota-Baxter operator of a Rota-Baxter algebra. As a consequence, we develop a cohomology theory of Rota-Baxter algebras of any given weight and justify it by interpreting the lower degree cohomology groups as formal deformations and abelian extensions. The notion of homotopy Rota-Baxter algebras is
introduced and it is shown that the operad governing homotopy Rota-Baxter algebras is a minimal
model of the operad of Rota-Baxter algebras.

\end{abstract}

\subjclass[2010]{
16E40   
16S80   
16S70   
}

\keywords{cohomology, abelian extension, formal deformation,  $L_\infty$-algebra,  minimal model, operad, Rota-Baxter algebra, homotopy Rota-Baxter algebra}

\maketitle

 \tableofcontents

\allowdisplaybreaks

\section*{Introduction}

A general philosophy of deformation theory of mathematical structures, as evolved from ideas of Gerstenhaber, Nijenhuis, Richardson, Deligne, Schlessinger, Stasheff, Goldman,   Millson,  is that  the deformation theory of any given
mathematical object can be described  by
a certain differential graded (=dg) Lie algebra or more generally a $L_\infty$-algebra associated to the
mathematical object  (whose underlying complex is called the deformation complex).  This philosophy has been made into a theorem in characteristic zero by J.~Lurie \cite{Lurie}  and J.~Pridham \cite{Pridham}, expressed in terms of infinity categories. It is an important question   to construct explicitly this dg Lie algebra or $L_\infty$-algebra governing deformation theory of this mathematical object.

Another important question about  algebraic structures is to study their homotopy versions, just like   $A_\infty$-algebras for  usual associative algebras.  The nicest   result would be providing a minimal model of the operad governing an algebraic structure.  When this operad  is Koszul, there exists a  general theory, the so-called Koszul duality for operads \cite{GK94}\cite{GJ94}\cite{LV}, which defines a   homotopy version of this algebraic structure via the cobar construction of the Koszul dual cooperad, which, in this case,  is  a minimal model. However, when the operad   is NOT Koszul, essential difficulties arise and   few examples of minimal models   have been  worked out.
For instance, G\'alvez-Carrillo,  Tonks and   Vallette \cite{GCTV12}  gave a cofibrant resolution of the Batatlin-Vilkovisky operad using inhomogeneous Koszul duality theory. However, their cofibrant resolution is not minimal and in another paper of Drummond-Cole and  Vallette \cite{DCV13}, the authors succeeded in finding a minimal model which is a deformation retract of the cofibrant resolution found in the previous paper.  Dotsenko and Khoroshkin \cite{DK13} constructed resolutions for
shuffle monomial operads by the inclusion-exclusion principle and for operads presented by a Gr\"{o}bner
basis \cite{DK10} by deformation of the monomial case.

These two questions, say, describing controlling $L_\infty$-algebras  and constructing homotopy versions,  are closed related. In fact, given a cofibrant resolution, in particular, a minimal model,  of the operad in question, one can form the deformation complex of the algebraic structure and construct its $L_\infty$-structure as explained by Kontsevich and Soibelman \cite{KS} and  van der Laan \cite{VdL1, VdL2}.  This method has been generalised to properads by Markl \cite{Markl}, Merkulov and Vallette \cite{MV1, MV2}, and to colored props by  Fr\'{e}gier, Markl and Yau \cite{FMY09}.

In this paper, we follow  a somehow inverse  direction and make use of  an ad hoc method. Given an algebraic structure on a space $V$ realised as an algebra over an operad, by considering the  formal deformations of this algebraic structure, we first  construct the deformation complex  and find an $L_\infty$-structure on the underlying graded space of this complex such that the Maurer-Cartan elements are in bijection with the algebraic structures on $V$.
When $V$ is graded, we define a homotopy version of this algebraic structure as Maurer-Cartan elements in the $L_\infty$-algebra constructed above. Finally under suitable conditions,  we could  show that the operad governing the homotopy version is a minimal model of the original operad.

\bigskip

 The algebraic structure  investigated in this paper  is  Rota-Baxter algebras of any weight.

Rota-Baxter algebras (previously known as  Baxter algebras) originated with the work of    Baxter \cite{Bax} in his study on probability theory.
Baxter's work was further investigated by,  among others,   Rota \cite{Rota69}  (hence the name  ``Rota-Baxter algebras''),   Cartier \cite{Cartier} and   Atkinson \cite{Atkinson} etc. The subject was revived beginning with the  work of Guo et al. \cite{GuoKeigher00a, GuoKeigher00b,Guo00, Aguiar}.
Nowadays, Rota-Baxter algebras have  numerous  applications and connections to many  mathematical branches,  to name a few, such as combinatorics \cite{GuGuo,Rota95},  renormalization
in quantum field theory   \cite{CK}, multiple zeta values in number theory \cite{Guozhang},   operad theory \cite{Aguiar,BBGN},
  Hopf algebras \cite{CK}, Yang-Baxter equation \cite{Bai}. For basic theory about Rota-Baxter algebras, we refer the reader to the short introduction \cite{Guo09b} and to  the comprehensive monograph \cite{Guo12}.

The deformation theory  and cohomology theory of  Rota-Baxter algebras had been absent for a long time despite   the   importance of Rota-Baxter algebras.  Recently there are some breakthroughs in this direction.
  Tang,  Bai,  Guo and  Sheng \cite{TBGS19} developed the deformation theory  and cohomology theory of $\mathcal{O}$-operators (also called relative Rota-Baxter operators) on Lie algebras, with applications
to  Rota-Baxter Lie algebras in mind.  Das \cite{Das20}  developed a similar theory for Rota-Baxter associative algebras of weight zero.  Lazarev,  Sheng and   Tang  \cite{LST} succeeded  in establishing   deformation theory  and cohomology theory of  relative  Rota-Baxter Lie algebras of weight zero   and found  applications to triangular Lie bialgebras.   They determined the $L_\infty$-algebra that controls deformations of a relative Rota-
Baxter Lie algebra and  introduced  the notion
of a homotopy relative Rota-Baxter Lie algebra.  The same group of authors also related homotopy relative Rota-Baxter Lie algebras and triangular $L_\infty$-bialgebras via  a functorial approach to Voronov's higher derived brackets construction \cite{LST2}.  Later Das and Misha also  determined the $L_\infty$-structures underlying the cohomology theory for Rota-Baxter associative algebras of weight zero \cite{DasM}.   There are some other related work  \cite{TSZ20, THS21, Das20b,Das20c,Das20d, DasGuo,DasGuob,JiangSheng21, LiuLiuSheng21}.
These work all concern    Rota-Baxter  operators of weight zero.

A recent paper by  Pei,  Sheng,  Tang and   Zhao \cite{PSTZ} considered  cohomologies of crossed homomorphisms for Lie algebras and they found a DGLA controlling deformations of crossed homomorphisms. Another  exciting progress in this subject is the introduction of the  notion of Rota-Baxter Lie groups by Guo, Lang and Sheng \cite{GLS21}; as a successor to  this work, Jiang, Sheng and Zhu considered cohomology of Rota-Baxter operators of weight $1$ on Lie groups and Lie algebras and relationship between them \cite{JSZ21}. While this paper is ready to submit,  another paper  appeared \cite{Das21} in which Das investigated cohomology of Rota-Baxter operators of arbitrary weights on associative algebras and which has some overlap with  Sections 5.1 and 6.2 of this paper; in another paper \cite{Das21b}, he studied twisted Rota-Baxter operators on Leibniz algebras.   It seems that these are the only papers which investigates Rota-Baxter operators of nonzero weight (for a related work on differential algebras of nonzero weight, see \cite{GLSZ20}).   In these papers, the authors   dealt with  the deformations of only the Rota-Baxter operators with the Lie algebra or associative algebra structure unchanged.
The goal of the present paper is to study simultaneous deformations of Rota-Baxter operators of nonzero weight and of    associative algebra structures.  One of the reasons is that when  one structure remains undeformed, the homotopy version obtained could not be a minimal model  of the operad of Rota-Baxter Lie algebras or Rota-Baxter associative  algebras.

\bigskip

In this paper, we follow  a somehow inverse  direction to the classical approach from cofibrant resolutions to $L_\infty$-algebras.  Given an algebraic structure on a space $V$ realised as an algebra over an operad, by considering the  formal deformations of this algebraic structure, we first  construct the deformation complex  and find an $L_\infty$-structure on the underlying graded space of this complex such that the Maurer-Cartan elements are in bijection with the algebraic structures on $V$.
When $V$ is graded, we define a homotopy version of this algebraic structure as Maurer-Cartan elements in the $L_\infty$-algebra constructed above. Finally under suitable conditions,  we could  show that the operad governing the homotopy version is a minimal model of the original operad.

While the above mentioned papers of Sheng et al.  use derived brackets \cite{Kos, Vor, Vor2} as a main tool, our method is somehow ad hoc. We give a direct proof of   the constructed $L_\infty$-structure.
Finally, we could show that the resulting homotopy version is indeed the  minimal model of the operad of Rota-Baxter associative  algebras.
It might be appropriate to explain  here the relationship of our result with the   paper of  Dotsenko and  Khoroshkin \cite{DK13}. In that paper, the authors tried to deform the minimal model of the corresponding monomial operads obtained by Gr\"{o}bner basis of  the Rota-Baxter operad and they got the generators of the operad of homotopy Rota-Baxter algebras. It seems that it is not easy to obtain all the relations.
While our generators of homotopy Rota-Baxter algebras are the same, we could determine all the relations in an indirect way with the aid  of $L_\infty$-structure on the deformation complex  found using our ad hoc method. However,   our method to verify the minimal model was inspired by Dotsenko and  Khoroshkin \cite{DK13}.  Dotsenko kindly pointed out another proof based the paper \cite{DK13}; see Remark~\ref{rem: dotsenko}.

Another remark is in order. Once we show that the obtained operad of homotopy Rota-Baxter algebras is the minimal model of that of Rota-Baxter algebras, one could use the method of  Kontsevich and Soibelman \cite{KS} and  van der Laan \cite{VdL1, VdL2} to produce the $L_\infty$-structure on the deformation complex instead of our direct method. We offer two reasons. One reason is that the direct method is not much more complicated than the latter method and another reason is that we want to exhibit the effectiveness of  our ad hoc method to deal with deformation theory and homotopy theory of algebraic structures. 

It would be an interesting problem  to give a general approach for  operated algebras in the sense of Guo \cite{Guo09a} and other algebaic structures.  For instance, it is desirable to  deal with  differential algebras (continuing \cite{GLSZ20}) and  averaging algebras (continuing and completing \cite{WangZhouAv}) etc. We are working on these projects using our method.

\bigskip

This paper is organised as follows. The first section contains some preliminaries.
Section 2 recalls the language of differential graded Lie algebras and $L_\infty$-algebras. Associative algebras are taken as baby model of our method in the third section.
Basic definitions and facts about Rota-Baxter algebras which are mostly  well known are recalled in  Section 4.
After defining a cohmology complex of Rota-Baxter operators,  with the help of the usual Hochschild cocohain complex,   a cochain complex,  whose cohomology groups  should control deformation theory of Rota-Baxter algerbas,  is exhibited in Section 5. We  justify this cohomology theory by interpreting lower degree cohomology groups as formal deformations (Section 6) and abelian extensions of Rota-Baxter algebras (Section 7).   Rota-Baxter algebra structures over the underlying space of this cochain complex is then realized as the Maurer-Cartan elements of an    L-infinity algebra structure over the cochain complex, as is done in the eighth section. With the help of this L-infinity algebra, one  introduces the notion of homotopy Rota-Baxter algebras of any weight in the ninth section.  Finally it is shown that the operad governing homotopy Rota-Baxter algebras is a minimal model of the operad of Rota-Baxter algebras in the tenth section.   We postpone the lengthy proof of the central result Theorem~\ref{Thm: rb-L-infty} to  Appendix A and Appendix B contains a proof of another technical result Proposition~\ref{Prop: homotopy RB-arising-from-homotopy-Rota-Baxter}

\bigskip

\section{Preliminaries}\label{Sect: Preliminaries}

 Throughout this paper, let $\bfk$ be a field of characteristic $0$.  All vector spaces are defined over $\bfk$,  all  tensor products and Hom-spaces  are taken over $\bfk$.

 A   (homologically) graded   vector space  is a family of vector spaces $V=\{V_n\}_{n\in \ZZ}$  indexed by integers.   Elements of $\cup_{n\in \ZZ} V_n$ are called homogeneous and
the  degree  of $v\in V_n$ is written as $|v|:=n$.

 We use both homological and cohomological  gradings. For a homologically graded   space $V=\bigoplus_{n\in \mathbb{Z}} V_n$, write $V^n=V_{-n}$ will transform homological grading to cohomological grading and vice versa.

  Let $V$ and $W$ be graded vector spaces.
A graded map  $f: V\to W$ of degree $r$ is by definition a linear map $f: V\to W$ such that  $f(V_n)\subseteq  W_{n+r}$ for all $n$.  In this case, denote  $|f|=r$.
Write  $$\mathrm{Hom}(V, W)_r=\prod_{p\in \ZZ} \Hom(V_p, W_{p+r})$$   the space of graded maps of degree $r$ and
denote $$\mathrm{Hom}(V, W)=\{\mathrm{Hom}(V, W)_r\}_{r\in \ZZ}$$ to be  the graded space of graded linear maps from $V$ to $W$.

   Let $V$ and $W$ be graded vector spaces. The  tensor product  $V\otimes W$ of $V$ and $W$ is
  graded
  whose grading is given by
\[(V\otimes W)_n:=\bigoplus_{p+q=n} V_p \otimes W_q.\]


Denote by $\Bbbk s$  the $1$-dimensional  graded vector space  spanned by $s$ with $|s|=1$.
The  suspension  of $V$ is $sV:=\Bbbk s\otimes V$,
so $(sV)_i $ can be identified with $V_{i-1}$ for any $i\in \ZZ$.
Note that for $v\in V_n$, $sv\in sV$ is of degree $n+1$ and the map $s: V\to sV, v\mapsto sv$ is a graded map of degree $1$.
One can also define another $1$-dimensional  graded vector space $\Bbbk s^{-1}$ with  $|s^{-1}|=-1$.
The  desuspension  of $V$ is $s^{-1}V: =\Bbbk s^{-1}\otimes V$ and the desuspension   map $s^{-1}: V\to s^{-1}V, v\mapsto s^{-1}v$ is a graded map of degree $-1$.

We will encounter many signs in the graded world. The basic principle to determine signs is the so-called Koszul rule,  that is, when we exchange the positions of two graded objects in an expression, we need to multiply the expression by a power of $-1$ whose  exponent is  the product of their degrees.
For instance, given two graded maps
$f: V\to V', g: W\to W'$, define
$f\otimes g: V\otimes W \to V'\otimes W'$  via $$(f\otimes g)(v\otimes w)=(-1)^{|g|\cdot |v|} f(v)\otimes g(w). $$
Another example is given as follows: for four graded maps
$f, f': V\to V', g, g': W\to W'$, the composition of $f\otimes g$ and $  f'\otimes g'$ is defined to be
\[ (f\otimes g)\circ (f'\otimes g')=(-1)^{|g|\cdot |f'|} (f\circ f')\otimes (g\circ g'). \]

 For $v_1, \dots, v_n\in V$, write
$v_{1, n}:=v_1\ot  \cdots \ot v_n\in V^{\ot n}$ and also $sv_{1, n}=sv_1\ot  \cdots \ot sv_n\in (sV)^{\ot n}$.

Let  $n\geq 1$. Recall $S_n$ denotes the symmetric group in $n$ variables.
For $0\leq i_1, \dots, i_r\leq n$ with $i_1+\cdots+i_r=n$,  $\Sh(i_1, i_2,\dots,i_r)$ is the   set of $(i_1,\dots, i_r)$-shuffles, i.e., those permutation $\sigma\in S_n$ such that
		 $$\sigma(1)<\sigma(2)<\dots<\sigma(i_1),  \ \sigma(i_1+1)< \dots<\sigma(i_1+i_2),\ \dots,\
		\sigma(i_{r-1}+1)< \cdots<\sigma(n).$$
The following fact is well known:
\begin{lem}{\label{Lem: permutation}}  Let $n \geqslant   1, 1\leqslant i\leqslant n-1$. Then for any $\delta\in S_n$, there exists a unique triple $(\tau,\sigma,\pi)$ with $\sigma\in \Sh(i,n-i),\tau\in S_i, \pi\in S_{n-i}$ such that $\delta(l)=\sigma\tau(l)$ for $1\leqslant l\leqslant i$, and $\delta(i+m)=\sigma(i+\pi(m))$ for $1\leqslant m\leqslant n-i$.
\end{lem}

Let $V$ be a graded vector space. Define the graded symmetric algebra $S(V)$ of $V$ to be $T(V)/I$ where the two-sided ideal $I$ is generated by
$x\ot y -(-1)^{|x||y|}y\ot x$ for all homogeneous elements $x, y\in V$. For $x_1\ot \cdots\ot x_n\in T(V)$,  write $x_1\odot \cdots \odot x_n$  to be the corresponding element in $S(V)$.
Define the weight of  $x_1\odot \cdots \odot x_n$ to be $n$, so $S(V)$ is weight graded whose  weight  $n$-th component  is written as $S(V)^{(n)}, n\geq 0$.

For homogeneous elements $x_1,\dots,x_n \in V$ and $\sigma\in S_n$, the Koszul sign $\epsilon(\sigma;  x_1,\dots, x_n)$ is defined by
\begin{equation} \label{Eq: epsilon sign} x_1\odot x_2\odot\dots\odot x_n=\epsilon(\sigma;  x_1,\dots,x_n)x_{\sigma(1)}\odot x_{\sigma(2)}\odot\dots\odot x_{\sigma(n)}\in S(V).\end{equation}

Define the graded exterior algebra $\Lambda(V)$ of $V$ to be $T(V)/J$ where the two-sided ideal $J$ is generated by
$x\ot y +(-1)^{|x||y|}y\ot x$ for all homogeneous elements $x, y\in V$.
For $x_1\ot \cdots\ot x_n\in T(V)$, write the corresponding element in $\Lambda(V)$ as $x_1\wedge \cdots  \wedge x_n$. Define the weight of  $x_1\wedge  x_2\wedge\cdots  \wedge x_n$ to be $n$, so $\Lambda(V)$ is weight graded whose  weight  $n$-th component  is written as $\Lambda(V)^{(n)}, n\geq 0$.
For homogeneous elements $x_1,\dots,x_n \in V$ and $\sigma\in S_n$, the Koszul sign $\chi(\sigma;  x_1,\dots, x_n)$ is defined by \begin{equation} \label{Eq: chi sign origin} x_1\wedge x_2\wedge\cdots\wedge x_n=\chi(\sigma;  x_1,\dots,x_n)x_{\sigma(1)}\wedge x_{\sigma(2)}\wedge\cdots\wedge x_{\sigma(n)}\in S(V).\end{equation}
Obviously
\begin{equation}\label{Eq: chi sign} \chi(\sigma;  x_1,\dots,x_n)=\sgn(\sigma)\epsilon(\sigma;  x_1,\dots,x_n),\end{equation}
where $\sgn(\sigma)$ is the sign of $\sigma$.

Fix an isomorphism  $S(sV)^{(n)}\cong s^n\Lambda(V)^{(n)} $ by sending
$sx_1\odot \cdots \odot sx_n$ to $$(-1)^{\sum\limits_{k=1}^{n-1}\sum\limits_{j=1}^k|x_j|} s^n(x_1\wedge x_2\wedge \cdots  \wedge x_n).$$
Under this isomorphism, we have the equality:
\begin{equation} \label{Eq: interpaly epsilon and chi}\chi(\sigma; x_1,\dots,x_n)(-1)^{\sum\limits_{k=1}^{n-1}\sum\limits_{j=1}^k|x_{\sigma(j)}|}=\varepsilon(\sigma; sx_1,\dots,sx_n)(-1)^{\sum\limits_{k=1}^{n-1}\sum\limits_{j=1}^k|x_j|},  \end{equation}
for any $\sigma\in S_n$ and homogeneous elements $x_1, \dots,  x_n\in V$.

For permutations $\delta, \sigma, \pi, \tau$ appearing in Lemma~\ref{Lem: permutation}, we have the following equality:
\begin{eqnarray}\label{Eq: sign formulae for shuffles}
\chi(\delta; x_1,\dots,x_n)=\chi(\sigma; x_1,\dots,x_n)\chi(\pi; x_{\sigma(i+1)},\dots,x_{\sigma(i+n-i)})\chi(\tau; x_{\sigma(1)},\dots,x_{\sigma(i)}),
	\end{eqnarray}
whose proof is left to the reader.

\bigskip

\section{Differential graded Lie algebras and $L_\infty$-algebras}
\label{Sect: DGLAs and Linfinity algebras}

In this section, we will recall some preliminaries on differential graded Lie algebras and $L_\infty$-algebras. For more background  on differential graded Lie algebras and $L_\infty$-algebras, we refer the reader to \cite{Sta92, LS93, LM95,Get09}.

\subsection{Differential graded Lie algebras and Maurer-Cartan elements}\

\begin{defn}\label{Def: DGLA} A differential graded (=dg) Lie algebra is a triple $(L, l_1, l_2)$ where $L=\bigoplus\limits_{i\in \mathbb{Z}}L_i $  is a graded $\bfk$-space,
  $l_1:L\rightarrow L$ and  $l_2:L^{\ot 2}\rightarrow L$ are  two graded linear maps with $|l_1|=-1$ and $|l_2|=0$, subject to  the following conditions:
\begin{itemize}
	\item[(i)] $l_1\circ l_1=0$;
	\item[(ii)] $l_1\circ l_2=l_2\circ (l_1\ot \id+\id\ot l_1)$;
	\item[(iii)] (anti-symmetry)  $l_2(x\ot y)+(-1)^{|x||y|}l_2(y\ot x)=0, \forall x,y\in L;$
	\item[(iv)] (Jacobi identity)
$$l_2(l_2(x\ot y)\ot z)+(-1)^{|x|(|y|+|z|)}l_2(l_2(y\ot z)\ot x)+(-1)^{|z|(|x|+|y|)}l_2(l_2(z\ot x)\ot y)=0, \forall x, y, z\in L.$$
\end{itemize}
When $l_1=0$, the pair $(L,l_2)$ is called a graded Lie algebra.
 \end{defn}

\begin{defn}\label{Def: MC element in DGLA}
	Let $(L,l_1,l_2)$ be a dg Lie algebra. An element $\alpha\in L_{-1}$ is called a Maurer-Cartan element if it satisfies the  Maurer-Cartan equation
	\begin{eqnarray}\label{Eq: MC equation for DGLA} l_1(\alpha)-\frac{1}{2}l_2(\alpha\ot \alpha)=0.\end{eqnarray}
\end{defn}

Given an arbitrary  Maurer-Cartan element in a dg Lie algebra, one can  get a new dg Lie algebra by twisting  the original dg Lie algebra structure using  this element .
\begin{lem}\label{Lem: twist dgla}
Let $(L,l_1,l_2)$ be a dg Lie algebra and $\alpha\in L_{-1}$ be a Maurer-Cartan element. Define new operations $l_1^\alpha$ and $l_2^{\alpha}$ on $L$ as
	 \begin{eqnarray}\label{Eq: twist dgla} l_1^\alpha(x) = l_1(x)-l_2(\alpha\ot x), \forall x\in L  \ \mathrm{and}\  l_2^\alpha = l_2.	\end{eqnarray}
		 Then $(L,l_1^\alpha,l_2^\alpha)$ is   a dg Lie algebra as well. This new dg Lie algebra  is called the twisted dg Lie algebra (by the Maurer-Cartan element  $\alpha$).
	\end{lem}

\subsection{$L_\infty$-algebras and Maurer-Cartan elements}\



\begin{defn}\label{Def:L-infty}
	Let $L=\bigoplus\limits_{i\in\mathbb{Z}}L_i$ be a graded space over $\bfk$. Assume that $L$ is endowed with a family of graded linear operators $l_n:L^{\ot n}\rightarrow L, n\geqslant 1$ with  $|l_n|=n-2$ subject to  the following conditions:
for any $n\geqslant 1$,  $ \sigma\in S_n$ and $x_1,\dots, x_n\in L$,
	\begin{enumerate}
		\item[(i)](generalised anti-symmetry) $$l_n(x_{\sigma(1)}\ot \cdots \ot x_{\sigma(n)})=\chi(\sigma; x_1,\dots,   x_n)\ l_n(x_1 \ot \cdots  \ot x_n);$$

		\item[(ii)](generalised Jacobi identity)
$$\sum\limits_{i=1}^n\sum\limits_{\sigma\in \Sh(i,n-i)}\chi(\sigma; x_1, \dots, x_n)(-1)^{i(n-i)}l_{n-i+1}(l_i(x_{\sigma(1)}\ot \cdots \ot x_{\sigma(i)})\ot x_{\sigma(i+1)}\ot \cdots \ot x_{\sigma(n)})=0,$$
		where recall that  $\Sh(i,n-i)$ is the set of $(i,n-i)$
		shuffles.
	\end{enumerate}
	Then $(L,\{l_n\}_{n\geqslant1})$ is called an $L_\infty$-algebra.
\end{defn}

\begin{remark} \label{Rem: L-infinity for small n}   Let us consider the generalised Jacobi identity for   $n\leqslant 3$ with the assumption of generalised anti-symmetry.
	
	\begin{enumerate}
		\item[(i)]  $n=1$,  $l_1\circ l_1=0$, that is,  $l_1$ is a differential,

		\item[(ii)]  $n=2$, $l_1\circ l_2=l_2\circ (l_1\ot\Id+\Id\ot l_1)$, that is   $l_1$ is a derivation for  $l_2$,

		 \item[(iii)] $n=3$, for homogeneous elements $x_1, x_2, x_3\in L$
 $$\begin{array}{ll} &l_2(l_2(x_1\ot x_2)\ot x_3)+(-1)^{|x_1|(|x_2|+|x_3|)} l_2(l_2(x_2\ot x_3)\ot x_1)+
(-1)^{|x_3|(|x_1|+|x_2|)} l_2(l_2(x_3\ot x_1)\ot x_2)
\\
=&-\Big(l_1(l_3(x_1\ot x_2\ot x_3))+ l_3(l_1 (x_1)\ot x_2\ot x_3 )+(-1)^{|x_1|} l_3(x_1\ot l_1 (x_2)\ot x_3 )+\\
&(-1)^{|x_1|+|x_2|} l_3(x_1\ot x_2\ot l_1 (x_3) )\Big),\end{array}$$
that is, $l_2$ satisfies the   Jacobi identity up to homotopy.
	\end{enumerate}

 In particular,    if all   $l_n=0$ with  $n\geqslant 3$, then $(L,l_1,l_2)$ is just a dg Lie algebra.

\end{remark}

One can also define Maurer-Cartan elements in $L_\infty$-algebras.
\begin{defn}\label{Def: MC element in L infinity algebra}
	Let $(L,\{l_n\}_{n\geqslant1})$ be an $L_\infty$-algebra. An element $\alpha\in L_{-1}$ is called a Maurer-Cartan element if it satisfies the Maurer-Cartan equation:
	\begin{eqnarray}\label{Eq: mc-equation}\sum_{n=1}^\infty\frac{1}{n!}(-1)^{\frac{n(n-1)}{2}} l_n(\alpha^{\ot n})=0,\end{eqnarray}
whenever this infinite sum exists.
\end{defn}

Lemma~\ref{Lem: twist dgla} can be generalised to $L_\infty$-algebras.
\begin{prop}[Twisting procedure]\label{Prop: deformed-L-infty}
Given a Maurer-Cartan element $\alpha$ in $L_\infty$-algebra $L$, one can introduce  a new $L_\infty$-structure $\{l_n^\alpha\}_{n\geqslant 1}$ on graded space $L$, where $l_n^{\alpha}: L^{\ot n}\rightarrow L$ is defined as :
	\begin{eqnarray}\label{Eq: twisted L infinity algebra} l^\alpha_n(x_1\ot \cdots\ot x_n)=\sum_{i=0}^\infty\frac{1}{i!}(-1)^{in+\frac{i(i-1)}{2}}l_{n+i}(\alpha^{\ot i}\ot x_1\ot \cdots\ot x_n),\ \forall x_1, \dots, x_n\in L,\end{eqnarray}
whenever these infinite sums exist. The new $L_\infty$-algebra $(L, \{l_n^\alpha\}_{n\geqslant 1})$ is called the twisted $L_\infty$-algebra (by the Maurer-Cartan element $\alpha$).
\end{prop}

\begin{remark}

\begin{itemize}

\item[(i)]The signs in Definition~\ref{Def: MC element in L infinity algebra} and Proposition~\ref{Prop: deformed-L-infty} are different from those appearing in \cite{LST}, as the conventions in \cite{LST} are essentially about $L_\infty[1]$-algebras \cite{Schatz, Vor, Vit}.   We refer the reader to \cite{Vit} for the translation  between $L_\infty$-structures and $L_\infty[1]$-structures.

\item[(ii)] Proposition~\ref{Prop: deformed-L-infty} is essentially contained in \cite[Section 4]{Get09}.
Notice that here  we only ask the existence of the infinite sums, although in many references, nilpotent  or weakly filtered  $L_\infty$ algebras \cite{Get09,LST} are used to guarantee the convergence of these sums.

\end{itemize}

\end{remark}

\bigskip

\section{Formal deformations,  Hochschild cohomology and homotopy theory of  associative algebras}
In this section, we will recall the formal deformations and Hochschild cohomology of associative algebras. We will see how the dg Lie algebra structure  on the underlying graded space of the Hochschild cochain complex introduced by Gerstenhaber will enable defining $A_\infty$-algebras which is the homotopy version of  associative algebras.

 This example is our baby model for deformation theory and homotopy theory of Rota-Baxter algebras.

\subsection{Hochschild cohomology of associative algebras}\

Let $(A,\mu)$ be an associative $\bfk$-algebra. We often write  $\mu(a\ot b)=a
\cdot b=ab$   for any $a,b\in A$.
Let  $M$ be a bimodule over $A$. The Hochschild cochain complex of $A$ with coefficients in $M$ is $$\C^\bullet_{\mathrm{Alg}}(A,M):=\bigoplus\limits_{n=0}^\infty \C^n_{\mathrm{Alg}}(A,M),$$ where $\C^n_{\mathrm{Alg}}(A,M)=\Hom(A^{\ot n},M)$ and the differential $\delta^n: \C^n_{\mathrm{Alg}}(A,M)\rightarrow \C^{n+1}_{\mathrm{Alg}}(A,M)$ is defined as:
$$\delta^n(f)(a_{1, n+1}  )=  (-1)^{n+1} a_1f(a_{2, n+1})+\sum\limits_{i=1}^n(-1)^{n-i+1}f(a_{1, i-1}\ot a_i\cdot a_{i+1}\ot  a_{i+2, n+1})\\
	 + f(a_{1, n})a_{n+1}
$$
for all $f\in \C^n_{\mathrm{Alg}}(A,M), a_1,\dots,a_{n+1}\in A$.

 The cohomology of the Hochschild cochain complex $\C^\bullet_{\mathrm{Alg}}(A,M)$ is called the Hochschild cohomology of $A$ with coefficients in $M$,  denoted by $\mathrm{HH}^\bullet(A,M)$.
When the bimodule $M$ is the regular bimodule $A$ itself, we just denote $\C^\bullet_{\mathrm{Alg}}(A,A)$ by $\C^\bullet_{\mathrm{Alg}}(A)$ and call it the Hochschild cochain complex of associative algebra $(A, \mu)$.  Denote the cohomology $\mathrm{HH}^\bullet(A, A)$ by $\mathrm{HH}^\bullet(A)$, called the Hochschild cohomology of associative algebra $(A,\mu)$.

\subsection{Formal deformations of associative algebras}\

Given  an associative $\bfk$-algebra $(A,\mu)$, consider
  $k[[t]]$-bilinear associative products   on
$$A[[t]]=\{\sum_{i=0}^\infty a_it^i\ | \ a_i\in A, \forall i\geqslant 0\}.$$  Such a product is determined by
$$\mu_t =\sum_{i=0}^{\infty} \mu_i  t^i: A \otimes A\to A[[t]], $$   where for all $i\geqslant 0$, $\mu_i: A\otimes A\to A$  are linear maps. When $\mu_0=\mu$, we say that $\mu_t$ is a formal deformation of $\mu$ and
$\mu_1$ is called the   infinitesimal  of  formal deformation $\mu_t$.

The only constraint is the  associativity of $\mu_t$: $$\mu_t(\mu_t(a\ot b)\ot c)=\mu_t(a\ot \mu_t(b\ot c)), \forall a, b, c\in A$$
where is equivalent to the following family of equations:
\begin{eqnarray}\label{Eq: Deform product}
\sum_{i+j=n\atop i, j\geqslant 0} (\mu_i(\mu_j(a\ot b)\ot c)-\mu_i(a\ot \mu_j(b\ot c))=0, \forall a, b, c\in A, n\geqslant 0.  \end{eqnarray}
Looking closely at the cases  $n=0$ and $n=1$,  one obtains:
\begin{itemize}
	\item[(i)] when $n=0$,  $(a \cdot b) \cdot c = a \cdot  (b\cdot  c), \forall a, b, c\in A$, which is exactly the associativity of $\mu$;
	
	\item[(ii)]  when $n=1$,  $$a \mu_1(b\ot c)-\mu_1(ab\ot c)+\mu_1(a\ot b c)-\mu_1(a\ot b)c=0,\forall a, b, c\in A,$$ which says that  the infinitesimal $\mu_1$ is a 2-cocycle in the Hochschild cochain complex $\C^\bullet_{\Alg}(A)$.
	
\end{itemize}

In general, we can rewrite Equation~(\ref{Eq: Deform product}) as
$$\delta^2(\mu_n)=\frac{1}{2}\sum_{i=1}^{n-1} [\mu_i, \mu_{n-i}]_G$$
where $[-, -]_G$ is the Gerstenhaber bracket; see the next subsection.

%
%
%
%
%
%
%
%
%
%
%
%

\subsection{Gerstenhaber   brackets}\  \label{subsec: Gerstenhaber brackets}


Let $V=\bigoplus\limits_{i\in \mathbb{Z}}V_i$ be a graded space and recall that  $sV$ denotes the suspension of $V$, i.e., $(sV)_n=V_{n-1}, \forall n\in \mathbb{Z}$. The  free conilpotent tensor coalgebra $T^c(sV)$ is defined to
$$T^c(sV)=\bfk\oplus sV\oplus (sV)^{\ot 2}\oplus \cdots \oplus  (sV)^{\ot n} \oplus \cdots $$
with the usual  deconcaternation coproduct.
Let $ \mathfrak{C}_{\Alg}(V)=\Hom(T^c(sV),sV)$.

Let $m\geqslant 1, 1\leqslant n\leqslant m$. Given homogeneous elements  $f\in \Hom((sV)^{\ot m},V)$  and $ g_i\in \Hom((sV)^{\ot l_i},V), i=1.\dots, n$ with all $l_i\geqslant 0$,
then $sf,sg_1,sg_2,\dots,sg_n\in \mathfrak{C}_{\Alg}(V)=\Hom(T^c(sV),sV)$, define the brace operation $$sf \{sg_1,\dots,sg_n\}\in \Hom((sV)^{\ot u},sV)$$ to be
\begin{align*}&sf \{sg_1,\dots,sg_n\}(sa_{1,  u})= \sum\limits_{   0\leqslant i_1\leqslant i_1+l_1\leqslant  i_2\leqslant i_2+l_2\leqslant \dots \leqslant i_n\leqslant i_n+l_n\leqslant u}(-1)^{\xi}\\
&\quad \quad \quad \quad 	sf(sa_{1, i_1}\ot sg_1(sa_{i_1+1, i_1+l_1})\ot sa_{i_1+l_1+1, i_2} \ot sg_2(sa_{i_2+1, i_2+l_2 }) \ot \cdots \ot  sg_{n}(sa_{i_n+1, i_n+l_n})\ot sa_{i_n+l_n+1, u} ),
	\end{align*}
where $a_1, \dots, a_u\in V$, $u=m+l_1+\dots+l_n-n  $ and $ \xi=\sum\limits_{k=1}^n(|g_k|+1)(\sum\limits_{j=1}^{i_k}(|a_j|+1))$.

In particular, for homogeneous elements  $f\in \Hom((sV)^{\ot m},V)$ with $m\geqslant 1$   and $ g \in \Hom((sV)^{\ot n},V)$ with  $n\geqslant 0$, for  each $1\leqslant i\leqslant m$, write
$$sf\circ_i sg=sf\circ (\Id^{\otimes (i-1)}\ot sg \ot \Id^{\otimes (m-i)}).$$
These notations will be very useful while dealing with operads.

For two homogeneous  elements $sf\in \Hom((sV)^{\ot m},sV),sg\in \Hom((sV)^{\ot n},sV)$, define
\begin{eqnarray}\label{Eq: Gerstahaber bracket} [sf,sg]_G=sf\{sg\}-(-1)^{(|f|+1)(|g|+1)}sg\{sf\},\end{eqnarray}
called the Gerstenhaber bracket of $sf$ and $sg$.
\begin{thm}[\cite{Ge63}]
For any given graded space $V$, the Gerstenhaber bracket makes the graded space $\mathfrak{C}_{\Alg}(V)$ into a graded Lie algebra.
\end{thm}

Moreover, the brace operation   on $\mathfrak{C}_{\Alg}(V)$ satisfies the following pre-Jacobi identities:
\begin{prop}[\cite{Ge63, Get93, GV95}]
		For any homogeneous elements $sf, sg_1,\dots, sg_m,  sh_1,\dots,sh_n$ in $\mathfrak{C}_{\Alg}(V)$, the following identity holds:
	\begin{eqnarray}
		\label{Eq: pre-jacobi}	&&\Big(sf \{sg_1,\dots,sg_m\}\Big)\{sh_1,\dots,sh_n\}=\\
		\notag &&\quad  \sum\limits_{0\leqslant i_1\leqslant j_1\leqslant i_2\leqslant j_2\leqslant \dots \leqslant i_m\leqslant  j_m\leqslant n}(-1)^{\sum\limits_{k=1}^m(|g_k|+1)(\sum\limits_{j=1}^{i_k}(|h_j|+1))}
sf\{sh_{1, i_1},  sg_1\{sh_{i_1+1, j_1}\},\dots,   sg_m\{sh_{i_m+1, j_m} \}, sh_{j_m+1,  n}\}.
	\end{eqnarray}
\end{prop}

\medskip

The following two results are well known and we include    sketch  of  proofs to fix the notations.
Fix two  isomorphisms
 \begin{eqnarray}\label{Eq: first can isom} \Hom((sV)^{\ot n}, sV) \simeq \Hom(V^{\ot n}, V),  f\mapsto \tilde{f}:=    s^{-1}\circ  f \circ s^{\ot n}\end{eqnarray}
 for $f\in \Hom((sV)^{\ot n},  sV)$  and
\begin{eqnarray}\label{Eq: second can isom}\Hom((sV)^{\ot n},  V) \simeq \Hom(V^{\ot n}, V), g\mapsto \hat{g}:=     g \circ  s^{\ot n} \end{eqnarray} for $g\in \Hom((sV)^{\ot n},  V)$
\begin{prop}\label{Prop: bijection between associative product and MC elements}
Let $V$ be an ungraded space considered as a graded space concentrated in degree 0. Then there is a bijection between the set of Maurer-Cartan elements in the graded Lie algebra $\mathfrak{C}_{\Alg}(V)$ and the set of associative algebra structure on space $V$.
	\end{prop}
 \sproof  Since $V$ is concentrated in degree 0, the degree $-1$ part of $\mathfrak{C}_{\Alg}(V)$ is $\Hom((sV)^{\ot 2}, sV)$. Given an element  $\alpha\in \mathfrak{C}_{\Alg}(V)$ of degree $-1$, we define an operation $\mu: V^{\ot 2}\rightarrow V$ as $$\mu=\widetilde{ \alpha}= s^{-1}\circ \alpha\circ (s\ot s).$$ Then it can be checked that the fact that  $\alpha$ satisfying the Maurer-Cartan equation in  graded Lie algebra $\mathfrak{C}_{\Alg}(V)$ is equivalent the associativity of the operation $\mu$.
	

\begin{prop}\label{Prop: Twisted cx is orginal cx}
Let $(A, \mu)$ be an associative algebra and $\alpha$ be the corresponding Maurer-Cartan element in $\mathfrak{C}_{\Alg}(A)$. Then the underlying complex of the twisted dg Lie algebra $(\mathfrak{C}_{\Alg}(A), l_1^\alpha, l_2^\alpha)$ is  exactly    $s\C^\bullet_\Alg(A)$, the shift of the Hochschild cochain complex of   associative algebra $A$. 	
\end{prop}

\sproof
Recall   that $\alpha=-s\circ \mu\circ(s^{-1}\ot s^{-1}):  sV\ot sV\rightarrow sV$.
Then one checks that
$$\widetilde{l_1^{\alpha}( f)}=-\widetilde{ [\alpha,  f]_G}= -\delta^n(\tilde{f})$$
for any $f\in \Hom((sA)^{\ot n},  sA)$.
This shows that the complex $(\mathfrak{C}_{\Alg}(A), l_1^\alpha)$ is  isomorphic to    $s\C^\bullet_\Alg(A)$, the shift of the Hochschild cochain complex of   associative algebra $A$.

 \subsection{Defining $A_\infty$-algebras via Maurer-Cartan elements}\
\label{A-infinity algebras}


    Let $V$ be a graded vector space.
    Recall that the reduced cofree conilpotent coalgebra is  $$ \overline{T^c}(sV)=sV\oplus (sV)^{\otimes 2}\oplus \cdots $$ with the usual  deconcaternation coproduct. Define $ \overline{\mathfrak{C}_{\Alg}}(V) =\Hom(\overline{T^c} (sV),sV)$. It is easy to verify that $ \overline{\mathfrak{C}_{\Alg}}(V)$ is  a  graded Lie subalgebra of $\mathfrak{C}_{\Alg}(V)$.
\begin{defn}\label{Def: A infinity algebras}
 An  $A_\infty$-algebra structure  on graded space $V$ is defined to be a Maurer-Cartan element in graded Lie algebra $  \overline{\mathfrak{C}_{\Alg}}(V)$.
   \end{defn}
   By   definition, an $A_\infty$-algebra structure on $V$ consists of  a family of operators $$b_n:(sV)^{\ot n}\rightarrow sV, \forall n\geqslant1$$ with $|b_n|=-1$ satisfying
   $$\sum_{j=1}^n b_{n-j+1}\{ b_{j}\}=0,\forall n\geqslant 1.$$
  Define operators $m_n=\widetilde{b_n}=s^{-1}\circ b_n\circ s^{\ot n}: V^{\ot n}\rightarrow V$. Then one can get the following equivalent definition of $A_\infty$-algebras, which is the original definition due to Stasheff.
\begin{defn}[\cite{Sta63}] An  $A_\infty$-algebra structure  on $V$ consists of a family of operators $$m_n: V^{\ot n}\rightarrow V, n\geqslant 1$$ with $|m_n|=n-2$, which fulfill the   Stasheff identities:
		\begin{eqnarray}\label{Eq: Stasheff} \sum_{i+j+k=n\atop i, k\geqslant 0, j\geqslant 1}(-1)^{i+jk}m_{i+1+k}\circ(\id^{\ot i}\ot m_j\ot \id^{\ot k})=0,\forall n\geqslant 1.\end{eqnarray}
	\end{defn}	

For later use, we record here the definition of $A_\infty$-morphisms.
\begin{defn}\label{Def: A-infinity morphism}
	Let $V$, $W$ be two $A_\infty$-algebras and $b=\sum_{i\geqslant 1}b_i\in \overline{\mathfrak{C}_{\Alg}}(V), b'=\sum_{i\geqslant 1}b_i'\in \overline{\mathfrak{C}_{\Alg}}(W)$ be the corresponding Maurer-Cartan elements respectively. An $A_\infty$-morphism $\phi$ from $V$ to $W$ consists of a family of operators $\phi_i:(sV)^{\ot i}\rightarrow sW, i\geqslant 1$ of degree $0$ satisfying the following equations:
	\begin{eqnarray}\label{Eq: Ainfinity morphism} \sum_{i+j+k=n\atop  i, k\geqslant 0, j\geqslant 1} \phi_{i+1+k}(\id^{\ot i}\ot b_j\ot \id^{\ot k})=\sum_{m\geqslant 1}\sum_{i_1+\dots+i_m=n\atop i_1, \dots, i_m\geqslant 1}b_m'\circ(\phi_{i_1}\ot \dots\ot \phi_{i_m}), \forall n\geqslant 1.\end{eqnarray}
\end{defn}

\begin{remark}\label{Remark: curved A infinity and Koszul duality}

\begin{itemize}
	\item[(i)] In Definition~\ref{Def: A infinity algebras}, we use the reduced version  $  \overline{\mathfrak{C}_{\Alg}}(V)$ to define $A_\infty$-algebras, while  Maurer-Cartan elements in the full version $  \mathfrak{C}_{\Alg}(V)$ would give curved $A_\infty$-algebras \cite{GJ90}.

\item[(ii)] We introduce $A_\infty$-algebras using the naive approach in this section. However, one can use Koszul duality theory for operads \cite{GK94, LV} instead, as the operad governing $A_\infty$-algebras is a minimal cofibrant resolution of the operad of associative algebras in the model category of operads \cite{Hinich, BM03}.

Nevertheless, as we will see soon in the next section,  the operad of Rota-Baxter algebras is NOT quadratic, so it seems that the  Koszul duality theory for operads \cite{GK94}\cite{LV}  can not apply directly. This is why we adopt the naive approach in this paper while developing  homotopy theory of Rota-Baxter algebras.

\end{itemize}
\end{remark}

\section{Rota-Baxter algebras and Rota-Baxter bimodules}

In this section, we recall some basic definitions and facts about Rota-Baxter algebras.

\begin{defn}\label{Def: Rota-Baxter algebra}
Let $(A, \mu=\cdot)$ be an associative algebra over field $\bfk$ and
$\lambda\in \bfk$. A linear
operator $T: A\rightarrow A$ is said to be a Rota-Baxter operator of
weight $\lambda$ if it satisfies
\begin{eqnarray}\label{Eq: Rota-Baxter relation}
T(a)\cdot T(b)=T\big(a\cdot T(b)+T(a)\cdot b+\lambda\  a\cdot
b\big)\end{eqnarray}	
for any $a,b \in A$, or in terms of maps
\begin{eqnarray}\label{Eq: Rota-Baxter relation in terms of maps}
\mu\circ (T \ot T)=T\circ (\Id\ot T +T \ot \Id)+\lambda\ T\circ \mu.\end{eqnarray} In this case,  $(A,\mu,T)$ is called a Rota-Baxter
algebra of weight $\lambda$. Denote by $\RBA$ the category of Rota-Baxter algebras of weight $\lambda$ with obvious morphisms.
\end{defn}

\begin{remark}\label{Remark: Koszul duality not applied to RB algebras}
  As mentioned by Remark~\ref{Remark: curved A infinity and Koszul duality} (ii),  although the associativity of $\mu$ is quadratic, the defining relation Equation~\eqref{Eq: Rota-Baxter relation in terms of maps} of   Rota-Baxter operator $T$ is not quadratic and  not even homogeneous when $\lambda\neq 0$,  so   the  Koszul duality theory for operads \cite{GK94}\cite{LV} could not be applied directly to develop a cohomology theory of Rota-Baxter algebras. 

\end{remark}


\begin{defn}\label{Def: Rota-Baxter bimodules}
Let $(A,\mu,T)$ be a Rota-Baxter algebra and $M$ be a bimodule
over associative algebra $(A,\mu)$. We say that $M$ is a bimodule
over Rota-Baxter algebra $(A,\mu, T)$  or a Rota-Baxter bimodule if  $M$ is endowed with a
linear operator $T_M: M\rightarrow M$ such that the following
equations
\begin{eqnarray}T(a) T_M(m)&=&T_M\big(aT_M(m)+T(a)m+\lambda am\big),\\
T_M(m)T(a)&=&T_M\big(mT(a)+T_M(m)a+\lambda ma\big).
\end{eqnarray}
hold for any $a\in A$ and $m\in M$.
\end{defn}
Of course, $(A,\mu, T)$ itself is a bimodule over the Rota-Baxter algebra $(A,\mu,T)$, called the regular Rota-Baxter bimodule.

 The following   result  is easy  whose  proof is left to the reader:
\begin{prop}  \label{Prop: trivial extension of Rota-Baxter bimodule}
Let $(A,\mu,T)$ be a Rota-Baxter algebra and $M$ be a bimodule
over associative algebra $(A,\mu)$.  It is well known that  $A\oplus M$
	 becomes  an associative algebra whose  multiplication is
	\begin{eqnarray}(a,m)(b,n)=(a\cdot b, an+mb).\end{eqnarray}  Write $\iota: A\to A\oplus M, a\mapsto (a, 0)$ and $\pi: A\oplus M\to A, (a, m)\mapsto a$. Then $A\oplus M$ is a Rota-Baxter algebra such that $\iota$ and $\pi$ are both morphisms of Rota-Baxter algebras if and only if $M$ is a Rota-Baxter bimodule over $A$.

 This new Rota-Baxter algebra will be denoted by $A\ltimes  M$, called the semi-direct product (or trivial extension) of $A$ by $M$.
	\end{prop}
In fact, we will see that the above result is a special case of Propositions~\ref{Prop: new RB bimodules from abelian extensions} and \ref{prop:2-cocycle}.

There is a  definition of bimodules over two Rota-Baxter algebras in \cite{QGG}.

\begin{remark} One can use   monoid objects in certain slice categories to justify Definition~\ref{Def: Rota-Baxter bimodules} following \cite{DMZ}. In fact, one can show an equivalence
  between the category of monoids in the slice category $\RBA/A$ and that of Rota-Baxter bimodules over   a   Rota-Baxter  algebra $A$.
\end{remark}


Recall first the  following interesting observation:
\begin{prop}[{\cite[Theorem 1.1.17]{Guo12}}]\label{Prop: new RB algebra}
	Let $(A,\mu,T)$ be a Rota-Baxter algebra. Define a new binary operation as:
\begin{eqnarray}a\star  b:=a\cdot T(b)+T(a)\cdot b+\lambda a\cdot b\end{eqnarray}
for any $a,b\in A$. Then
\begin{itemize}

\item[(i)]   the operation $\star $ is associative and  $(A,\star )$ is a new  associative algebra;

    \item[(ii)] the triple  $(A,\star ,T)$ also forms a Rota-Baxter algebra of weight $\lambda$  and denote it by $A_\star $;

\item[(iii)] the map $T:(A,\star , T)\rightarrow (A,\mu, T)$ is a  morphism of Rota-Baxter  algebras.
    \end{itemize}
	\end{prop}

One can also construct new Rota-Baxter bimodules from old ones.
\begin{prop}\label{Prop: new-bimodule}
Let $(A,\mu,T)$ be a Rota-Baxter algebra of weight $\lambda$ and $(M,T_M)$ be a Rota-Baxter bimodule over it. We define a left action $``\rhd"$ and a right action $``\lhd"$ of $A$ on $M$ as follows: for any $a\in A,m\in M$,
\begin{eqnarray}
a\rhd m:&=& T(a)m-T_M(am),\\
m\lhd a:&=& mT(a)-T_M(ma).
\end{eqnarray}
Then these actions make $M$ into a Rota-Baxter bimodule over $A_\star $ and denote this new  bimodule by $_\rhd M_\lhd$.
\end{prop}

\begin{proof}
	Firstly, we show that $(M,\rhd)$ is a left module over $(A,\star )$.
	\begin{eqnarray*}
	a\rhd(b\rhd m)&=& a\rhd(T(b)m-T_M(bm))\\
	&=&T(a)\big(T(b)m-T_M(bm)\big)-T_M\big(aT(b)m-aT_M(bm)\big)\\
	&=&T(a)T(b)m-T_M\big(a  T_M(bm)+T(a)bm+\lambda abm\big)\\
	&\ &-T_M\big(aT(b)m\big)+T_M\big(aT_M(bm)\big)\\
	&=&T(a)T(b)m-T_M(T(a)bm)-\lambda T_M(abm)-T_M(aT(b)m),\\
	\\
	(a\star  b)\rhd m&=&T(a\star  b)m-T_M\big((a\star  b)m\big)\\
	&=&T(a)T(b)m-T_M\big(aT(b)m+T(a)bm+\lambda abm \big).
	\end{eqnarray*}
So we have
\[	a\rhd(b\rhd m)=(a\star  b)\rhd m.\]
Thus the operation $\rhd$ makes $M$ into a left module over $(A,\star )$. Similarly, one can check that operation $\lhd$ defines a right module structure on $M$ over $(A,\star )$.

Now, we are going to check the compatibility of operations $\rhd$ and $\lhd$. We have the following equations:
\begin{eqnarray*}
	(a\rhd m)\lhd b&=& \big(T(a)m-T_M(am)\big)\lhd b\\
	&=& \big(T(a)m-T_M(am)\big)T(b)-T_M\big(T(a)mb-T_M(am)b\big)\\
&=&T(a)mT(b)-T_M\big(T_M(am)b+amT(b)+\lambda amb\big)\\
&\ &-T_M\big(T(a)mb\big)+T_M\big(T_M(am)b\big)\\
&=& T(a)mT(b)-T_M\big(amT(b)\big)-\lambda T_M(amb)-T_M(T(a)mb)  ,\\
\\
a\rhd(m\lhd b)&=&a\rhd(mT(b)-T_M(mb))\\
&=&T(a)\big(mT(b)-T_M(mb)\big)-T_M\big(amT(b)-aT_M(mb)\big)\\
&=&T(a)mT(b)-T_M\big(aT_M(mb)+T(a)mb+\lambda amb\big)\\
&\ &-T_M\big(amT(b)\big)+T_M\big(aT_M(mb)\big)\\
&=& T(a)mT(b)-T_M\big(T(a)mb\big)-\lambda T_M\big(amb\big)-T_M\big(amT(b)\big).
	\end{eqnarray*}
Thus we have \[	(a\rhd m)\lhd b=a\rhd(m\lhd b),\]
that is, operations $\rhd$ and $\lhd$ make $M$ into a bimodule over associative algebra $(A,\star )$.

Finally, we show that $_\rhd M_\lhd$ is a Rota-Baxter bimodule over $A_\star $. that is, for any $a\in A$ and $m\in M$,  $$
\begin{array}{rcl} T(a)\rhd T_M(m)&=&T_M\big(a\rhd T_M(m)+T(a)\rhd m+\lambda a\rhd m\big),\\
T_M(m)\lhd T(a)&=&T_M\big(m \lhd T(a)+T_M(m) \lhd a+\lambda m\lhd a\big).\end{array}
$$
We only prove the first equality, the second being similar.

In fact,
$$\begin{array}{rcl}
T(a)\rhd T_M(m)&=&T^2(a)T_M(m)-T_M(T(a)T_M(m))\\
&=&  T_M (T(a)T_M(m) + T^2(a)m +\lambda  T(a)m )-T_M \big(aT_M(m)+T(a)m+\lambda am\big)  \\
&=& T_M\big(   T^2(a)m +\lambda  T(a)m \big)
\end{array}$$
and
$$\begin{array}{rl}
 & T_M\big(a\rhd T_M(m)+T(a)\rhd m+\lambda a\rhd m\big),\\
=&T_M\big( T(a)T_M(m)-T_M(aT_M(m))+ T^2(a)m-T_M(T(a)m)+\lambda  T(a)m-\lambda T_M(am)\big)\\
=& T_M\big( T_M\big(aT_M(m)+T(a)m+\lambda am\big) -T_M(aT_M(m))+ T^2(a)m-T_M(T(a)m)\\&+\lambda  T(a)m-\lambda T_M(am)\big)\\
=& T_M\big(   T^2(a)m +\lambda  T(a)m \big)\\
=& T(a)\rhd T_M(m).
\end{array}$$
	\end{proof}

\bigskip

\section{Cohomology theory of Rota-Baxter algebras} \label{Sect: Cohomology theory of Rota-Baxter algebras}
In this  section, we will define a cohomology theory for   Rota-Baxter algebras of any weight.

\subsection{Cohomology of Rota-Baxter Operators}\
\label{Subsect: cohomology RB operator}

Firstly, let's study the cohomology of  Rota-Baxter
operators.

 Let $(A, \mu, T)$ be a
Rota-Baxter algebra and $(M,T_M)$ be a Rota-Baxter bimodule over it. Recall that
Proposition~\ref{Prop: new RB algebra} and Proposition~\ref{Prop: new-bimodule}  give a new
associative algebra  $A_\star $ and
  a new   Rota-Baxter bimodule  $_\rhd M_\lhd$ over $A_\star $.
 Consider the Hochschild cochain complex of $A_\star $ with
 coefficients in $_\rhd M_\lhd$:
 $$\C^\bullet_{\mathrm{Alg}}(A_\star , {_\rhd
 	M_\lhd})=\bigoplus\limits_{n=0}^\infty \C^n_{\mathrm{Alg}}(A_\star , {_\rhd
 	M_\lhd}).$$
  More precisely,  for $n\geqslant 0$,  $ \C^n_{\mathrm{Alg}}(A_\star , {_\rhd M_\lhd})=\Hom  (A^{\ot n},M)$ and its differential $$\partial^n:
 \C^n_{\mathrm{Alg}}(A_\star ,\  _\rhd M_\lhd)\rightarrow  \C^{n+1}_{\mathrm{Alg}}(A_\star , {_\rhd M_\lhd}) $$ is defined as:
 \begin{align*}&\partial^n(f)(a_{1, n+1}) \\
 =&(-1)^{n+1} a_1\rhd f(a_{2, n+1})+\sum_{i=1}^n(-1)^{n-i+1}f(a_{1, i-1}\ot a_{i}\star  a_{i+1} \ot   a_{i+2, n+1})
  +f(a_{1, n})\lhd a_{n+1}\\
 =&(-1)^{n+1}\Big(T(a_1)f(a_{2, n+1})-T_M\big(a_1f(a_{2, n+1})\big)\Big)\\
&+\sum_{i=1}^n(-1)^{n-i+1}\Big(f(a_{1, i-1}\ot a_iT(a_{i+1})\ot   a_{i+2, n+1})+f(a_{1, i-1} \ot T(a_i)a_{i+1}\ot   a_{i+2, n+1})
\\ &\quad  +\lambda f(a_{1, i-1} \ot a_ia_{i+1}\ot   a_{i+2, n+1})\Big)\\
&+ \Big(f(a_{1, n})T(a_{n+1})-T_M\big(f(a_{1,n})a_{n+1}\big)\Big)
 \end{align*}
 for any $f\in  \C^n_{\Alg}(A_\star ,\  _\rhd M_\lhd)$ and $a_1,\dots,a_{n+1}\in A$.

 \smallskip

 \begin{defn}
 	Let $A=(A,\mu,T)$ be a Rota-Baxter algebra of weight $\lambda$ and $M=(M,T_M)$ be a Rota-Baxter bimodule over it. Then the cochain complex $(\C^\bullet_\Alg(A_\star, {_\rhd M_\lhd}),\partial)$ is called the cochain complex of Rota-Baxter operator $T$ with coefficients in $(M, T_M)$,  denoted by $C_{\RBO}^\bullet(A, M)$. The cohomology of $C_{\RBO}^\bullet(A,M)$, denoted by $\mathrm{H}_{\RBO}^\bullet(A,M)$, are called the cohomology of Rota-Baxter operator $T$ with coefficients in $(M, T_M)$.
 	
 	 When $(M,T_M)$ is the regular Rota-Baxter bimodule $ (A,T)$, we denote $\C^\bullet_{\RBO}(A,A)$ by $\C^\bullet_{\RBO}(A)$ and call it the cochain complex of  Rota-Baxter operator $T$, and denote $\rmH^\bullet_{\RBO}(A,A)$ by $\rmH^\bullet_{\RBO}(A)$ and call it the cohomology of Rota-Baxter operator $T$.
 \end{defn}

\subsection{Cohomology of Rota-Baxter algebras}\
\label{Subsec:chomology RB}

In this subsection, we will combine the Hochschild cohomology of   associative algebras and the cohomology of Rota-Baxter operators to define a cohomology theory for Rota-Baxter algebras.

Let $M=(M,T_M)$ be a  Rota-Baxter bimodule over a Rota-Baxter algebra of weight $\lambda$ $A=(A,\mu,T)$. Now, let's construct a chain map   $$\Phi^\bullet:\C^\bullet_{\Alg}(A,M) \rightarrow C_{\RBO}^\bullet(A,M),$$ i.e., the following commutative diagram:
\[\xymatrix{
		\C^0_{\Alg}(A,M)\ar[r]^-{\delta^0}\ar[d]^-{\Phi^0}& \C^1_{\Alg}(A,M)\ar@{.}[r]\ar[d]^-{\Phi^1}&\C^n_{\Alg}(A,M)\ar[r]^-{\delta^n}\ar[d]^-{\Phi^n}&\C^{n+1}_{\Alg}(A,M)\ar[d]^{\Phi^{n+1}}\ar@{.}[r]&\\
		\C^0_{\RBO}(A,M)\ar[r]^-{\partial^0}&\C^1_{\RBO}(A,M)\ar@{.}[r]& \C^n_{\RBO}(A,M)\ar[r]^-{\partial^n}&\C^{n+1}_{\RBO}(A,M)\ar@{.}[r]&
.}\]

Define $\Phi^0=\Id_{\Hom(k,M)}=\Id_M$, and for  $n\geqslant 1$ and $ f\in \C^n_{\Alg}(A,M)$,  define $\Phi^n(f)\in \C^n_{\RBO}(A,M)$ as:
\begin{align*}
 &\Phi^n(f)(a_1\ot\cdots \ot a_n) \\
=&f(T(a_1)\ot \cdots \ot T(a_n))\\
&-\sum_{k=0}^{n-1}\lambda^{n-k-1}\sum_{1\leqslant i_1<i_2<\dots<i_k\leqslant n}T_M\circ f(a_{1, i_1-1} \ot T(a_{i_1})\ot a_{i_1+1, i_2-1}\ot T(a_{i_2})\ot \cdots \ot T(a_{i_k})\ot   a_{i_k+1, n}).
\end{align*}

\smallskip

\begin{prop}\label{Prop: Chain map Phi}
	The map $\Phi^\bullet: \C^\bullet_\Alg(A,M)\rightarrow \C^\bullet_{\RBO}(A,M)$ is a chain map.
\end{prop}

This result follows from the $L_\infty$-structure over the cochain complex of Rota-Baxter algebras, so we omit it; see Proposition~\ref{Prop: cohomlogy complex as underlying complex of L infinity algebra}.

\smallskip

\begin{defn}
 Let $M=(M,T_M)$ be a  Rota-Baxter bimodule over a Rota-Baxter algebra of weight $\lambda$ $A=(A,\mu,T)$.  We define the  cochain complex $(\C^\bullet_{\RBA}(A,M), d^\bullet)$  of Rota-Baxter algebra $(A,\mu,T)$ with coefficients in $(M,T_M)$ to the negative shift of the mapping cone of $\Phi^\bullet$, that is,   let
\[\C^0_{\RBA}(A,M)=\C^0_\Alg(A,M)  \quad  \mathrm{and}\quad   \C^n_{\RBA}(A,M)=\C^n_\Alg(A,M)\oplus \C^{n-1}_{\RBO}(A,M), \forall n\geqslant 1,\]
 and the differential $d^n: \C^n_{\RBA}(A,M)\rightarrow \C^{n+1}_{\RBA}(A,M)$ is given by \[d^n(f,g)= (\delta^n(f), -\partial^{n-1}(g)  -\Phi^n(f))\]
 for any $f\in \C^n_\Alg(A,M)$ and $g\in \C^{n-1}_{\RBO}(A,M)$.
The  cohomology of $(\C^\bullet_{\RBA}(A,M), d^\bullet)$, denoted by $\rmH_{\RBA}^\bullet(A,M)$,  is called the cohomology of the Rota-Baxter algebra $(A,\mu,T)$ with coefficients in $(M,T_M)$.
When $(M,T_M)=(A,T)$, we just denote $\C^\bullet_{\RBA}(A,A), \rmH^\bullet_{\RBA}(A,A)$   by $\C^\bullet_{\RBA}(A),  \rmH_{\RBA}^\bullet(A)$ respectively, and call  them the cochain complex, the cohomology of Rota-Baxter algebra $(A,\mu,T)$ respectively.
\end{defn}
There is an obvious short exact sequence of complexes:
\begin{eqnarray}\label{Seq of complexes} 0\to s\C^\bullet_{\RBO}(A,M)\to \C^\bullet_{\RBA}(A,M)\to \C^\bullet_{\Alg}(A,M)\to 0\end{eqnarray}
which induces a long exact sequence of cohomology groups
$$0\to \rmH^{0}_{\RBA}(A, M)\to\mathrm{HH}^0(A, M)\to\rmH^0_{\RBO}(A, M) \to \rmH^{1}_{\RBA}(A, M)\to\mathrm{HH}^1(A, M)\to\cdots$$
$$\cdots\to \mathrm{HH}^p(A, M)\to \rmH^p_{\RBO}(A, M)\to \rmH^{p+1}_{\RBA}(A, M)\to \mathrm{HH}^{p+1}(A, M)\to \cdots$$


\bigskip


\bigskip

\section{Formal deformations of Rota-Baxter algebras and cohomological interpretation}

In this section, we will study formal deformations of Rota-Baxter algebras and interpret  them  via    lower degree   cohomology groups  of Rota-Baxter algebras defined in last section.

\subsection{Formal deformations of Rota-Baxter algebras}\

Let $(A,\mu, T)$ be a Rota-Baxter algebra of weight $\lambda$.   Consider a 1-parameterized family:
\[\mu_t=\sum_{i=0}^\infty \mu_it^i, \ \mu_i\in \C^2_\Alg(A),\quad  T_t=\sum_{i=0}^\infty T_it^i,  \ T_i\in \C^1_{\RBO}(A).\]

\begin{defn}
	A  1-parameter formal deformation of    Rota-Baxter algebra $(A, \mu,T)$ is a pair $(\mu_t,T_t)$ which endows the flat $\bfk[[t]]$-module $A[[t]]$ with a  Rota-Baxter algebra structure over $\bfk[[t]]$ such that $(\mu_0,T_0)=(\mu,T)$.
\end{defn}

 Power series $\mu_t$ and $ T_t$ determine a  1-parameter formal deformation of Rota-Baxter algebra $(A,\mu,T)$ if and only if for any $a,b,c\in A$, the following equations hold :
 \begin{eqnarray*}
 \mu_t(a\ot \mu_t(b\ot c))&=&\mu_t(\mu_t(a\ot b)\ot c),\\
 \mu_t(T_t(a)\ot T_t(b))&=& T_t\Big(\mu_t(a\ot T_t(b))+\mu_t(T_t(a)\ot b)+\lambda \mu_t(a\ot b)\Big).
 \end{eqnarray*}
By expanding these equations and comparing the coefficient of $t^n$, we obtain  that $\{\mu_i\}_{i\geqslant0}$ and $\{T_i\}_{i\geqslant0}$ have to  satisfy: for any $n\geqslant 0$,
\begin{equation}\label{Eq: deform eq for  products in RBA}
	\sum_{i=0}^n\mu_i\circ(\mu_{n-i}\ot \id)=\sum_{i=0}^n\mu_i\circ(\id\ot \mu_{n-i}),\end{equation}
\begin{equation}\label{Eq: Deform RB operator in RBA} \begin{array}{rcl}
\sum_{i+j+k=n\atop i, j, k\geqslant 0}	\mu_{i}\circ(T_j\ot T_{k})&=&\sum_{i+j+k=n\atop i, j, k\geqslant 0} T_{i}\circ \mu_j\circ (\id\ot T_{k})\\
&  &+\sum_{i+j+k=n\atop i, j, k\geqslant 0} T_{i}\circ\mu_j\circ (T_{k}\ot \id)+\lambda\sum_{i+j=n\atop i, j \geqslant 0}T_i\circ\mu_{j}.
\end{array}\end{equation}
Obviously, when $n=0$, the above conditions are exactly the associativity of $\mu=\mu_0$ and Equation~(\ref{Eq: Rota-Baxter relation}) which is the defining relation of Rota-Baxter operator $T=T_0$.

\smallskip

\begin{prop}\label{Prop: Infinitesimal is 2-cocyle}
	Let $(A[[t]],\mu_t,T_t)$ be a  1-parameter formal deformation of
	Rota-Baxter algebra $(A,\mu,T)$ of weight $\lambda$. Then
	$(\mu_1,T_1)$ is a 2-cocycle in the cochain complex
	$C_{\RBA}^\bullet(A)$.
\end{prop}
\begin{proof} When $n=1$,   Equations~(\ref{Eq: deform eq for  products in RBA}) and (\ref{Eq: Deform RB operator in RBA})  become
	$$\mu_1\circ(\mu\ot \id)+\mu\circ(\mu_1\ot \id)=\mu_1\circ(\id\ot \mu)+\mu\circ (\id\ot \mu_1),$$
and
$$\begin{array}{cl}
	&\mu_1 (T\ot T)-\{T\circ\mu_1\circ(\id\ot T)+T\circ\mu_1\circ(T\ot \id)+\lambda T\circ \mu_1\}\\
	=&-\{\mu\circ(T\ot T_1)-T\circ\mu\circ(\id\ot T_1)\}
	+\{T_1\circ\mu\circ(\id\ot T)+T_1\circ\mu\circ(T\ot \id)+\lambda T_1\circ \mu\}\\
	&-\{\mu\circ(T_1\ot T)-T\circ\mu\circ(T_1\ot\id)\},
	\end{array}$$
Note that  the first equation is exactly $\delta^2(\mu_1)=0\in \C^\bullet_{\Alg}(A)$ and that  second equation is exactly to  \[\Phi^2(\mu_1)=-\partial^1(T_1) \in \C^\bullet_{\RBO}(A).\]
	So $(\mu_1,T_1)$ is a 2-cocycle in $\C^\bullet_{\RBA}(A)$.
	\end{proof}

\smallskip

\begin{defn} The 2-cocycle $(\mu_1,T_1)$ is called the infinitesimal of the 1-parameter formal deformation $(A[[t]],\mu_t,T_t)$ of Rota-Baxter algebra $(A,\mu,T)$.
	\end{defn}

In general, we can rewrite Equation~(\ref{Eq: deform eq for  products in RBA} ) and (\ref{Eq: Deform RB operator in RBA}) as
\begin{eqnarray} \label{Eq: general formal of deform product in RBA} \delta^2(\mu_n) = \frac{1}{2}\sum_{i=1}^{n-1} [\mu_i, \mu_{n-i}]_G\end{eqnarray}
\begin{equation} \label{Eq: general formal of deform RBO in RBA}
\begin{array}{rcl}\partial^{1}(T_n)   +\Phi^2(\mu_n)&=& \sum_{i+j+k=n\atop 0 \leqslant i, j, k\leqslant n-1}	\mu_{i}\circ(T_j\ot T_{k})-\sum_{i+j+k=n\atop 0 \leqslant i, j, k\leqslant n-1} T_{i}\circ \mu_j\circ (\id\ot T_{k})\\ &&-\sum_{i+j+k=n\atop 0 \leqslant i, j, k\leqslant n-1} T_{i}\circ\mu_j\circ (T_{k}\ot \id)-\sum_{i+j=n\atop 0 \leqslant i, j\leqslant n-1}T_i\circ\mu_{j}.
\end{array}\end{equation}

\smallskip
\begin{defn}
Let $(A[[t]],\mu_t,T_t)$ and $(A[[t]],\mu_t',T_t')$ be two 1-parameter formal deformations of Rota-Baxter algebra $(A,\mu,T)$. A formal isomorphism from $(A[[t]],\mu_t',T_t')$ to $(A[[t]], \mu_t, T_t)$ is a power series $\psi_t=\sum_{i=0}\psi_it^i: A[[t]]\rightarrow A[[t]]$, where $\psi_i: A\rightarrow A$ are linear maps with $\psi_0=\id_A$, such that:
\begin{eqnarray}\label{Eq: equivalent deformations}\psi_t\circ \mu_t' &=& \mu_t\circ (\psi_t\ot \psi_t),\\
\psi_t\circ T_t'&=&T_t\circ\psi_t. \label{Eq: equivalent deformations2}
	\end{eqnarray}
	In this case, we say that the two 1-parameter formal deformations $(A[[t]], \mu_t,T_t)$ and
	$(A[[t]],\mu_t',T_t')$ are  equivalent.
\end{defn}

\smallskip

Given a Rota-Baxter algebra $(A,\mu,T)$, the power series $\mu_t,T_t$
with $\mu_i=\delta_{i,0}\mu, T_i=\delta_{i,0}T$ make
$(A[[t]],\mu_t,T_t)$ into a $1$-parameter formal deformation of
$(A,\mu,T)$. Formal deformations equivalent to this one are called trivial.
\smallskip

\begin{thm}
The infinitesimals of two equivalent 1-parameter formal deformations of $(A,\mu,T)$ are in the same cohomology class in $\rmH^\bullet_{\RBA}(A)$.
\end{thm}

\begin{proof} Let $\psi_t:(A[[t]],\mu_t',T_t')\rightarrow (A[[t]],\mu_t,T_t)$ be a formal isomorphism.
	Expanding the identities and collecting coefficients of $t$, we get from Equations~(\ref{Eq: equivalent deformations}) and (\ref{Eq: equivalent deformations2}):
	\begin{eqnarray*}
		\mu_1'&=&\mu_1+\mu\circ(\id\ot \psi_1)-\psi_1\circ\mu+\mu\circ(\psi_1\ot \id),\\
		T_1'&=&T_1+T\circ\psi_1-\psi_1\circ T,
		\end{eqnarray*}
	that is, we have\[(\mu_1',T_1')-(\mu_1,T_1)=(\delta^1(\psi_1), -\Phi^1(\psi_1))=d^1(\psi_1,0)\in  \C^\bullet_{\RBA}(A).\]
\end{proof}

\smallskip

\begin{defn}
	A Rota-Baxter algebra $(A,\mu,T)$ is said to be rigid if every 1-parameter formal deformation is trivial.
\end{defn}

\begin{thm}
	Let $(A,\mu,T)$ be a Rota-Baxter algebra of weight $\lambda$. If $\rmH^2_{\RBA}(A)=0$, then $(A,\mu,T)$ is rigid.
\end{thm}

\begin{proof}Let $(A[[t]], \mu_t, T_t)$ be a $1$-parameter formal deformation of $(A, \mu, T)$. By Proposition~\ref{Prop: Infinitesimal is 2-cocyle},
$(\mu_1, T_1)$ is a $2$-cocycle. By $\rmH^2_{\RBA}(A)=0$, there exists a $1$-cochain $$(\psi_1', x) \in \C^1_\RBA(A)= C^1_{\Alg}(A)\oplus \Hom(k, A)$$ such that
$(\mu_1, T_1) =  d^1(\psi_1', x), $
that is, $\mu_1=\delta^1(\psi_1')$ and $T_1=-\partial^0(x)-\Phi^1(\psi_1')$. Let $\psi_1=\psi_1'+\delta^0(x)$. Then
 $\mu_1= \delta^1(\psi_1)$ and $T_1=-\Phi^1(\psi_1)$, as it can be readily seen that $\Phi^1(\delta^0(x))=\partial^0(x)$.

Setting $\psi_t = \Id_A -\psi_1t$, we have a deformation $(A[[t]], \overline{\mu}_t, \overline{T}_t)$, where
$$\overline{\mu}_t=\psi_t^{-1}\circ \mu_t\circ (\psi_t\times \psi_t)$$
and $$\overline{T}_t=\psi_t^{-1}\circ T_t\circ \psi_t.$$
  It can be easily verify  that $\overline{\mu}_1=0, \overline{T}_1=0$. Then
    $$\begin{array}{rcl} \overline{\mu}_t&=& \mu+\overline{\mu}_2t^2+\cdots,\\
 T_t&=& T+\overline{T}_2t^2+\cdots.\end{array}$$
   By Equations~(\ref{Eq: general formal of deform product in RBA}) and (\ref{Eq: general formal of deform RBO in RBA}), we see that $(\overline{\mu}_2,  \overline{T}_2)$ is still a $2$-cocyle, so by induction, we can show that
  $ (A[[t]], \mu_t , T_t) $ is equivalent to the trivial extension $(A[[t]], \mu, T).$
Thus, $(A,\mu,T)$ is rigid.

\end{proof}

\subsection{Formal deformations of Rota-Baxter operator  with product fixed}\

Let $(A,\mu=\cdot, T)$ be a Rota-Baxter algebra of weight $\lambda$. Let us consider the case where  we only deform the Rota-Baxter operator with the product fixed. So    $A[[t]]=\{\sum_{i=0}^\infty a_it^i\ | \ a_i\in A, \forall i\geqslant 0\}$ is endowed with the product induced from that of $A$, say,
$$(\sum_{i=0}^\infty a_it^i)(\sum_{j=0}^\infty b_jt^j)=\sum_{n=0}^\infty (\sum_{i+j=n\atop i,j\geqslant 0} a_ib_j)t^n.$$
Then $A[[t]]$ becomes a flat $\bfk[[t]]$-algebra, whose product is still denoted by $\mu$.

 In this case, a  1-parameter formal deformation $(\mu_t,T_t)$ of    Rota-Baxter algebra $(A, \mu,T)$  satisfies
 $\mu_i=0, \forall i\geqslant 1$. So Equation~(\ref{Eq: deform eq for  products in RBA}) degenerates and Equation~(\ref{Eq: Deform RB operator in RBA}) becomes
 \begin{eqnarray*}
 \mu\circ (T_t \ot T_t )&=& T_t\circ \Big(\mu\circ (\id \ot T_t )+\mu \circ (T_t \ot \id)+\lambda \mu \Big).
 \end{eqnarray*}
Expanding  these equations and comparing the coefficient of $t^n$, we obtain  that  $\{T_i\}_{i\geqslant0}$ have to  satisfy: for any $n\geqslant 0$,
\begin{eqnarray}\label{Eq: deform eq for RBO}	
\sum_{i+j=n\atop i, j\geqslant 0} 	\mu \circ(T_i\ot T_{j})&=&\sum_{i+j=n\atop i, j\geqslant 0} T_{ i}\circ \mu\circ (\id\ot T_{j}) +\sum_{i+j=n\atop i, j\geqslant 0}T_{i}\circ\mu\circ (T_{j}\ot \id)+\lambda T_n\circ\mu.
\end{eqnarray}

Obviously, when $n=0$,  Equation~(\ref{Eq: deform eq for RBO}) becomes
  exactly Equation~(\ref{Eq: Rota-Baxter relation}) defining Rota-Baxter operator $T=T_0$.

When $n=1$,    Equation~(\ref{Eq: deform eq for RBO}) has the form
 $$  	\mu \circ(T\ot T_1+T_1\otimes T)= T \circ \mu\circ (\id\ot T_1) +T_1 \circ \mu\circ (\id\ot T)+ T \circ\mu\circ (T_{1}\ot \id)+  T_1 \circ\mu\circ (T\ot \id)+  \lambda T_1\circ\mu$$
 which     says exactly that $\partial^1(T_1)=0\in \C^\bullet_{\RBO}(A)$.
This proves the following result:
\begin{prop}
	Let $ T_t $ be a  1-parameter formal deformation of
	Rota-Baxter operator  $ T $ of weight $\lambda$. Then
	$ T_1 $ is a 1-cocycle in the cochain complex
	$\C_{\RBO}^\bullet(A)$.
\end{prop}
This means that the cochain complex $\C^\bullet_\RBO(A)$ controls formal deformations of Rota-Baxter operators.


\subsection{Formal deformations of associative product  with Rota-Baxter operator    fixed}\

Let $(A,\mu, T)$ be a Rota-Baxter algebra of weight $\lambda$. Let us consider the case where  we only deform the associative product with Rota-Baxter operator  fixed. So   the induced Rota-Baxter operator
on $A[[t]]$ is given by
$\sum_{i=0}^\infty a_i t^i\mapsto \sum_{i=0}^\infty T(a_i) t^i$, still denoted by $T$.

 In this case, a  1-parameter formal deformation $(\mu_t,T_t)$ of    Rota-Baxter algebra $(A, \mu,T)$  satisfies
 $T_i=0, \forall i\geqslant 1$. So    Equation~(\ref{Eq: deform eq for  products in RBA}) remains unchanged and Equation~(\ref{Eq: Deform RB operator in RBA}) becomes
 for any $n\geqslant 0$,
 \begin{equation} \label{Eq: def constaint when deform product but RBO fixed}
 	\mu_n \circ(T \ot T )=  T \circ \mu_n\circ (\id\ot T+T\ot \id) +\lambda  T \circ\mu_{n}.
 \end{equation}
 As usual, Equation~(\ref{Eq: deform eq for  products in RBA}) for $n=1$ says that
 $\delta^2(\mu_1)=0\in \C^\bullet_{\Alg}(A)$, but
 Equation~(\ref{Eq: def constaint when deform product but RBO fixed}) implies that
 $\mu_n$ lies in $\mathrm{Ker}(\Phi^2: \C^2_{\Alg}(A)\to \C^2_{\RBO}(A))$.

This proves the following result:
\begin{prop}
	Let $ \mu_t $ be a  1-parameter formal deformation of
	associative product   $ \mu  $ with Rota-Baxter operator $T$ fixed. Then
	$ \mu_1 $ is a 2-cocycle in the cochain complex
	$\mathrm{Ker}(\Phi^\bullet : \C^\bullet_{\Alg}(A)\to \C^\bullet_{\RBO}(A))$.
\end{prop}
This means that the cochain complex $\mathrm{Ker}(\Phi^\bullet : \C^\bullet_{\Alg}(A)\to \C^\bullet_{\RBO}(A))$  controls formal deformations of associative product with Rota-Baxter operator   fixed.

\section{Abelian extensions of Rota-Baxter algebras}

 In this section, we study abelian extensions of Rota-Baxter algebras and show that they are classified by the second cohomology, as one would expect of a good cohomology theory.

 Notice that a vector space $M$ together with a linear transformation $T_M:M\to M$ is naturally a Rota-Baxter algebra where the multiplication on $M$ is defined to be $uv=0$ for all $u,v\in M.$

 \begin{defn}
 	An   abelian extension  of Rota-Baxter algebras is a short exact sequence of  morphisms of Rota-Baxter algebras
 \begin{eqnarray}\label{Eq: abelian extension} 0\to (M, T_M)\stackrel{i}{\to} (\hat{A}, \hat{T})\stackrel{p}{\to} (A, T)\to 0,
 \end{eqnarray}
 that is, there exists a commutative diagram:
 	\[\begin{CD}
 		0@>>> {M} @>i >> \hat{A} @>p >> A @>>>0\\
 		@. @V {T_M} VV @V {\hat{T}} VV @V T VV @.\\
 		0@>>> {M} @>i >> \hat{A} @>p >> A @>>>0,
 	\end{CD}\]
 where the Rota-Baxter algebra $(M, T_M)$	satisfies  $uv=0$ for all $u,v\in M.$

 We will call $(\hat{A},\hat{T})$ an abelian extension of $(A,T)$ by $(M,T_M)$.
 \end{defn}

 \begin{defn}
 	Let $(\hat{A}_1,\hat{T}_1)$ and $(\hat{A}_2,\hat{T}_2)$ be two abelian extensions of $(A,T)$ by $(M,T_M)$. They are said to be  isomorphic  if there exists an isomorphism of Rota-Baxter algebras $\zeta:(\hat{A}_1,\hat{T}_1)\rar (\hat{A}_2,\hat{T}_2)$ such that the following commutative diagram holds:
 	\begin{eqnarray}\label{Eq: isom of abelian extension}\begin{CD}
 		0@>>> {(M,T_M)} @>i >> (\hat{A}_1,{\hat{T}_1}) @>p >> (A,T) @>>>0\\
 		@. @| @V \zeta VV @| @.\\
 		0@>>> {(M,T_M)} @>i >> (\hat{A}_2,{\hat{T}_2}) @>p >> (A,T) @>>>0.
 	\end{CD}\end{eqnarray}
 \end{defn}

 A   section of an abelian extension $(\hat{A},{\hat{T}})$ of $(A,T)$ by $(M,T_M)$ is a linear map $s:A\rar \hat{A}$ such that $p\circ s=\Id_A$.

 We will show that isomorphism classes of  abelian extensions of $(A,T)$ by $(M,T_M)$ are in bijection with the second cohomology group   ${\rmH}_{\RBA}^2(A,M)$.

 \bigskip

Let    $(\hat{A},\hat{T})$ be  an abelian extension of $(A,T)$ by $(M,T_M)$ having the form Equation~\eqref{Eq: abelian extension}. Choose a section $s:A\rar \hat{A}$. We   define
 $$
 am:=s(a)m,\quad ma:=ms(a), \quad \forall a\in A, m\in M.
 $$
 \begin{prop}\label{Prop: new RB bimodules from abelian extensions}
 	With the above notations, $(M, T_M)$ is a Rota-Baxter  bimodule over $(A,T)$.
 \end{prop}
 \begin{proof}
 	For any $a,b\in A,\,m\in M$, since $s(ab)-s(a)s(b)\in M$ implies $s(ab)m=s(a)s(b)m$, we have
 	\[ (ab)  m=s(ab)m=s(a)s(b)m=a(bm).\]
 	Hence,  this gives a left $A$-module structure and the case of right module structure is similar.

 Moreover, ${\hat{T}}(s(a))-s(T(a))\in M$ means that  ${\hat{T}}(s(a))m=s(T(a))m$. Thus we have
 	\begin{align*}
 		T(a)T_M(m)&=s(T(a))T_M(m)\\
 		&=\hat{T}(s(a))T_M(m)\\
 		&=\hat{T}(\hat{T}(s(a))m+s(a)T_M(m)+\lambda s(a)m)\\
 		&=T_M(T(a)m+aT_M(m)+\lambda am)
 	\end{align*}
 	It is similar to see $T_M(m)T(a)=T_M(T_M(m)a+mT(a)+\lambda ma)$.

 	Hence, $(M, T_M)$ is a  Rota-Baxter  bimodule over $(A,T)$.
 \end{proof}

 We  further  define linear maps $\psi:A\ot A\rar M$ and $\chi:A\rar M$ respectively by
 \begin{align*}
 	\psi(a\ot b)&=s(a)s(b)-s(ab),\quad\forall a,b\in A,\\
 	\chi(a)&={\hat{T}}(s(a))-s(T(a)),\quad\forall a\in A.
 \end{align*}

 \begin{prop}\label{prop:2-cocycle}
 	 The pair
 	$(\psi,\chi)$ is a 2-cocycle  of   Rota-Baxter algebra $(A,T)$ with  coefficients  in the Rota-Baxter bimodule $(M,T_M)$ introduced in Proposition~\ref{Prop: new RB bimodules from abelian extensions}.
 \end{prop}

 	The proof is by direct computations, so it is left to the reader.

 The choice of the section $s$ in fact determines a splitting
 $$\xymatrix{0\ar[r]&  M\ar@<1ex>[r]^{i} &\hat{A}\ar@<1ex>[r]^{p} \ar@<1ex>[l]^{t}& A \ar@<1ex>[l]^{s} \ar[r] & 0}$$
 subject to $t\circ i=\Id_M, t\circ s=0$ and $ it+sp=\Id_{\hat{A}}$.
 Then there is an induced isomorphism of vector spaces
 $$\left(\begin{array}{cc} p& t\end{array}\right): \hat{A}\cong   A\oplus M: \left(\begin{array}{c} s\\ i\end{array}\right).$$
We can  transfer the Rota-Baxter algebra structure on $\hat{A}$ to $A\oplus M$ via this isomorphism.
  It is direct to verify that this  endows $A\oplus M$ with a multiplication $\cdot_\psi$ and an Rota-Baxter operator $T_\chi$ defined by
 \begin{align}
 	\label{eq:mul}(a,m)\cdot_\psi(b,n)&=(ab,an+mb+\psi(a,b)),\,\forall a,b\in A,\,m,n\in M,\\
 	\label{eq:dif}T_\chi(a,m)&=(T(a),\chi(a)+T_M(m)),\,\forall a\in A,\,m\in M.
 \end{align}
 Moreover, we get an abelian extension
 $$0\to (M, T_M)\stackrel{\left(\begin{array}{cc} s& i\end{array}\right) }{\to} (A\oplus M, T_\chi)\stackrel{\left(\begin{array}{c} p\\ t\end{array}\right)}{\to} (A, T)\to 0$$
 which is easily seen to be  isomorphic to the original one \eqref{Eq: abelian extension}.

 \medskip

 Now we investigate the influence of different choices of   sections.

 \begin{prop}\label{prop: different sections give}
 \begin{itemize}
 \item[(i)] Different choices of the section $s$ give the same  Rota-Baxter bimodule structures on $(M, T_M)$;

 \item[(ii)]   the cohomological class of $(\psi,\chi)$ does not depend on the choice of sections.

 \end{itemize}

 \end{prop}
 \begin{proof}Let $s_1$ and $s_2$ be two distinct sections of $p$.
  We define $\gamma:A\rar M$ by $\gamma(a)=s_1(a)-s_2(a)$.

  Since the Rota-Baxter algebra $(M, T_M)$	satisfies  $uv=0$ for all $u,v\in M$,
  $$s_1(a)m= s_2(a)m+\gamma(a)m=s_2(a)m.$$ So different choices of the section $s$ give the same  Rota-Baxter bimodule structures on $(M, T_M)$;

  We   show that the cohomological class of $(\psi,\chi)$ does not depend on the choice of sections.   Then
 	\begin{align*}
 		\psi_1(a,b)&=s_1(a)s_1(b)-s_1(ab)\\
 		&=(s_2(a)+\gamma(a))(s_2(b)+\gamma(b))-(s_2(ab)+\gamma(ab))\\
 		&=(s_2(a)s_2(b)-s_2(ab))+s_2(a)\gamma(b)+\gamma(a)s_2(b)-\gamma(ab)\\
 		&=(s_2(a)s_2(b)-s_2(ab))+a\gamma(b)+\gamma(a)b-\gamma(ab)\\
 		&=\psi_2(a,b)+\delta(\gamma)(a,b)
 	\end{align*}
 	and
 	\begin{align*}
 		\chi_1(a)&={\hat{T}}(s_1(a))-s_1(T(a))\\
 		&={\hat{T}}(s_2(a)+\gamma(a))-(s_2(T(a))+\gamma(T(a)))\\
 		&=({\hat{T}}(s_2(a))-s_2(T(a)))+{\hat{T}}(\gamma(a))-\gamma(T(a))\\
 		&=\chi_2(a)+T_M(\gamma(a))-\gamma(T_A(a))\\
 		&=\chi_2(a)-\Phi^1(\gamma)(a).
 	\end{align*}
 	That is, $(\psi_1,\chi_1)=(\psi_2,\chi_2)+d^1(\gamma)$. Thus $(\psi_1,\chi_1)$ and $(\psi_2,\chi_2)$ form the same cohomological class  {in $\rmH_{\RBA}^2(A,M)$}.

 \end{proof}

 We show now the isomorphic abelian extensions give rise to the same cohomology classes.
 \begin{prop}Let $M$ be a vector space and  $T_M\in\End_\bfk(M)$. Then $(M, T_M)$ is a Rota-Baxter algebra with trivial multiplication.
 Let $(A,T)$ be a  Rota-Baxter algebra.
 Two isomorphic abelian extensions of Rota-Baxter algebra $(A, T)$ by  $(M, T_M)$  give rise to the same cohomology class  in $\rmH_{\RBA}^2(A,M)$.
 \end{prop}
 \begin{proof}
  Assume that $(\hat{A}_1,{\hat{T}_1})$ and $(\hat{A}_2,{\hat{T}_2})$ are two isomorphic abelian extensions of $(A,T)$ by $(M,T_M)$ as is given in \eqref{Eq: isom of abelian extension}. Let $s_1$ be a section of $(\hat{A}_1,{\hat{T}_1})$. As $p_2\circ\zeta=p_1$, we have
 	\[p_2\circ(\zeta\circ s_1)=p_1\circ s_1=\Id_{A}.\]
 	Therefore, $\zeta\circ s_1$ is a section of $(\hat{A}_2,{\hat{T}_2})$. Denote $s_2:=\zeta\circ s_1$. Since $\zeta$ is a homomorphism of Rota-Baxter  algebras such that $\zeta|_M=\Id_M$, $\zeta(am)=\zeta(s_1(a)m)=s_2(a)m=am$, so $\zeta|_M: M\to M$ is compatible with the induced  Rota-Baxter bimodule structures.
 We have
 	\begin{align*}
 		\psi_2(a\ot b)&=s_2(a)s_2(b)-s_2(ab)=\zeta(s_1(a))\zeta(s_1(b))-\zeta(s_1(ab))\\
 		&=\zeta(s_1(a)s_1(b)-s_1(ab))=\zeta(\psi_1(a,b))\\
 		&=\psi_1(a,b)
 	\end{align*}
 	and
 	\begin{align*}
 		\chi_2(a)&={\hat{T}_2}(s_2(a))-s_2(T(a))={\hat{T}_2}(\zeta(s_1(a)))-\zeta(s_1(A(a)))\\
 		&=\zeta({\hat{T}_1}(s_1(a))-s_1(T(a)))=\zeta(\chi_1(a))\\
 		&=\chi_1(a).
 	\end{align*}
 	Consequently, two isomorphic abelian extensions give rise to the same element in {$\rmH_{\RBA}^2(A,M)$}.
\end{proof}
 \bigskip

 Now we consider the reverse direction.

 Let $(M, T_M)$ be a Rota-Baxter bimodule over Rota-Baxter algebra $(A, T)$, given two linear maps  $\psi:A\ot A\rar M$ and $\chi:A\rar M$, one can define  a multiplication $\cdot_\psi$ and an operator $T_\chi$  on  $A\oplus M$ by Equations~(\ref{eq:mul})(\ref{eq:dif}).
 The following fact is important:
 \begin{prop}\label{prop:2-cocycle}
 	The triple $(A\oplus M,\cdot_\psi,T_\chi)$ is a Rota-Baxter algebra   if and only if
 	$(\psi,\chi)$ is a 2-cocycle  of the Rota-Baxter algebra $(A,T)$ with  coefficients  in $(M,T_M)$.
 In this case,    we obtain an abelian extension
  $$0\to (M, T_M)\stackrel{\left(\begin{array}{cc} 0& \Id  \end{array}\right) }{\to} (A\oplus M, T_\chi)\stackrel{\left(\begin{array}{c} \Id\\ 0\end{array}\right)}{\to} (A, T)\to 0,$$
  and the canonical section $s=\left(\begin{array}{cc}   \Id  & 0\end{array}\right): (A, T)\to (A\oplus M, T_\chi)$ endows $M$ with the original Rota-Baxter bimodule structure.
 \end{prop}
 \begin{proof}
 	If $(A\oplus M,\cdot_\psi,T_\chi)$ is a Rota-Baxter algebra, then the associativity of $\cdot_\psi$ implies
 	\begin{equation}
 		\label{eq:mc}a\psi(b\ot c)-\psi(ab\ot c)+\psi(a\ot bc)-\psi(a\ot b)c=0,
 	\end{equation}
 	which means $\delta^2(\phi)=0$ in $\C^\bullet(A,M)$.
 	Since $T_\chi$ is an Rota-Baxter operator,
 	for any $a,b\in A, m,n\in M$, we have
 	$$T_\chi((a,m))\cdot_\psi T_\chi((b,n))=T_\chi\Big(T_\chi(a,m)\cdot_\psi(b,n)+(a,m)\cdot_\psi T_\chi(b,n)+\lambda(a,m)\cdot_\psi(b,n)\Big)$$
 	Then $\chi,\psi$ satisfy the following equations:
 	\begin{align*}
 		&T(a)\chi(b)+\chi(a)T(b)+\psi(T(a)\ot T(b))\\
 		=&T_M(\chi(a)b)+T_M(\psi(T(a)\ot b))+\chi(T(a)b)\\
 		&+T_M(a\chi(b))+T_M(\psi(a\ot T(b)))+\chi(aT(b))\\
 		&+\lambda T_M(\psi(a\ot b))+\lambda \chi(ab)
 	\end{align*}
 	
 	That is,
 	\[ \partial^1(\chi)+\Phi^2(\psi)=0.\]
 	Hence, $(\psi,\chi)$ is a  2-cocycle.
 	
 	 Conversely, if $(\psi,\chi)$ is a 2-cocycle, one can easily check that $(A\oplus M,\cdot_\psi,T_\chi)$ is a  Rota-Baxter algebra.

  The last statement is clear.
 \end{proof}

 Finally, we show the following result:
 \begin{prop}
 	Two cohomologous $2$-cocyles give rise to isomorphic abelian extensions.
 \end{prop}
 \begin{proof}

 	  Given two 2-cocycles $(\psi_1,\chi_1)$ and $(\psi_2,\chi_2)$, we can construct two abelian extensions $(A\oplus M,\cdot_{\psi_1},T_{\chi_1})$ and  $(A\oplus M,\cdot_{\psi_2},T_{\chi_2})$ via Equations~\eqref{eq:mul} and \eqref{eq:dif}. If they represent the same cohomology  class {in $\rmH_{\RBA}^2(A,M)$}, then there exists two linear maps $\gamma_0:k\rightarrow M, \gamma_1:A\to M$ such that $$(\psi_1,\chi_1)=(\psi_2,\chi_2)+(\delta^1(\gamma_1),-\Phi^1(\gamma_1)-\partial^0(\gamma_0)).$$
 	Notice that $\partial^0=\Phi^1\circ\delta^0$. Define $\gamma: A\rightarrow M$ to be $\gamma_1+\delta^0(\gamma_0)$. Then $\gamma$ satisfies
 	\[(\psi_1,\chi_1)=(\psi_2,\chi_2)+(\delta^1(\gamma),-\Phi^1(\gamma)).\]
 	
 	Define $\zeta:A\oplus M\rar A\oplus M$ by
 	\[\zeta(a,m):=(a, -\gamma(a)+m).\]
 	Then $\zeta$ is an isomorphism of these two abelian extensions $(A\oplus M,\cdot_{\psi_1},T_{\chi_1})$ and  $(A\oplus M,\cdot_{\psi_2},T_{\chi_2})$.
 \end{proof}

\section{$L_\infty$-algebra structure on the  cochain complex of a Rota-Baxter algebra }

In this section, we will  consider  $L_\infty$-algebra structures controling deformations of Rota-Baxter algebras.
     Rota-Baxter
algebra structures on  a vector space will be realised  as
Maurer-Cartan elements in an explicitly constructed    $L_\infty$-algebra and it will be seen  that the shift of  the cochain
complex of a Rota-Baxter algebra is exactly the underlying complex
of the twisted  $L_\infty$-algebra   by the corresponding Maurer-Cartan
element corresponding to the Rota-Baxter algebra structure.

\subsection{ $L_\infty$-algebra structure on $\mathfrak{C}_{\RBA}(V)$}\

   Let $V$ be a graded vector space.  We define
a graded space $\mathfrak{C}_{\RBA}(V)$ as :
$$\mathfrak{C}_{\RBA}(V)=\mathfrak{C}_\Alg(V)\oplus \mathfrak{C}_{\RBO}(V),$$
where $$\mathfrak{C}_{\Alg}(V)=\Hom(T^c(sV),sV) \quad \mathrm{and}  \quad  \mathfrak{C}_{\RBO}(V)=\Hom(T^c(sV),V).$$

Notice that for a Rota-Baxter algebra $A$,   $\mathfrak{C}_{\RBA}(A)$ is just the underlying space of   the cochain complex of  Rota-Baxter algebra $A$ up to shift.

Now, we are going to build an $L_\infty$-algebra structure on $\mathfrak{C}_{\RBA}(V)$.
 The   operators $\{l_n\}_{n\geqslant 1}$ on $\mathfrak{C}_{\RBA}(V)$ are defined as follows:
\begin{itemize}
	\item[(I)] For $sh\in sV=\Hom(\bfk,sV)\subset \mathfrak{C}_{\Alg}(V)$ with $h\in V$, define $$l_1(sh)=h \in V=\Hom(\bfk,V) \subset\mathfrak{C}_{\RBO}(V).$$

	\item[(II)] For homogeneous elements $sf, sh\in \mathfrak{C}_{\Alg}(V)$, define $$l_2(sf\ot sh):= [sf, sh]_G\in\mathfrak{C}_{\Alg}(V),$$  where $[-, -]_G$ is the Gerstenhaber bracket defined in Equation~(\ref{Eq: Gerstahaber bracket}).

	\item[(III)]
	\begin{itemize}	
		\item[(i)] Let $n\geqslant 1$.  For homogeneous elements $sh\in \Hom((sV)^{\ot n},sV)\subset \mathfrak{C}_{\Alg}(V)$ and $g_1,\dots, g_n\in \mathfrak{C}_{\RBO}(V)$,	define $$l_{n+1}(sh\ot g_1\ot \cdots \ot g_n)\in \mathfrak{C}_{\RBO}(V)$$  as :
\begin{align*}&l_{n+1}(sh\ot g_1\ot \cdots \ot g_n)=\\
			&  \sum_{\sigma\in S_n}(-1)^{\eta}\Big(h\circ(sg_{\sigma(1)}\ot \cdots \ot sg_{\sigma(n)})-(-1)^{(|g_{\sigma(1)}|+1)(|h|+1)}s^{-1}(sg_{\sigma(1)})\big\{sh\big\{sg_{\sigma(2)},\dots,sg_{\sigma(n)}\big\}\big\}\Big),
		\end{align*}
		where $(-1)^{\eta}=\chi(\sigma; g_1,\dots,g_n)(-1)^{n(|h|+1)+\sum\limits_{k=1}^{n-1}\sum\limits_{j=1}^k|g_{\sigma(j)}|}$.

		\item[(ii)]  Let $n\geqslant 2$.  For homogeneous elements $sh\in \Hom((sV)^{\ot n},sV)\subset \mathfrak{C}_{\Alg}(V)$ and $g_1,\dots ,g_m\in \mathfrak{C}_{\RBO}(V)$ with $1\leqslant m\leqslant n-1$, define
$$l_{m+1}(sh\ot g_1\ot \cdots\ot g_m)\in \mathfrak{C}_{\RBO}(V)$$ to be:
		\[l_{m+1}(sh\ot g_1\ot \cdots\ot g_m)=\sum_{\sigma\in S_m}(-1)^\xi\lambda^{n-m}  s^{-1} (sg_{\sigma(1)})\big\{sh\big\{sg_{\sigma(2)},\dots,sg_{\sigma(m)}\big\}\big\},\]
		where $(-1)^\xi=\chi(\sigma; g_1,\dots,g_m)(-1)^{1+m(|h|+1)+\sum\limits_{k=1}^{m-1}\sum\limits_{j=1}^k|g_{\sigma(j)}|+(|h|+1)(|g_{\sigma(1)}|+1)}$.
	\end{itemize}
	
	\smallskip
	
	\item[(IV)]  Let $m\geqslant 1$.  For homogeneous elements $sh\in \Hom((sV)^{\ot n},sV)\subset \mathfrak{C}_{\Alg}(V), g_1,\dots,g_m\in \Hom(T^c(sV),V)\subset \mathfrak{C}_{\RBO}(V)$ with $1\leqslant m\leqslant n$, for $1\leqslant k\leqslant m$,  define $$l_{m+1}(g_1\ot \cdots\ot g_k\ot sh \ot g_{k+1}\ot \cdots\ot g_m)\in \mathfrak{C}_{\RBO}(V)$$ to be
	$$l_{m+1}(g_1\ot \cdots\ot g_k\ot sh \ot g_{k+1}\ot \cdots\ot g_m)=(-1)^{(|h|+1)(\sum\limits_{j=1}^k|g_j|)+k}l_{m+1}(sh\ot g_1\ot \cdots \ot g_m),$$
	where the RHS has been introduced in (III) (i) and (ii).

	\item[(V)] All other  components of operators $\{l_n\}_{n\geqslant 1}$ vanish.
\end{itemize}

\medskip

\begin{thm}\label{Thm: rb-L-infty}
	Given a graded space $V$  and a scalar $\lambda\in \bfk$, the graded space ${\mathfrak{C}_{\RBA}}(V)$ endowed with operations $\{l_n\}_{n\geqslant 1}$ defined above forms an $L_\infty$-algebra.
\end{thm}

This theorem is one of the main results in this paper, whose  proof requires quite a lot of technical details, so we postpone it to Appendix A.

\subsection{Realising Rota-Baxter algebra structures as Maurer-Cartan elements}\


\begin{thm}\label{Thm: RB is MC element}  Let $V$ be a ungraded space considered  as a graded space concentrated in degree 0.  Then a Rota-Baxter algebra structure of weight $\lambda$ on $V$ is equivalent to a Maurer-Cartan element in the $L_\infty$-algebra ${\mathfrak{C}_{\RBA}}(V)_\lambda$ introduced above.
\end{thm}
\begin{proof}
	Since $V$ is concentrated in degree 0, the degree $-1$ part of ${\mathfrak{C}_{\RBA}}(V) $ is $\Hom((sV)^{\ot 2},sV)\oplus \Hom(sV,V)$. Let $\alpha=(m,\tau) \in \mathfrak{C}_{\RBA}(V)_{-1}$.
	Then
	\begin{eqnarray*}
		 l_2(\alpha\ot \alpha)&=&(l_2(m\ot m), l_2(m\ot \tau )+l_2(\tau\ot m))\\
		&=&(l_2(m\ot m), 2l_2(m\ot \tau))\\
&=&( [m, m]_G, 2\lambda \tau\circ m),\\
		 l_3(\alpha^{\ot 3})&=&(0,l_3(m\ot \tau\ot \tau )+l_3(\tau\ot m\ot \tau)+l_3(\tau\ot \tau\ot m))\\
&=&(0,3l_3(m\ot \tau\ot \tau))\\
&=&\Big(0, -6 \Big((s^{-1}m)\circ (s\tau\ot s\tau)-s^{-1} (s\tau) \{ m\big\{ s\tau\}\} \Big)\Big),\\
&=& \Big(0, -6  (s^{-1}m)\circ (s\tau\ot s\tau)+6  \tau\circ  m\circ (  s\tau \ot \Id+\Id\ot s\tau)\Big)
	\end{eqnarray*}
	and $l_n(\alpha^{\ot n})=0$ for $n\ne 2,3$. 
	By expanding the Maurer-Cartan Equation~(\ref{Eq: mc-equation})
	\[\sum_{i=1}^\infty\frac{1}{n!}(-1)^{\frac{n(n-1)}{2}}l_n(\alpha^{\ot n})=0,\]
	we get:
	\begin{eqnarray}
		\label{Eq: ass-mc}	&& [ m,  m]_G=0,\\
		\label{Eq: rb-mc}&&	-\lambda \tau\circ m+(s^{-1}  m)\circ(s\tau\ot s\tau)- \tau   \circ   m \circ (s\tau\ot \id+ \id\ot s\tau )   =0.
	\end{eqnarray}
	Set $\mu =\tilde{m}=s^{-1}\circ m\circ s^{\ot 2}: V^{\ot 2}\rightarrow V$  and $ T=\hat{\tau}=\tau\circ s:V\rightarrow  V $ via the  fixed isomorphisms (\ref{Eq: first can isom}) and (\ref{Eq: second can isom}). Equation~(\ref{Eq: ass-mc}) is equivalent to saying that  $\mu$ is associative, see also Proposition~\ref{Prop: bijection between associative product and MC elements};   Equation~(\ref{Eq: rb-mc}) is equivalent to
\begin{eqnarray*}
		  \lambda T\circ \mu- \mu  \circ (T\ot T) +T  \circ   \mu\circ (T\ot \id+ \id\ot T )   =0,
	\end{eqnarray*}
which says exactly  that $T$ is a Rota-Baxter operator   of weight $\lambda$ on associative algebra $(V,\mu)$.
	
	Conversely, Given a Rota-Baxter algebra structure $(\mu,T)$ of weight $\lambda$  on vector space $V$,   define $$m=-s\circ\mu\circ (s^{-1})^{\ot 2}: (sV)^{\ot 2}\rightarrow sV  \ \mathrm{and}\ \tau=T\circ s^{-1}: sV\rightarrow V.$$ Then $(m,\tau)$ is a Maurer-Cartan element in ${\mathfrak{C}_{\RBA}}(V)$.
\end{proof}

\begin{prop}\label{Prop: cohomlogy complex as underlying complex of L infinity algebra}
	Let $(A,\mu,T)$ be a Rota-Baxter algebra of weight $\lambda$. Twist the $L_\infty$-algebra $\mathfrak{C}_{\RBA}(A)$ by the Maurer-Cartan element corresponding to the Rota-Baxter algebra structure $(A,\mu,T)$, then its  underlying complex is exactly $s\C^\bullet_{\RBA}(A)$, the shift of the cochain complex of Rota-Baxter algebra $(A, \mu, T)$  defined in Section \ref{Subsec:chomology RB}.
\end{prop}

\begin{proof}
	By Theorem~\ref{Thm: RB is MC element}, the Rota-Baxter algebra structure $(A,\mu,T)$ is equivalent to a Maurer-Cartan element $\alpha=(m,\tau)$ in the $L_\infty$-algebra $\mathfrak{C}_{\RBA}(A)$ with
$$m=-s\circ\mu\circ (s^{-1})^{\ot 2}: (sV)^{\ot 2}\rightarrow sV  \ \mathrm{and}\ \tau=T\circ s^{-1}: sV\rightarrow V.$$

 By Proposition~\ref{Prop: deformed-L-infty}, the Maurer-Cartan element induces a new $L_\infty$-algebra structure $\{l_n^\alpha\}_{n\geqslant 1}$ on the graded space $\mathfrak{C}_{\RBA}(A)$.
	By definition,
	for any $sf\in\Hom((sA)^{\ot n}, sA)\subset \mathfrak{C}_{\Alg}(A)$,
	\begin{align*}l_1^\alpha(sf)&=\sum_{i=0}^\infty\frac{1}{i!}(-1)^{i+\frac{i(i-1)}{2}}l_{i+1}(\alpha^{\ot i}\ot sf)\\
		&=\Big(-l_2(m\ot sf),\  \sum_{i=1}^n\frac{1}{i!}(-1)^{ \frac{i(i+1)}{2}}l_{i+1}(\tau^{\ot i}\ot sf)\Big).	
	\end{align*}
	By definition of $\{l_n\}_{n\geqslant 1}$ on $\mathfrak{C}_{\RBA}(A)$, $-l_2(m\ot sf)=-[m,sf]_G$, which corresponds to  $-\delta^n(\tilde{f})$ in  $s\C^\bullet_\Alg(A)$ under the fixed  isomorphism  (\ref{Eq: first can isom});  for details, see Proposition~\ref{Prop: Twisted cx is orginal cx}.

	On the other hand, we have \begin{align*}
		& \sum_{i=0}^\infty\frac{1}{i!}(-1)^{ \frac{i(i+1)}{2}}l_{i+1}(\tau^{\ot i}\ot sf)\\
		=&\frac{1}{n!}(-1)^{\frac{n(n+1)}{2}}l_{n+1}(\tau^{\ot n}\ot sf)+\sum_{i=1}^{n-1}\frac{1}{i!}(-1)^{\frac{i(i+1)}{2}}l_{i+1}(\tau^{\ot i}\ot sf)\\
		\stackrel{(IV)}{=}&\frac{1}{n!}(-1)^{\frac{n(n+1)}{2}}(-1)^{n(|f|+1)+n}l_{n+1}(sf\ot \tau^{\ot n})+\sum_{i=1}^{n-1}\frac{1}{i!}(-1)^{\frac{i(i+1)}{2}}(-1)^{i(|f|+1)+i}l_{i+1}(sf \ot \tau^{\ot i})\\
		\stackrel{(III)}{=} &\frac{1}{n!}(-1)^{\frac{n(n+1)}{2}+n|f|}(-1)^{n(|f|+1)+\frac{n(n-1)}{2}}n!\Big(f\circ(s\tau)^{\ot n}-\tau \circ \big(sf \big\{\underbrace{s\tau,\dots,s\tau}_{n-1}\big\}\big)\Big)\\
		&\ \ \ \ \ +\sum_{i=1}^{n-1}\frac{1}{i!}(-1)^{\frac{i(i+1)}{2}+i|f|+i(|f|+1)+\frac{i(i-1)}{2}+1}i!\lambda^{n-i}\tau\circ(sf\big\{\underbrace{s\tau,\dots,s\tau}_{i-1}\big\})\\
		= &f\circ (s\tau)^{\ot n} -\sum_{k=1}^n\lambda^{n-k}\tau\circ(sf\big\{ \underbrace{s\tau,\dots,s\tau}_{k-1})\big\},
	\end{align*}
which can be seen to be correspondent  to  $\Phi^n(\tilde{f})$ via the  fixed isomorphism  (\ref{Eq: second can isom}).
	
	For any $g\in \Hom((sA)^{\ot (n-1)},A)\subset \mathfrak{C}_{\RBO}(A)$, we have
	\begin{align*}
		l_1^{\alpha}(g)=&\sum_{k=0}^\infty\frac{1}{i!}(-1)^{\frac{i(i+1)}{2}}l_{i+1}(\alpha^{\ot i}\ot  g)\\
		=&-l_2(m\ot g)-\frac{1}{2!}\Big(l_3(m\ot \tau\ot g)+l_3(\tau\ot m\ot g)\Big)\\
		=&(-1)^{n}\lambda g \circ m+ s^{-1} m\circ(s\tau\ot sg)-\tau\circ\big(m\big\{sg\big\}\big) + s^{-1} m\circ(sg\ot s\tau)-(-1)^{n}g\big\{m\big\{s\tau\big\}\big\}. \end{align*}
	which  corresponds to  $ \partial^{n-1}(\hat{g})$   via the  fixed isomorphism    (\ref{Eq: second can isom}).
	
	We have shown that the underlying complex  of twisted  $L_\infty$-algebra $\mathfrak{C}_{\RBA}(A)$ by Maurer-Cartan element $\alpha$ is exactly the complex $s\C^\bullet_{\RBA}(A)$, the shift of the complex $\C^\bullet_{\RBA}(A)$ defined in Section~\ref{Subsec:chomology RB}.
\end{proof}

Although $\mathfrak{C}_{\RBA}(A)$ is an $L_\infty$-algebra, the next result shows that once the associative algebra structure $\mu$ over $A$ is fixed, the graded space  $\mathfrak{C}_{\RBO}(A)$, which , after twisting procedure, controls deformations of Rota-Baxter operators, is a genuine differential graded Lie algebra.

\begin{prop}
	Let $(A,\mu)$ be an associative algebra. Then the graded space  $\mathfrak{C}_{\RBO}(A)$ can be endowed with a   dg Lie algebra structure, and a Rota-Baxter operator $T$ of weight $\lambda$ on $(A,\mu)$ is equivalent to a Maurer-Cartan element in this dg Lie algebra.  Given a Rota-Baxter operator $T$ on associative algebra $(A, \mu)$, the underlying complex of the twisted dg Lie algebra $\mathfrak{C}_{\RBO}(A)$ by the corresponding Maurer-Cartan element   is exactly   the cochain complex of Rota-Baxter operator $C_{\RBO}^\bullet(A)$.
\end{prop}
\begin{proof}Consider $A$ as graded space concentrated in degree 0. Define $m=- s\circ \mu\circ (s^{-1}\ot s^{-1}): (sA)^{\ot 2}\rightarrow sA$. Then by Equations~(\ref{Eq: ass-mc})(\ref{Eq: rb-mc}),  $\alpha=(m,0)$ is naturally a Maurer-Cartan element in $L_\infty$-algebra $\mathfrak{C}_{\RBA}(A)$.
	By the construction of   $l_n$ on $\mathfrak{C}_{\RBA}(A)$,    the graded subspace $\mathfrak{C}_{\RBO}(A)$ is closed under the action of operators $\{l_n^\alpha\}_{n\geqslant 1}$. Since the arity of $m$ is 2,  the restriction of $l_n^\alpha$ on $\mathfrak{C}_{\RBO}(A)$ is $0$ for $n\geqslant 3$. Thus $(\mathfrak{C}_{\RBO},\{l_n^\alpha\}_{n=1,2})$ forms a dg Lie algebra.
	More explicitly, for $f\in \Hom((sA)^{\ot n},A),g\in \Hom((sA)^{\ot k},A)$,
	\begin{align*}
		l_1^\alpha(f)=&-l_2(m\ot f)=-(-1)^{|f|+1}\lambda f\big\{ m\big\} =(-1)^n\lambda f\big\{ m\big\}\\
		l_2^\alpha(f\ot g)=&l_3(m\ot f\ot g)\\
		=&(-1)^{|f|}\Big(s^{-1}m\circ (sf\ot sg)-(-1)^{|f|+1}f \big\{m \big\{sg\big\} \big\}\Big)\\
		&+(-1)^{|f||g|+1+|g|}\Big(s^{-1}m\circ(sg\ot sf)-(-1)^{|g|+1}g\big\{m\big\{ sf\big\} \big\}\Big)\\
		=&(-1)^n s^{-1}m\circ (sf\ot sg)+ f \big\{m\big\{ sg\big\} \big\} \\
		&+(-1)^{nk+1+k} s^{-1}m\circ(sg\ot sf)-(-1)^{nk}g\big\{m\big\{ sf\big\}\big\} .
	\end{align*}
	
	Since $A$ is concentrated in degree 0, we have $\mathfrak{C}_{\RBO}(A)_{-1}=\Hom(sA,A)$. Take an element  $\tau\in \Hom(sA,A)_{-1}$  Then $\tau$
	satisfies the  Maurer-Cartan equation:
	$$l_1^\alpha(\tau)-\frac{1}{2}l_2^{\alpha}(\tau\ot \tau)=0,$$
	if and only if
	$$-\lambda \tau\circ m+s^{-1}m\circ (s\tau\ot s\tau)-\tau\circ(m\big\{s\tau\big\})=0.$$
	Define $T=\tau\circ s:A\rightarrow A$. The above equation is exactly the statement  that $T$ is a Rota-Baxter operator of weight $\lambda$ on associative algebra $(A,\mu)$.

Now let $T$ be a Rota-Baxter oeprator on    associative algebra $(A, \mu)$. By the first statement, it corresponds to a Maurer-Cartan element $\beta$ in the dg Lie algebra $(\mathfrak{C}_{\RBA}(A), l_1^\alpha, l_2^\alpha)$. More precisely, $\beta\in \mathfrak{C}_{\RBO}(A)_{-1}=\Hom(sA,A)$ is defined to be
$\beta=T\circ s^{-1}$.
 For $f\in   \Hom((sA)^{\ot n}, A)$, we compute $(l_1^\alpha)^\beta(f)$.  In fact,
 $$\begin{array}{rcl} (l_1^\alpha)^\beta(f)&=& l_1^\alpha(f)-l_2^\alpha(\beta\ot f)\\
 &=& (-1)^n \lambda f\{m\}+s^{-1} m \circ (sf\ot s\beta)-\beta\{m\{sf\}\}\\
 &&+s^{-1}m \circ (sf\ot s\beta)+(-1)^n f\{m\{s\beta\}\}, \end{array}$$
which  corresponds to  $ \partial^{n}(\hat{f})$   via the  fixed isomorphism    (\ref{Eq: second can isom}). So
 the underlying complex of the twisted dg Lie algebra $\mathfrak{C}_{\RBO}(A)$ by the corresponding Maurer-Cartan element $\beta$   is exactly   the cochain complex of Rota-Baxter operator $C_{\RBO}^\bullet(A)$.
	
\end{proof}

\bigskip

 \section{Homotopy Rota-Baxter Algebras}

In this subsection, we will introduce the notion of homotopy Rota-Baxter algebras of any weight.

Recall $\overline{T^c}(sV)=\bigoplus\limits_{n=1}^\infty(sV)^{\ot n}$ and $\overline{\mathfrak{C}_{\Alg}}(V)=\Hom(\overline{T^c}(sV),sV)\subset \mathfrak{C}_{\Alg}(V)$. Denote
$$ \overline{\mathfrak{C}_{\RBO}}(V)=\Hom(\overline{T^c}(sV),V)\subset \mathfrak{C}_{\RBO}(V), $$ and set $$\overline{\mathfrak{C}_{\RBA}}(V) =\overline{\mathfrak{C}_{\Alg}}(V)\oplus\overline{\mathfrak{C}_{\RBO}}(V) \subset \mathfrak{C}_\RBA(V).$$  It is not difficult to see that    $\overline{{\mathfrak{C}_{\RBA}}}(V)$ is  an $L_\infty$-subalgebra of ${\mathfrak{C}_{\RBA}}(V)$.

\begin{defn}\label{Def: homotopy RB algebras}
	Let $V=\bigoplus\limits_{i\in \mathbb{Z}}V_i$ be a graded space. Then a homotopy Rota-Baxter algebra structure of weight $\lambda$ on $V$ is defined to be a Maurer-Cartan element in the $L_\infty$-algebra $\overline{{\mathfrak{C}_{\RBA}}}(V)$.
\end{defn}

\medskip

Let's make the definition explicit. Given  an element $$\alpha=(\{b_i\}_{i\geqslant 1}, \{R_i\}_{i\geqslant 1})\in \overline{{\mathfrak{C}_{\RBA}}}(V)_{-1}=\Hom(\overline{T^c}(sV),sV)_{-1}\oplus \Hom(\overline{T^c}(sV),V)_{-1} $$ with  $b_i:(sV)^{\ot i}\rightarrow sV$ and $R_i:(sV)^{\ot i}\rightarrow V$,   $\alpha$ satisfies the Maurer-Cartan equation if and only if  for each $n\geqslant 1$, the following equalities hold:
\begin{eqnarray}
	\label{Eq: A infinity}   \sum_{i+j+k= n,\atop i,   k\geq 0, j\geq 1  }b_{i+1+k}\circ(\id^{\ot i}\ot b_j\ot \id^{\ot k})=0,\end{eqnarray}
\begin{eqnarray} \label{Eq: homotopy-rb-operator}
	     \sum_{ l_1+\dots+l_k=n,\atop
 l_1, \dots, l_k\geqslant 1 } b_k\circ(sR_{l_1}\ot \cdots\ot sR_{l_k}) = \sum_{ 1\leqslant q\leqslant p} \sum_{ r_1+\dots+r_q+p-q=n \atop
r_1, \dots, r_q\geqslant 1   }\lambda^{p-q}\  (s R_{r_1})\big\{ b_p\{sR_{r_2},\dots,sR_{r_q}\}\big\}.
\end{eqnarray}

Define  two family of operators $\{m_n \}_{n\geqslant 1}$ and $ \{T_n\}_{n\geqslant1}$ as:
$$ m_n=s^{-1}\circ b_n\circ s^{\ot n}: V^{\ot n}\rightarrow V \quad \mathrm{and}\quad  T_n=R_n\circ s^{\ot n}: V^{\ot n}\rightarrow V.$$
For each $n\geqslant 1$, Equations~(\ref{Eq: A infinity}) and (\ref{Eq: homotopy-rb-operator})    are equivalent, respectively,  to:
\begin{eqnarray}\label{Eq: stasheff-id}
	\sum_{    i+j+k= n,\atop
i, k\geqslant 0, j\geqslant 1 } (-1)^{i+jk}m_{i+1+k}\circ\Big(\id^{\ot i}\ot m_j\ot \id^{\ot k}\Big)=0
\end{eqnarray} and
\begin{eqnarray} \label{Eq: homotopy RB-operator-version-2}
	   \sum\limits_{ l_1+\dots+l_k=n,\atop
 l_1, \dots, l_k\geqslant 1 } (-1)^{\alpha}m_k\circ\Big(T_{l_1}\ot \cdots \ot T_{l_k}\Big)=\sum\limits_{1\leqslant q\leqslant p}\sum\limits_{  r_1+\dots+r_q+p-q=n,\atop
  r_1, \dots, r_q\geqslant 1 } \sum\limits_{ i+1+k=r_1,\atop
   i, k\geqslant 0 }\sum\limits_{  j_1+\dots+j_q+q-1=p,\atop
j_1, \dots, j_q\geqslant 0  }
\end{eqnarray}
	$$ \quad  \quad\quad\quad\quad\quad  (-1)^\beta\lambda^{p-q} T_{r_1}\circ\Big(\id^{\ot i}\ot m_p\circ(\id^{\ot j_1}\ot T_{r_2}\ot \id^{\ot j_2}\ot \cdots \ot T_{r_q}\ot \id^{\ot j_q})\ot \id^{\ot k}\Big),
$$
where
\begin{align*}\alpha&=\frac{k(k-1)}{2}+\frac{n(n-1)}{2}+\sum_{j=1}^k(k-j)l_j,\\
	\beta&=\frac{p(p-1)}{2} +\sum_{j=1}^q\frac{r_j(r_j-1)}{2}+k+\sum_{l=2}^q\big(r_l-1\big)\big(i+\sum_{r=1}^{l-1}j_{r}+\sum_{t=2}^{l-1}r_t\big)+pi\\
	&=\frac{n(n-1)}{2}+i+(p+\sum\limits_{j=2}^q(r_j-1))k+\sum\limits_{l=2}^q(r_l-1)(\sum\limits_{r=l}^qj_r+q-l)
\end{align*}

As introduced in Subsection \ref{A-infinity algebras}, Equation~(\ref{Eq: stasheff-id}) is exactly the Stasheff identity (\ref{Eq: Stasheff}) in the definition of $A_\infty$-algebras. In particular,  the operator $m_1$ is a differential on $V$, and the operator $m_2$ induces an associative algebra structure on the homology space $\rmH_\bullet(V, m_1)$.

Equation~(\ref{Eq: homotopy RB-operator-version-2}) for $n=1,2$  gives
\begin{eqnarray}
	\label{operator-differential}   m_1\circ T_1=T_1\circ m_1,\end{eqnarray}
and \begin{eqnarray}
	\label{rbo-homotopy}  &  m_2\circ(T_1\ot T_1)-T_1\circ m_2\circ (\id\ot T_1)-T_1\circ m_2\circ (T_1\ot \id)-\lambda T_1\circ m_2 \\
	\notag   &= -\big(m_1\circ T_2+T_2\circ (\id\ot m_1)+T_2\circ(m_1\ot \id)\big).
\end{eqnarray}
Equation~(\ref{operator-differential}) implies that $T_1: (V, m_1)\to (V, m_1)$ is a chain map,  thus $T_1$ is well-defined on the $\rmH_\bullet(V, m_1)$;  Equation~(\ref{rbo-homotopy}) indicates that $T_1$ is a Rota-Baxter operator of weight $\lambda$ with respect to $m_2$ up to homotopy, whose  obstruction is just operator $T_2$. As a consequence, $(\rmH_\bullet(V, m_1), m_2, T_1)$ is  a Rota-Baxter algebra.

\bigskip

Now, we will give a homotopy version of  Proposition~\ref{Prop: new RB algebra}.

Let $V$ be a graded vector space.
Let $(\{b_n\}_{n\geqslant 1}, \{R_n\}_{n\geqslant 1})$ be a Maurer-Cartan element in $\overline{{\mathfrak{C}_{\RBA}}}(V)$.
 So we have the corresponding operators $\{m_n \}_{n\geqslant 1}$ and $ \{T_n\}_{n\geqslant1}$ which define a homotopy Rota-Baxter algebra structure on $V$.

 Define a new family of operators $\{\widetilde{b}_n\}_{n\geqslant 1}$
\[\widetilde{b}_n=\sum_{p=1}^n\sum_{q=0}^{p-1}\sum_{l_1+\dots+l_q+p-q=n,\atop l_1, \dots, l_q\geqslant 1}\lambda^{p-q-1}b_p\{sR_{l_1},\dots,sR_{l_q}\}\]
and  set  $\widetilde{m}_n:=s^{-1}\circ\widetilde{b}_n\circ s^{\ot n}:V^{\ot} \rightarrow V$.
Introduce   another family of operators $\{\widetilde{R}_n:(sV)^{\ot n}\rightarrow V\}_{n\geqslant 1}$  as follows:
	\begin{itemize}
		\item[(i)] put $R_n^1:=\lambda^{n-1}R_n:(sV)^{\ot n}\rightarrow V$ for any $n\geqslant 1$;
		\item[(ii)] define $R_n^2:=\sum\limits_{1\leqslant q\leqslant p-1}\sum\limits_{l_1+\dots+l_q+p-q=n}\lambda^{p-q-1}s^{-1}(sR_p)\{sR^1_{l_1},\dots,sR^1_{l_q}\}$;

		\item[(iii)] taking induction, define $R_n^k=\sum\limits_{1\leqslant q\leqslant p-1,t_j\leqslant k-1}\sum\limits_{l_1+\dots+l_q+p-q=n}\lambda^{p-q-1}s^{-1}(sR_p)\{sR^{t_1}_{l_1},\dots,sR^{t_{q}}_{l_q}\}$;

		\item[(iv)] define $\widetilde{R}_n=\sum\limits_{k=1}^\infty R_n^k$.
\end{itemize}
Note that for any given $n\geqslant 1$, this is always a finite sum, thus $\widetilde{R}_n$ is a well-defined map of degree $-1$ in  $ \Hom((sV)^{\ot n}, V)$. Impose  $\widetilde{T}_n=\widetilde{R}_n\circ s^{\ot n}: V^{\ot n}\rightarrow V.$

\begin{prop}{\label{Prop: homotopy RB-arising-from-homotopy-Rota-Baxter}}
\begin{itemize}
\item[(i)]
	The pair $(V,\{\widetilde{m}_n\}_{n\geqslant 1})$     forms an $A_\infty$-algebra. And the family of operators $\{T_n\}_{n\geqslant 1}$ defines an $A_\infty$-morphism from $(V,\{\widetilde{m}_n\}_{n\geqslant 1})$ to $(V,\{m_n\}_{n\geqslant 1})$.

 \item[(ii)]  These two family of operators $\{\widetilde{b}_n\}_{n\geqslant 1}\bigcup\{\widetilde{R}_n\}_{n\geqslant 1}$ is also a Maurer-Cartan element in $\mathfrak{C}_\RBA(V)$, thus a homotopy Rota-Baxter algebra structure of weight $\lambda$ on $V$.


     \end{itemize}
\end{prop}

 For a proof, see Appendix B.

\section{The Minimal model for the operad of Rota-Baxter algebras}
In the last section, we have defined the notion of homotopy Rota-Baxter algebras of any weight. In this section,  we will prove that the dg operad governing  homotopy Rota-Baxter algebras of weight $\lambda$ is a minimal model of the operad for Rota-Baxter algebras of weight $\lambda$.  Therefore,  the cohomology theory for Rota-Baxter algebras  defined before is  the right cohomology theory for Rota-Baxter algebras in the sense of operad theory.

\bigskip

For basic  theory of operads, we refer the reader to the textbooks \cite{LV, BD}.
As we will only care about nonsymmetric operads in this paper, we will delete the adjective ``nonsymmetric" everywhere.
For a collection $M=\{M(n)\}_{n\geqslant 1} $ of (graded) vector spaces, denote by $ \mathcal{F}(M)$ the free (graded) operad generated by $M$. Recall that a dg operad is called quasi-free if its underlying graded operad is free.

\begin{defn}[\cite{DCV13}] \label{Def: Minimal model of operads} A minimal model for an operad $P$  is a quasi-free dg operad $ (\mathcal{F}(M),d)$ together with a surjective quasi-isomorphism of operads $(\mathcal{F}(M), \partial)\overset{\sim}{\twoheadrightarrow}P$, where the dg operad $(\mathcal{F}(M),  \partial)$  satisfies the following conditions:
	\begin{itemize}
		\item[(i)] the differential $d$ is decomposable, i.e. $\partial$ takes $M$ to $\mathcal{F}(M)^{\geqslant 2}$, the subspace of $\mathcal{F}(M)$ consisting of elements with weight $\geqslant 2$;
		\item[(ii)] the generating collection $M$ admits a decomposition $M=\bigoplus\limits_{i\geqslant 1}M_{(i)}$  such that $\partial(M_{(k+1)})\subset \mathcal{F}\Big(\bigoplus\limits_{i=1}^kM_{(i)}\Big)$ for any $k\geqslant 1$
(usually $M_{(i)}$ is the degree $i$ part).  \end{itemize}
\end{defn}
\begin{thm}[\cite{DCV13}] When an operad $P$ admits a minimal model, it is unique up to isomorphisms.
\end{thm}

The  operad for Rota-Baxter algebras of weight $\lambda$, denoted by $\RB^\lambda$,  is generated by a unary operator $T$ and a binary operatore $\mu$ with the operadic relation generated by $$ \mu\circ_1\mu-\mu\circ_2\mu\quad \mathrm{and}\quad (\mu\circ_1T)\circ_2T-(T\circ_1\mu)\circ_1T-(T\circ_1\mu)\circ_2T-\lambda T\circ_1\mu.$$

Recall that a homotopy Rota-Baxter algebra structure on a graded space $V$ consists of two families of operators
$\{m_n\}_{n\geqslant 1}$ and $\{T_n\}_{n\geqslant 1}$ satisfying Equations~(\ref{Eq: stasheff-id})(\ref{Eq: homotopy RB-operator-version-2}). As  operator $-m_1$ makes $V$ into a complex, it induces a differential operator  $\partial$ on graded space $\Hom(V^{\ot n},V)$ containing $m_n, T_n$. Rewriting  Equations (\ref{Eq: stasheff-id})(\ref{Eq: homotopy RB-operator-version-2}) gives:
\begin{eqnarray*}
	\partial{m_n} = (-m_1)\circ m_n-(-1)^{n-2}\sum_{i=1}^n m_n\circ_i (-m_1)=\sum_{j=2}^{n-1}\sum_{i=1}^{n-j+1}(-1)^{i+1+j(n-i)}m_{n-j+1}\circ_i m_j
\end{eqnarray*}
\begin{eqnarray*}   \partial T_n&=&(-m_1)\circ T_n-(-1)^{n-1}\sum_{i=1}^n T_n\circ_i (-m_1)\\
	\notag  &=&\sum_{k=2}^n\sum_{l_1+\cdots+l_k=n\atop l_1, \dots, l_k\geqslant 1}(-1)^{\alpha'}\Big(\cdots\big((m_k\circ_1 T_{l_1})\circ_{l_1+1}T_{l_2}\big)\cdots\Big)\circ_{l_1+\cdots+l_{k-1}+1}T_{l_k}+\\
	\notag  & &\sum\limits_{{\small\substack{2\leqslant p\leqslant n \\ 1\leqslant q\leqslant p
				 }}}\sum\limits_{\small\substack{ r_1+\dots+r_q+p-q=n\\r_1, \dots, r_q\geqslant 1\\1\leqslant i\leqslant r_1\\1\leqslant k_1<\dots< k_{q-1}\leqslant p }}(-1)^{\beta'}\lambda^{p-q}\Big(T_{r_1}\circ_i \Big(\big((\cdots(( m_p\circ_{k_1}T_{r_2})\circ_{k_2+r_2-1} T_{r_3}))\cdots \big)\circ_{k_{q-1}+r_2+\dots+r_{q-1}-q+2}T_{r_q}\Big)\Big),
\end{eqnarray*}
where
\begin{eqnarray}\label{Eq: sign   alpha'}
	\alpha'&=&\frac{k(k-1)}{2}+\sum_{j=1}^k(k-j)l_j=\sum_{j=1}^k(k-j)(l_j-1),\\
	\label{Eq: sign   beta'}\beta'&=&i+\big(p+\sum\limits_{j=2}^q(r_j-1)\big)\big(r_1-i\big)+\sum\limits_{j=2}^q(r_j-1)(p-k_{j-1})	\end{eqnarray}

Now we introduce the dg operad of homotopy Rota-Baxter algebras of weight $\lambda$.
\begin{defn}
Let $M=(M(0), M(1),\dots, M(n),\dots)$ be the graded collection where $M(0)=0$, arity $1$ part $M(1)$ is the one-dimensional graded  space spanned by $T_1$ with $|T_1|=0$, and for $ n\geqslant 2$,  arity $n$ part $M(n)$  is the two-dimensional graded  space  spanned by $T_n, m_n$ with $|T_n|=n-1$, $|m_n|=n-2$.
The dg operad for homotopy Rota-Baxter  algebras of weight $\lambda$, denoted by $\RB_\infty^\lambda$, is the free graded operad generated by $M$   endowed with differential   $\partial$ subject to
\begin{eqnarray}\label{Eq: defining HRB 1}
	\partial{m_n} =  \sum_{j=2}^{n-1}\sum_{i=1}^{n-j+1}(-1)^{i+1+j(n-i)}m_{n-j+1}\circ_i m_j
\end{eqnarray} and
\begin{eqnarray} \label{Eq: defining HRB 2}   & &   \\
  \notag   &\partial T_n =&\sum\limits_{k=2}^n\sum_{l_1+\cdots+l_k=n\atop l_1, \dots, l_k\geqslant 1}(-1)^{\alpha'}\Big(\cdots\big((m_k\circ_1 T_{l_1})\circ_{l_1+1}T_{l_2}\big)\cdots\Big)\circ_{l_1+\cdots+l_{k-1}+1}T_{l_k}+\\
	\notag    & &\sum\limits_{{\small\substack{2\leqslant p\leqslant n \\ 1\leqslant q\leqslant p
				 }}}\sum\limits_{\small\substack{ r_1+\dots+r_q+p-q=n\\r_1, \dots, r_q\geqslant 1\\1\leqslant i\leqslant r_1\\1\leqslant k_1<\dots< k_{q-1}\leqslant p }}(-1)^{\beta'}\lambda^{p-q}\Big(T_{r_1}\circ_i \Big(\big((\cdots(( m_p\circ_{k_1}T_{r_2})\circ_{k_2+r_2-1} T_{r_3}))\cdots \big)\circ_{k_{q-1}+r_2+\dots+r_{q-1}-q+2}T_{r_q}\Big)\Big).
\end{eqnarray}
where the signs $(-1)^{\alpha'}$ and $(-1)^{\beta'}$ are determined by Equations~\eqref{Eq: sign   alpha'}\eqref{Eq: sign   beta'} respectively.
\end{defn}

We will use planar rooted  trees to display  elements in the dg operad $\RB_\infty^\lambda$.
  We use   the corolla with $n$ leaves and a black vertex to represent generators $m_n, n\geqslant 2$ and  the corolla with $n$ leaves and a white vertex to represent generators $T_n, n\geqslant 1$:
  \begin{eqnarray*}
  		\begin{tikzpicture}[grow'=up]
  			\tikzstyle{every node}=[thick,minimum size=6pt, inner sep=1pt]
  			 \node(r)[fill=black,circle,label=right:$m_n$]{}
  				child{node(1){1}}
  				child{node(i) {i}}
  				child{node(n){n}};
  			\draw [dotted,line width=0.5pt] (1)--(i);
  			\draw [dotted, line width=0.5pt] (i)--(n);
  		\end{tikzpicture}
  	\hspace{8mm}
  	\begin{tikzpicture}[grow'=up]
  		\tikzstyle{every node}=[thick,minimum size=6pt, inner sep=1pt]
  	\node[draw,circle,label=right:$T_n$]{}
  	child{node(1){1}}
  	child{node(i){i}}
  	child{node(n){n}};
	\draw [dotted,line width=0.5pt] (1)--(i);
\draw [dotted, line width=0.5pt] (i)--(n);
  	\end{tikzpicture}
  	\end{eqnarray*}

By \cite[Chapter 3]{BD}, a planar rooted  tree with all internal vertices dyed white or black (such a tree will be called a tree monomial) gives an element in $\RB_\infty^\lambda$ by composing its vertices clockwisely.
 Conversely, any element in $\RB_\infty^\lambda$ can be represented by such a unique planar rooted  tree in this way.
 For example, the element   $$(((m_3\circ_{1}T_4)\circ_3m_4)\circ_9m_2)\circ_{10}T_2$$ can be represented by the following tree:
 \begin{eqnarray*}\footnotesize
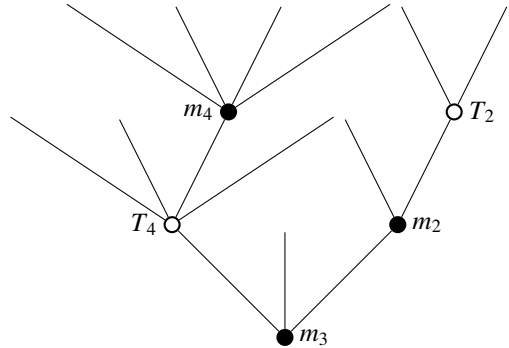
 	\begin{tikzpicture}[grow'=up]
\tikzstyle{every node}=[level distance=30mm,sibling distance=10em,thick,minimum size=5pt, inner sep=2pt]
\node[fill=black,draw,circle,label=right:$m_3$]{}
child{
	node[draw, circle, label=left:$T_4$](1){}
  child{node(1-1){}}
  child{node(1-2){}}
  child{
  	node[fill=black, draw, circle,label=left:$m_4$](1-3){}
  child{node(1-3-1){}}
  child{node(1-3-2){}}
  child{node(1-3-3){}}
  child{node(1-3-4){}}
}
  child{node(1-4){}}
}
child{node(2){}}
child{
	node[fill=black,draw,circle,label=right:$m_2$](3){}
    child{node(3-1){}}
child{
	node[circle, draw, label=right:$T_2$](3-2){}
child{node(3-2-1){}}
child{node(3-2-2){}}}
};
 	\end{tikzpicture}
 \end{eqnarray*}

In this means, the action of differential operator $\partial$ on generators can be expressed by trees as follows:

\begin{eqnarray*}
	\begin{tikzpicture}[scale=0.6]
		\tikzstyle{every node}=[thick,minimum size=6pt, inner sep=1pt]
		\node(a) at (-4,0.5){\begin{huge}$\partial$\end{huge}};
		\node[circle, fill=black, label=right:$m_n$] (b0) at (-2,-0.5)  {};
		\node (b1) at (-3.5,1.5)  [minimum size=0pt,label=above:$1$]{};
		\node (b2) at (-2,1.5)  [minimum size=0pt,label=above:$i$]{};
		\node (b3) at (-0.5,1.5)  [minimum size=0pt,label=above:$n$]{};
		\draw        (b0)--(b1);
		\draw        (b0)--(b2);
		\draw        (b0)--(b3);
		\draw [dotted,line width=1pt] (-3,1)--(-2.2,1);
		\draw [dotted,line width=1pt] (-1.8,1)--(-1,1);
	\end{tikzpicture}
	&
	\begin{tikzpicture}
		\node(0){{\Large$=\sum\limits_{j=2}^{n-1} \sum\limits_{i=1}^{n-j+1}(-1)^{i+1+j(n-i)}$}};
	\end{tikzpicture}
&
	\begin{tikzpicture}
		\tikzstyle{every node}=[thick,minimum size=6pt, inner sep=1pt]
		\node(e0) at (0,-1.5)[circle, fill=black,label=right:$m_{n-j+1}$]{};
		\node(e1) at(-1.5,0){{\tiny$1$}};
		\node(e2-0) at (0,-0.5){{\tiny$i$}};
		\node(e3) at (1.5,0){{\tiny{$n-j+1$}}};
		\node(e2-1) at (0,0.5) [circle,fill=black,label=right: $m_j$]{};
		\node(e2-1-1) at (-1,1.5){{\tiny$1$}};
		\node(e2-1-2) at (1, 1.5){{\tiny $j$}};
		\draw [dotted,line width=1pt] (-0.7,-0.5)--(-0.2,-0.5);
		\draw [dotted,line width=1pt] (0.3,-0.5)--(0.8,-0.5);
		\draw [dotted,line width=1pt] (-0.4,1)--(0.4,1);
		\draw        (e0)--(e1);
		\draw         (e0)--(e3);
		\draw         (e0)--(e2-0);
		\draw         (e2-0)--(e2-1);
		\draw        (e2-1)--(e2-1-1);
		\draw        (e2-1)--(e2-1-2);
		\end{tikzpicture}	
\end{eqnarray*}

\begin{eqnarray*}
	\begin{tikzpicture}[scale=0.6]
		\tikzstyle{every node}=[thick,minimum size=6pt, inner sep=1pt]
		\node(a) at (-4,0.5){\begin{huge}$\partial$\end{huge}};
		\node[circle, draw, label=right:$T_n$] (b0) at (-2,-0.5)  {};
		\node (b1) at (-3.5,1.5)  [minimum size=0pt,label=above:$1$]{};
		\node (b2) at (-2,1.5)  [minimum size=0pt,label=above:$i$]{};
		\node (b3) at (-0.5,1.5)  [minimum size=0pt,label=above:$n$]{};
		\draw        (b0)--(b1);
		\draw        (b0)--(b2);
		\draw        (b0)--(b3);
		\draw [dotted,line width=1pt] (-3,1)--(-2.2,1);
		\draw [dotted,line width=1pt] (-1.8,1)--(-1,1);
	\end{tikzpicture}
	&
	\begin{tikzpicture}
		\node(0){{\Large$= \sum\limits_{k=2}^n\sum\limits_{l_1+\cdots+l_k=n \atop l_1, \dots, l_k\geqslant 1}(-1)^{\alpha'}$}};
	\end{tikzpicture}
	&
	\begin{tikzpicture}
		\tikzstyle{every node}=[thick,minimum size=6pt, inner sep=1pt]
		\node(e0) at (0,-1.5)[circle, fill=black,label=right:$m_{k}$]{};
		\node(e1) at(-1.2,-0.3)[circle, draw, label=left:$T_{r_1}$]{};
		\node(e1-1) at(-2,0.8){};
		\node(e1-2) at (-1,0.8){};
		\node(e2-0) at (0,-0.7){{\tiny$i$}};
		\node(e3) at (1.2,-0.3)[draw, circle, label=right: $T_{r_k}$]{};
		\node(e3-1) at (1,0.8){};
		\node(e3-2) at (2,0.8){};
		\node(e2-1) at (0,0) [draw,circle,label=right: $T_{r_i}$]{};
		\node(e2-1-1) at (-0.5,1){};
		\node(e2-1-2) at (0.5, 1){};
		\draw [dotted,line width=1pt] (-0.7,-0.5)--(-0.2,-0.5);
		\draw [dotted,line width=1pt] (0.3,-0.5)--(0.8,-0.5);
		\draw [dotted,line width=1pt] (-0.2,0.5)--(0.2,0.5);
		\draw [dotted,line width=1pt] (-1.6,0.5)--(-1.2, 0.5);
		\draw [dotted, line width=1pt] (1.2,0.5)--(1.6,0.5);
		\draw        (e0)--(e1);
		\draw         (e0)--(e3);
		\draw         (e1)--(e1-1);
		\draw          (e1)--(e1-2);
		\draw         (e0)--(e2-0);
		\draw         (e2-0)--(e2-1);
		\draw        (e2-1)--(e2-1-1);
		\draw        (e2-1)--(e2-1-2);
		\draw        (e3)--(e3-1);
		\draw        (e3)--(e3-2);
		
	\end{tikzpicture}	\\
&\begin{tikzpicture}
	\node(0){{\Large$+\sum\limits_{{\small\substack{2\leqslant p\leqslant n \\ 1\leqslant q\leqslant p
				 }}}\sum\limits_{\small\substack{ r_1+\dots+r_q+p-q=n\\r_1, \dots, r_q\geqslant 1\\1\leqslant i\leqslant r_1\\1\leqslant k_1<\dots< k_{q-1}\leqslant p }}(-1)^{\beta'}\lambda^{p-q}$}};
\end{tikzpicture}&
\begin{tikzpicture}
		\tikzstyle{every node}=[thick,minimum size=4pt, inner sep=1pt]
	\node(e0) at (0,-1.5)[circle, draw,label=right:{\footnotesize $\ T_{r_1}$}]{};
	\node(e1) at(-1.8,-0.3){};
	\node(e1-1) at(-2,0.8){};
	\node(e1-2) at (-1,0.8){};
	\node(e2-0) at (0,-0.9){{\tiny$i$}};
	\node(e3) at (1.8,-0.3){};
	\node(e3-1) at (1,0.8){};
	\node(e3-2) at (2,0.8){};
	\node(e2-1) at (0,-0.3) [fill=black,draw,circle,label=right: {\scriptsize $\ \ m_p$}]{};
	\node(e2-1-1) at (-1.9,1){};
	\node(e2-1-2) at(-0.6,0.6){{\tiny$k_1$}};
	\node(e2-1-2-0) at (-1,1.2)[draw, circle,label=left:{\tiny $T_{r_2}$}]{};
	\node(e2-1-2-1) at (-1.4,1.9){};
	\node(e2-1-2-2) at (-0.6,1.9){};
	\node(e2-1-3) at(-0.4,1.5){};
	\node(e2-1-4) at (0.4,1.5){};
	\node(e2-1-5) at (0.6,0.6){{\tiny $k_{q-1}$}};
	\node(e2-1-5-0) at (1,1.2) [circle, draw,label=right: {\tiny $\ T_{r_q}$}]{};
	\node(e2-1-5-1) at (0.6,1.9){};
	\node(e2-1-5-2) at (1.4,1.9){};
	\node(e2-1-6) at (1.9, 1){};
	\draw [dotted,line width=1pt] (-0.7,-0.7)--(-0.2,-0.7);
	\draw [dotted,line width=1pt] (0.3,-0.7)--(0.8,-0.7);
	\draw [dotted, line width=1pt] (-1.3,0.9)--(1.3,0.9);
	\draw        (e0)--(e1);
	\draw [dotted, line width =1pt](-1.2,1.7)--(-0.7,1.7);
	\draw [dotted, line width=1pt](0.7,1.7)--(1.2,1.7);
	\draw         (e0)--(e3);
	\draw         (e0)--(e2-0);
	\draw         (e2-0)--(e2-1);
	\draw        (e2-1)--(e2-1-1);
	\draw      (e2-1-2)--(e2-1-2-0);
	\draw        (e2-1)--(e2-1-2);
		\draw        (e2-1)--(e2-1-3);
				\draw        (e2-1)--(e2-1-4);	
				\draw        (e2-1)--(e2-1-5);
				\draw         (e2-1-5)--(e2-1-5-0);
					\draw        (e2-1)--(e2-1-6);
					\draw (e2-1-2-0)--(e2-1-2-1);
					\draw (e2-1-2-0)--(e2-1-2-2);
					\draw (e2-1-5-0)--(e2-1-5-1);
					\draw (e2-1-5-0)--(e2-1-5-2);
\end{tikzpicture}
\end{eqnarray*}

The following result is the main result of this section, whose proof occupies the rest of this section.
\begin{thm}\label{Thm: Minimal model}
	The dg operad $\RB_\infty^\lambda$ is the minimal model of the operad $\RB^\lambda$.
\end{thm}

\bigskip

 Now, we are going to prove that there exists a quasi-isomorphism of dg operads $\RB_\infty^\lambda\xrightarrow{~}\RB^\lambda$, where $\RB^\lambda$ is considered as a dg operad concentrated in degree 0.

  The degree zero part of $\RB_\infty^\lambda$ is the free graded operad  generated by $\{m_2\}\cup\{T_1\}$. The image of $\partial$  in this  degree zero part is the operadic ideal   generated by $\partial T_2, \partial m_3$. By definition, we have:
   \begin{eqnarray*}
 	\partial(m_3)&=&m_2\circ_1 m_2-m_2\circ_2m_2,\\
 	\partial (T_2)&=&-T_1\circ_1(m_2\circ_1 T_1)-T_1\circ_1(m_2\circ_2 T_1)-\lambda T_1\circ m_2+(m_2\circ_1T_1)\circ_2 T_1.
 	\end{eqnarray*}
Thus $\rmH_0(\RB_\infty^\lambda)  \cong \RB^\lambda$.

To prove the natural map $\phi: \RB_\infty^\lambda\twoheadrightarrow \RB^\lambda$ is a quasi-isomorphism, we just need to prove that $\rmH_i(\RB_\infty^\lambda)=0$ for all $i\geqslant 1$.

\medskip

  We need the following notion of graded  path-lexicographic ordering on $\RB_\infty^\lambda$.

  Each tree monomial gives rise to a path sequence; for details, see \cite[Chapter 3]{BD}.  More precisely,
 to any tree monomial $\mathcal{T}$ with  $n$ leaves (written as $\mbox{arity}(\mathcal{T})=n$), we can associate   with a sequence   $(x_1, \dots, x_n)$  where  $x_i$ is the word formed by   generators of $\RB_\infty^\lambda$ corresponding to the vertices along the unique path from the root of $\mathcal{T}$  to its $i$-th leaf.

  For two graded tree monomials $\mathcal{T},\mathcal{T}'$, we compare $\mathcal{T},\mathcal{T}'$ in the following way:
 \begin{itemize}
 	\item[(i)] If $\mbox{arity}(\mathcal{T})>\mbox{arity}(\mathcal{T}')$, then  $\mathcal{T}>\mathcal{T}'$;
 	\item[(ii)] if $\mbox{arity}(\mathcal{T})=\mbox{arity}(\mathcal{T}')$, and $\deg(\mathcal{T})>\deg(\mathcal{T}')$, then $\mathcal{T}>\mathcal{T}'$, where $\deg(\mathcal{T})$ is the sum of the degrees of all generators of $\RB_\infty^\lambda$ appearing in   tree monomial $\mathcal{T}$;
 	\item[(iii)] if $\mbox{arity}(\mathcal{T})=\mbox{airty}(\mathcal{T}')(=n), \deg(\mathcal{T})=\deg(\mathcal{T}')$, then $\mathcal{T}>\mathcal{T}'$ if the path sequences   $(x_1,\dots,x_n), (x'_1,\dots,x_n')$ associated to $\mathcal{T}, \mathcal{T}'$ satisfies $(x_1,\dots,x_n)>(x_1',\dots,x_n')$ with respect to the length-lexicographic order of words induced by $$T_1<m_2<T_2<m_3<\cdots<T_n<m_{n+1}<T_{n+1}<\cdots.$$

 \end{itemize}
 It is ready to see that this is a well order.
Under this order, the leading term in the expansion of $\partial(m_n), \partial(T_n)$ are the following trees respectively:
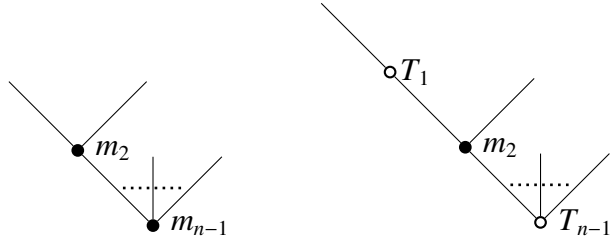
\begin{figure}[h]
	\begin{tikzpicture}
		\tikzstyle{every node}=[thick,minimum size=4pt, inner sep=1pt]
		\node(1) at (0,0) [draw, circle, fill=black, label=right:$\ m_{n-1}$]{};
		\node(2-1) at (-1,1) [draw, circle, fill=black, label=right: $\ m_2$]{};
		\node(3-1) at (-2,2){};
		\node(3-2) at (0,2){};
		\node(2-2) at (0,1){};
		\node(2-3) at (1,1){};
		\draw (1)--(2-1);
		\draw  (1)--(2-2);
		\draw  (1)--(2-3);
		\draw (2-1)--(3-1);
		\draw (2-1)--(3-2);
		\draw [dotted,line width=1pt](-0.4,0.5)--(0.4,0.5);
	\end{tikzpicture}
\hspace{8mm}
\begin{tikzpicture}
	\tikzstyle{every node}=[thick,minimum size=4pt, inner sep=1pt]
	\node(1) at (0,0) [draw, circle, label=right:$\ T_{n-1}$]{};
	\node(2-1) at (-1,1) [draw, circle, fill=black, label=right: $\ m_2$]{};
	\node(3-1) at (-2,2)[draw, circle, label=right: $T_1$]{};
	\node(3-2) at (0,2){};
	\node(2-2) at (0,1){};
	\node(2-3) at (1,1){};
	\node(4) at (-3,3){};
	\draw (1)--(2-1);
	\draw  (1)--(2-2);
	\draw  (1)--(2-3);
	\draw (2-1)--(3-1);
	\draw (2-1)--(3-2);
	\draw (3-1)--(4);
	\draw [dotted,line width=1pt](-0.4,0.5)--(0.4,0.5);
\end{tikzpicture}
\end{figure}

\begin{defn}
	Let $\mathcal{S}$ be a generator of degree $\geqslant 1$ in $ \RB_\infty^\lambda$. Denote the leading monomial of $\partial \mathcal{S}$ by $\widehat{\mathcal{S}}$ and the  coefficient of $\widehat{\mathcal{S}}$ in $\partial$ is written  as $l_\mathcal{S}$.  A tree monomial of the form $\widehat{\mathcal{S}}$ is called typical.
\end{defn}
It can be easily seen that the coefficient $l_\mathcal{S}$ is always $\pm 1$.

 To prove that $\rmH_i(\RB_\infty^\lambda)=0, i\geqslant 1$, we are going to construct a homotopy map $\mathfrak H$, i.e., a map of degree 1, $\mathfrak{H}: \RB_\infty^\lambda\rightarrow \RB_\infty^\lambda$ satisfying  $\partial \mathfrak{H}+\mathfrak{H}\partial =\mathrm{Id}$ in positive degrees.

 \begin{defn}\label{Def: effective tree monomials}  A tree monomial $\mathcal{T}$ in $\RB_\infty^\lambda$ is called effective if $\mathcal{T}$ satisfies the following conditions:
 	\begin{itemize}
 		\item[(i)] There exists a typical divisor $\mathcal{T}'=\widehat{\mathcal{S}}$ in $\mathcal{T}$ such that: on the path from the root of $\mathcal{T}'$ to the leftmost leaf $l$ of $\mathcal{T}$ above the root of $\mathcal{T}'$, there are no other typical divisors, and there are no vertex of positive degree on this path except the root of $\mathcal{T}'$ possibly.
 		\item[(ii)] For any leaf $l'$ of $\mathcal{T}$ which lies on the left of $l$, there are no vertices of positive degree and no typical divisors  on the path from the root of $\mathcal{T}$ to $l'$.
 	\end{itemize}
The typical divisor $\mathcal{T}'$ is called the effective divisor of $\mathcal{T}$ and  the leaf $l$ is called the typical leaf of $\mathcal{T}$.
\end{defn}

Morally, the effective divisor of a tree monomial $\mathcal{T}$ is the left-upper-most typical divisor of $\mathcal{T}$.
It can be easily see that for the effective divisor $\mathcal{T}'$ in $\mathcal{T}$ with effective leaf $l$, any vertex in $\mathcal{T}'$ doesn't belong to the path from root of $\mathcal{T}$ to any leaf $l'$ located on the left of $l$.

\begin{exam}Consider three tree monomials with the same underlying tree:
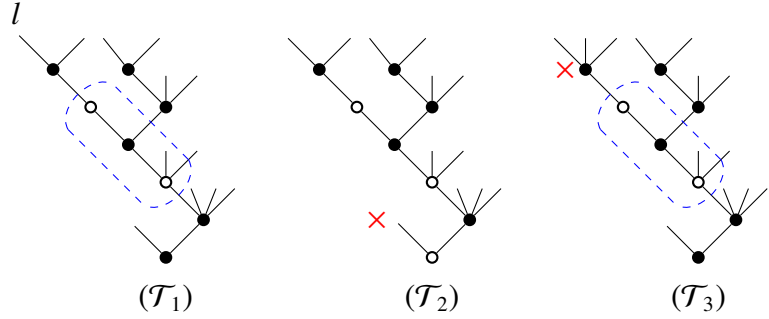
\begin{eqnarray*}
\begin{tikzpicture}
	\tikzstyle{every node}=[thick,minimum size=4pt, inner sep=1pt]
	\node(1)at (0,0)[circle, draw, fill=black]{};
	\node(2-1) at (-0.5,0.5){};
	\node(2-2) at (0.5,0.5)[circle,draw, fill=black]{};
	\node(3-1) at (0,1)[circle, draw]{};
	\node(3-2) at (0.3,1){};
	\node(3-3) at (0.7,1){};
	\node(3-4) at (1,1){};
	\node(4-1) at (-0.5,1.5)[circle, draw,fill=black]{};
	\node (4-2) at(0,1.5){};
	\node(4-3) at (0.5,1.5){};
	\node(5-1) at (-1,2)[circle, draw]{};
	\node(5-2) at(0,2)[circle, draw, fill=black] {};
	\node (6-1) at (-1.5, 2.5)[circle,draw, fill=black]{};
	\node(6-2) at (-0.5, 2.5)[circle, draw, fill=black]{};
	\node(6-3) at(0,2.5){};
	\node(6-4) at (0.5,2.5){};
	\node(7-1) at (-2,3)[minimum size=0pt, label=above:$l$]{};
	\node(7-2) at (-1,3){};
	\node(7-3) at(-0.9,3){};
	\node(7-4) at (0,3){};
	\draw (1)--(2-1);
	\draw (1)--(2-2);
	\draw (2-2)--(3-1);
	\draw (2-2)--(3-2);
	\draw (2-2)--(3-3);
	\draw (2-2)--(3-4);
	\draw (3-1)--(4-1);
	\draw (3-1)--(4-2);
	\draw (3-1)--(4-3);
	\draw (4-1)--(5-1);
	\draw (4-1)--(5-2);
	\draw (5-1)--(6-1);
	\draw (6-1)--(7-1);
	\draw (6-1)--(7-2);
	\draw (5-2)--(6-2);
	\draw (5-2)--(6-3);
	\draw (5-2)--(6-4);
	\draw(6-2)--(7-3);
	\draw(6-2)--(7-4);
\draw[dashed,blue](-1.3,1.7) to [in=150, out=120] (-0.7,2.3) ;
\draw[dashed,blue](-1.3,1.7)--(-0.3,0.7);
\draw[dashed,blue] (0.3,1.3)to [in=-30, out=-60] (-0.3,0.7);
	\node[minimum size=0pt,inner sep=0pt,label=below:$(\mathcal{T}_1)$] (name) at (0,-0.3){};
	\draw[dashed, blue](0.3,1.3)--(-0.7,2.3);
\end{tikzpicture}
\hspace{4mm}
\begin{tikzpicture}
	\tikzstyle{every node}=[thick,minimum size=4pt, inner sep=1pt]
	\node(1)at (0,0)[circle, draw]{};
	\node(2-1) at (-0.5,0.5)[minimum size=0pt, label=left:$\red \times$]{};
	\node(2-2) at (0.5,0.5)[circle,draw, fill=black]{};
	\node(3-1) at (0,1)[circle, draw]{};
	\node(3-2) at (0.3,1){};
	\node(3-3) at (0.7,1){};
	\node(3-4) at (1,1){};
	\node(4-1) at (-0.5,1.5)[circle, draw,fill=black]{};
	\node (4-2) at(0,1.5){};
	\node(4-3) at (0.5,1.5){};
	\node(5-1) at (-1,2)[circle, draw]{};
	\node(5-2) at(0,2)[circle, draw, fill=black] {};
	\node (6-1) at (-1.5, 2.5)[circle,draw, fill=black]{};
	\node(6-2) at (-0.5, 2.5)[circle, draw, fill=black]{};
	\node(6-3) at(0,2.5){};
	\node(6-4) at (0.5,2.5){};
	\node(7-1) at (-2,3){};
	\node(7-2) at (-1,3){};
	\node(7-3) at (-0.9,3){};
	\node(7-4) at(0,3){};
	\draw(6-2)--(7-3);
	\draw(6-2)--(7-4);
	\draw (1)--(2-1);
	\draw (1)--(2-2);
	\draw (2-2)--(3-1);
	\draw (2-2)--(3-2);
	\draw (2-2)--(3-3);
	\draw (2-2)--(3-4);
	\draw (3-1)--(4-1);
	\draw (3-1)--(4-2);
	\draw (3-1)--(4-3);
	\draw (4-1)--(5-1);
	\draw (4-1)--(5-2);
	\draw (5-1)--(6-1);
	\draw (6-1)--(7-1);
	\draw (6-1)--(7-2);
	\draw (5-2)--(6-2);
	\draw (5-2)--(6-3);
	\draw (5-2)--(6-4);
	\node[minimum size=0pt,inner sep=0pt,label=below:$(\mathcal{T}_2)$] (name) at (0,-0.3){};
\end{tikzpicture}
\hspace{4mm}\begin{tikzpicture}
	\tikzstyle{every node}=[thick,minimum size=4pt, inner sep=1pt]
	\node(1)at (0,0)[circle, draw, fill=black]{};
	\node(2-1) at (-0.5,0.5){};
	\node(2-2) at (0.5,0.5)[circle,draw, fill=black]{};
	\node(3-1) at (0,1)[circle, draw]{};
	\node(3-2) at (0.3,1){};
	\node(3-3) at (0.7,1){};
	\node(3-4) at (1,1){};
	\node(4-1) at (-0.5,1.5)[circle, draw,fill=black]{};
	\node (4-2) at(0,1.5){};
	\node(4-3) at (0.5,1.5){};
	\node(5-1) at (-1,2)[circle, draw]{};
	\node(5-2) at(0,2)[circle, draw, fill=black] {};
	\node (6-1) at (-1.5, 2.5)[circle,draw, fill=black,label=left:$\red \times$]{};
	\node(6-2) at (-0.5, 2.5)[circle,draw,fill=black]{};
	\node(6-3) at(0,2.5){};
	\node(6-4) at (0.5,2.5){};
	\node(7-1) at (-2,3){};
	\node(7-2) at (-1.5,3){};
	\node(7-3) at(-1,3){};
	\node(7-4) at(-0.9,3){};
	\node(7-5) at(0,3){};
	\draw(6-2)--(7-4);
	\draw(6-2)--(7-5);
	\draw (1)--(2-1);
	\draw (1)--(2-2);
	\draw (2-2)--(3-1);
	\draw (2-2)--(3-2);
	\draw (2-2)--(3-3);
	\draw (2-2)--(3-4);
	\draw (3-1)--(4-1);
	\draw (3-1)--(4-2);
	\draw (3-1)--(4-3);
	\draw (4-1)--(5-1);
	\draw (4-1)--(5-2);
	\draw (5-1)--(6-1);
	\draw (6-1)--(7-1);
	\draw (6-1)--(7-2);
	\draw (6-1)--(7-3);
	\draw (5-2)--(6-2);
	\draw (5-2)--(6-3);
	\draw (5-2)--(6-4);
	\node[minimum size=0pt,inner sep=0pt,label=below:$(\mathcal{T}_3)$] (name) at (0,-0.3){};

\draw[dashed,blue](-1.3,1.7) to [in=150, out=120] (-0.7,2.3) ;
\draw[dashed,blue](-1.3,1.7)--(-0.3,0.7);
\draw[dashed,blue] (0.3,1.3)to [in=-30, out=-60] (-0.3,0.7);
	
	\draw[dashed, blue](0.3,1.3)--(-0.7,2.3);
\end{tikzpicture}
\end{eqnarray*}

For the three trees displayed above, each has two typical divisors.
\begin{itemize}
	\item $\mathcal{T}_1$ is effective and the divisor in the blue  dashed circle is its effective divisor and $l$ is its effective leaf.
	\item $\mathcal{T}_2$ is not effective, since the first leaf belongs to a vertex of degree 1, say the root of $\mathcal{T}_2$, which violates Condition (ii) in Definition \ref{Def: effective tree monomials}.
	\item $\mathcal{T}_3$ is not effective since there is a vertex of degree 1 on the path from the root of the typical divisor in the blue  dashed circle to the leftmost leaf above it, which violates Condition (i) in Definition \ref{Def: effective tree monomials}.
\end{itemize}
\end{exam}

\begin{defn}
Let $\mathcal{T}$ be an effective tree monomial in $\RB_\infty^\lambda$ and $\mathcal{T}'$ be its effective divisor. Assume that $\mathcal{T}'=\widehat{\mathcal{S}}$, where $\mathcal{S}$ is a generator of positive degree. Then define $$\overline{\mathfrak{H}}(\mathcal{T})=(-1)^\omega \frac{1}{l_\mathcal{S}}m_{\mathcal{T}', \mathcal{S}}(\mathcal{T}),$$ where $m_{\mathcal{T}',\mathcal{S}}(\mathcal{T})$ is the tree monomial obtained from $\mathcal{T}$ by replacing the effective divisor $\mathcal{T}'$ by $\mathcal{S}$ , $\omega$ is the sum of degrees of all the vertices on the path from root of $\mathcal{T}'$ to the root of $\mathcal{T}$  (except the root vertex of $\mathcal{T}'$) and on the left of this path .
\end{defn}

Then we define a map $\mathfrak{H}$ of degree 1 on $\RB_\infty^\lambda$ as
\begin{itemize}
	\item[(i)] If $\mathcal{T}$ is not an effective tree monomial, then define $\mathfrak{H}(\mathcal{T})=0$;
	\item[(ii)] If $\mathcal{T}$ is effective, denote by $\overline{\mathcal{T}}$ is obtained from $\mathcal{T}$ by replacing $\mathcal{T}'$ by $\mathcal{T}'-\frac{1}{l_\mathcal{S}}\partial \mathcal{S}$ with $\mathcal{T}'$ being the leading term of $\partial \mathcal{S}$.  Define $\mathfrak{H}(\mathcal{T})=\overline{\mathfrak{H}}(\mathcal{T})+\mathfrak{H}(\overline{\mathcal{T}})$, where, since  each tree monomial in  $\overline{\mathcal{T}}$ is strictly smaller than $\mathcal{T}$, define $\mathfrak{H}(\overline{\mathcal{T}})$ by taking induction on leading terms (this can be done by Lemma~\ref{Lem: homotopy well defined}).
	\end{itemize}
Let's explain more on the definition of $\mathfrak{H}$. Denote $\mathcal{T}$ by $\mathcal{T}_1$. By definition above, $\mathfrak{H}(\mathcal{T})=\overline{\mathfrak{H}}(\mathcal{T}_1)+\mathfrak{H}(\overline{\mathcal{T}_1})$. Since $\mathfrak{H}$ vanishes on non-effective tree monomials, we have $\mathfrak{H}(\overline{\mathcal{T}}_1)=\mathfrak{H}(\sum_{i_2\in I_2} \mathcal{T}_{i_2})$ where $\{\mathcal{T}_{i_2}\}_{i_2\in I_2}$ is the set of effective tree monomials together with their coefficients appearing in the expansion of $\overline{\mathcal{T}_1}$. Then by definition of $\mathfrak{H}$, $\mathfrak{H}(\sum_{i_2\in I_2} \mathcal{T}_{i_2})=\overline{\mathfrak{H}}(\sum_{i_2\in I_2} \mathcal{T}_{i_2})+\mathfrak{H}(\sum_{i_2\in I_2}\overline{\mathcal{T}_{i_2}})$, then we have
$$\mathfrak{H}(\mathcal{T})=\overline{\mathfrak{H}}(\mathcal{T}_1)+\overline{\mathfrak{H}}(\sum_{i_2\in I_2} \mathcal{T}_{i_2})+\mathfrak{H}(\sum_{i_2\in I_2}\overline{\mathcal{T}_{i_2}}).$$   Take induction on leading terms,  $\mathfrak{H}(\mathcal{T})$ is the following series:
\begin{eqnarray}
	\label{Eq:Definition-of-homotopy}
	\mathfrak{H}(\mathcal{T})=\overline{\mathfrak{H}}(\mathcal{T}_1)+\overline{\mathfrak{H}}(\sum_{i_2\in I_2} \mathcal{T}_{i_2})+\overline{\mathfrak{H}}(\sum_{i_3\in I_3} \mathcal{T}_{i_3})+\dots+\overline{\mathfrak{H}}(\sum_{i_n\in I_n} \mathcal{T}_{i_n})+\dots,
	\end{eqnarray}
 where $\{\mathcal{T}_{i_n}\}_{i_n\in I_n}$ is the set of  the effective tree monomials with their nonzero coefficients  appearing in the expansion of $\sum_{i_{n-1}\in I_{n-1}} \overline{\mathcal{T}_{i_{n-1}}}$.
	
	\begin{lem}\label{Lem: homotopy well defined}  For any effective tree monomial $\mathcal{T}$, the expansion of $\mathfrak{H}(\mathcal{T})$ in Equation \eqref{Eq:Definition-of-homotopy} is always a finite sum, i.e, there exists some large integer $n$ such that all tree monomials in $\sum_{i_n\in I_n}\overline{\mathcal{T}_{i_n}}$ are not effective.
		\end{lem}
	\begin{proof}	 It is easy to see that $\max\{\mathcal{T}_{i_k}|i_k\in I_k\}>\max\{\mathcal{T}_{i_{k+1}}|i_{k+1}\in I_{k+1}\}$ for all   $k\geq 1$ (by convention, $i_1\in I_1=\{1\}$), so Equation~\eqref{Eq:Definition-of-homotopy} cannot be an infinite sum, as $>$ is a well order.

\end{proof}

\begin{lem}{\label{Lem: Induction}}Let $\mathcal{T}$ be an effective tree monomial. Then $\partial \overline{\mathfrak{H}}(\mathcal{T})+\mathfrak{H}\partial(\mathcal{T}-\overline{\mathcal{T}})=\mathcal{T}-\overline{\mathcal{T}}$.
	\end{lem}

\begin{proof}
	We can write $\mathcal{T}$ as a compositions  in the following way:
	$$(\cdots (((((\cdots(X_1\circ_{i_1}X_2)\circ\cdots )\circ_{i_{p-1}}X_p)\circ_{i_p} \widehat{\mathcal{S}})\circ_{j_1}Y_1)\circ_{j_2}Y_2)\cdots)\circ_{j_q}Y_q,$$
	 where $\widehat{\mathcal{S}}$ is the effective divisor of $\mathcal{T}$ and $X_1,\dots, X_p$ are generators of $\RB_\infty^\lambda$ corresponding to the vertices which live on the path from root of $\mathcal{T}$ and root of $\widehat{\mathcal{S}}$ (except the root of $\widehat{\mathcal{S}}$) and on the left of this path in the underlying  tree of  $\mathcal{T}$.
	
	 By definition,
	 \begin{eqnarray*}
	 	\partial \overline{\mathfrak{H}}(\mathcal{T})&=&
\frac{1}{l_\mathcal{S}}(-1)^{\sum\limits_{j=1}^p|X_j|}\partial
\Big((\cdots ((((\cdots ((X_1\circ_{i_1}X_2)\circ_{i_2}\cdots )\circ_{i_{p-1}}X_p)\circ_{i_p} \mathcal{S})\circ_{j_1}Y_1)\circ_{j_2} \cdots)\circ_{j_q}Y_q\Big)\\
&=&\frac{1}{l_\mathcal{S}}\Big(\sum_{k=1}^p(-1)^{\sum\limits_{j=1}^p|X_j|+\sum\limits_{j=1}^{k-1}|X_j|}\\
	 	&&(\cdots ((((\cdots ((\cdots(X_1\circ_{i_1} X_2)\circ_{i_2} \cdots ) \circ_{i_{k-1}}\partial X_k )\circ_{i_k} \cdots)\circ_{i_{p-1}}X_p)\circ_{i_p} \mathcal{S})\circ_{j_1}Y_1)\circ_{j_2} \cdots)\circ_{j_q}Y_q\Big)\\
	 	&&+\frac{1}{l_\mathcal{S}}\Big((\cdots ((( (\cdots\cdots(X_1\circ_{i_1}X_2)\circ_{i_2}\cdots )\circ_{i_{p-1}}X_p)\circ_{i_p} \partial \mathcal{S})\circ_{j_1}Y_1)\circ_{j_2} \cdots)\circ_{j_q}Y_q\Big)\\
	 	&&+\frac{1}{l_\mathcal{S}}\Big(\sum_{k=1}^q(-1)^{|\mathcal{S}|+\sum\limits_{j=1}^{k-1}|Y_j|}\\
	 	&&(\cdots ((\cdots((((\cdots(X_1\circ_{i_1}X_2)\circ\cdots )\circ_{i_{p-1}}X_p)\circ_{i_p} \mathcal{S})\circ_{j_1}Y_1)\circ_{j_2} \cdots ) \circ_{j_k}\partial Y_{k})\circ_{j_{k+1}}\cdots )\circ_{j_q}Y_q\Big)
	 	\end{eqnarray*}
 	 and
 	\begin{eqnarray*}
 	 	&&\mathfrak{H}\partial (\mathcal{T}-\overline{\mathcal{T}})\\
 	 	&=&\frac{1}{l_\mathcal{S}}\mathfrak{H}\partial \Big((\cdots((((\cdots(X_1\circ_{i_1}X_2)\circ_{i_2}\cdots )\circ_{i_{p-1}}X_p)\circ_{i_p} \partial \mathcal{S})\circ_{j_1}Y_1)\circ_{j_2} \cdots)\circ_{j_q}Y_q\Big)\\
 	 	&=&\frac{1}{l_\mathcal{S}}\mathfrak{H}\Big(\sum\limits_{k=1}^p(-1)^{\sum\limits_{j=1}^{k-1}|X_j|}
 (\cdots((((\cdots  ((\cdots(X_1\circ_{i_1} \cdots )\circ_{i_{k-1}}\partial X_k)\circ_{i_k}\cdots)\circ_{i_{p-1}}X_p)\circ_{i_p} \partial \mathcal{S})\circ_{j_1}Y_1)\circ_{j_2} \cdots)\circ_{j_q}Y_q\Big)\\
 	 	&&+\frac{1}{l_\mathcal{S}}\mathfrak{H}\Big(\sum_{k=1}^q(-1)^{\sum\limits_{j=1}^p|X_j|+|\mathcal{S}|-1+\sum_{j=1}^{k-1}|Y_j|}\\
 	 	&&(\cdots((\cdots ((((\cdots(X_1\circ_{i_1}X_2)\circ_{i_2}\cdots )\circ_{i_{p-1}}X_p)\circ_{i_p} \partial \mathcal{S})\circ_{j_1}Y_1)\circ_{j_2} \cdots)\circ_{j_k}\partial Y_k)\circ_{j_{k+1}}\cdots\circ_{j_q}Y_q\Big).
 	 	\end{eqnarray*}
	
By the definition of the effective divisor in an effective tree monomial, it can be easily seen that each tree monomial in the expansion of
 \begin{align*} (\cdots((((\cdots  ((\cdots(X_1\circ_{i_1} \cdots )\circ_{i_{k-1}}\partial X_k)\circ_{i_k}\cdots)\circ_{i_{p-1}}X_p)\circ_{i_p} \widehat{\mathcal{S}})\circ_{j_1}Y_1)\circ_{j_2} \cdots)\circ_{j_q}Y_q\end{align*}
 and of
 	\begin{align*}(\cdots((\cdots ((((\cdots(X_1\circ_{i_1}X_2)\circ_{i_2}\cdots )\circ_{i_{p-1}}X_p)\circ_{i_p}  \widehat{\mathcal{S}})\circ_{j_1}Y_1)\circ_{j_2} \cdots)\circ_{j_k}\partial Y_k)\circ_{j_{k+1}}\cdots\circ_{j_q}Y_q
 \end{align*}
 is still effective tree monomial whose  effective divisor is still $\widehat{\mathcal{S}}$.
 Thus we have
 \begin{eqnarray*}
 	&&\mathfrak{H}\partial(\mathcal{T}-\overline{\mathcal{T}})\\
 	 	&=&\frac{1}{l_S}\Big(\sum\limits_{k=1}^p(-1)^{\sum\limits_{j=1}^{k-1}|X_j|+\sum\limits_{j=1}^p|X_j|-1}\\
 &&(\cdots((((\cdots  ((\cdots(X_1\circ_{i_1} \cdots )\circ_{i_{k-1}}\partial X_k)\circ_{i_k}\cdots)\circ_{i_{p-1}}X_p)\circ_{i_p} \mathcal{S})\circ_{j_1}Y_1)\circ_{j_2} \cdots)\circ_{j_q}Y_q\Big)\\
 	 	&&+ \frac{1}{l_S}\Big(\sum_{k=1}^q(-1)^{|S|-1+\sum\limits_{j=1}^{k-1}|Y_j|}\\
 	 	&&(\cdots((\cdots ((((\cdots(X_1\circ_{i_1}X_2)\circ_{i_2}\cdots )\circ_{i_{p-1}}X_p)\circ_{i_p}   \mathcal{S})\circ_{j_1}Y_1)\circ_{j_2} \cdots)\circ_{j_k}\partial Y_k)\circ_{j_{k+1}}\cdots\circ_{j_q}Y_q\Big).
 	\end{eqnarray*}
Take sum of the above expansion, then we get $\partial \overline{\mathfrak{H}}(\mathcal{T})+\mathfrak{H}\partial(\mathcal{T}-\overline{\mathcal{T}})=\mathcal{T}-\overline{\mathcal{T}}$.
	\end{proof}

\begin{prop}The degree $1$ map $\mathfrak{H}$ defined above satisfies $\partial \mathfrak{H}+\mathfrak{H}\partial=\mathrm{Id}$ in all  positive degrees of $\RB_\infty^\lambda$.
	\end{prop}
\begin{proof}
Let $\mathcal{T}$  be an effective tree monomial.  Since the leading term of  $\overline{\mathcal{T}}$ is strictly smaller than $\mathcal{T}$, by induction, we have $$\mathfrak{H}\partial(\overline{\mathcal{T}})+\partial \mathfrak{H}(\overline{\mathcal{T}})=\overline{\mathcal{T}}.$$
	By the definition of $\mathfrak{H}$, $\mathfrak{H}(\mathcal{T})=\overline{\mathfrak{H}}(\mathcal{T})+\mathfrak{H}(\overline{\mathcal{T}})$ and we have $\partial \mathfrak{H}(\mathcal{T})=\partial \overline{\mathfrak{H}}(\mathcal{T})+\partial \mathfrak{H}(\overline{\mathcal{T}})$. Thus, $$\begin{array}{rcl} \partial \mathfrak{H}(\mathcal{T})+\mathfrak{H}\partial (\mathcal{T})&=&\partial \overline{\mathfrak{H}}(\mathcal{T})+\partial \mathfrak{H}(\overline{\mathcal{T}})+\mathfrak{H}\partial(\mathcal{T}-\overline{\mathcal{T}})+\mathfrak{H}\partial(\overline{\mathcal{T}})\\
  &=&\partial \overline{\mathfrak{H}}(\mathcal{T})+\mathfrak{H}\partial(\mathcal{T}-\overline{\mathcal{T}})+\partial \mathfrak{H}(\overline{\mathcal{T}})+\mathfrak{H}\partial(\overline{\mathcal{T}})\\
  &=&\mathcal{T}-\overline{\mathcal{T}}+\overline{\mathcal{T}}\\
  &=&\mathcal{T},\end{array}$$
  where in the third equality we have used the induction hypothesis and $$\partial \overline{\mathfrak{H}}(\mathcal{T})+\mathfrak{H}\partial(\mathcal{T}-\overline{\mathcal{T}})=\mathcal{T}-\overline{\mathcal{T}}$$ by Lemma \ref{Lem: Induction}.

  \medskip

	Next let's prove that for a non-effective tree monomial $\mathcal{T}$, the equation $\partial \mathfrak{H}(\mathcal{T})+\mathfrak{H}\partial(\mathcal{T})=\mathcal{T}$ holds.

By the definition of $\mathfrak{H}$, since $\mathcal{T}$ is not effective, $\mathfrak{H}(\mathcal{T})=0$, thus we just need to check that $\mathfrak{H}\partial(\mathcal{T})=\mathcal{T}$. Since $\mathcal{T}$ has   positive degree, there must exists at least one vertex of positive degree. Let's pick a special vertex $\mathcal{S}$ satisfying  the following conditions:
	\begin{itemize}
		\item[(i)] on the path from $\mathcal{S}$ to the leftmost leaf $l$ of $\mathcal{T}$ above $\mathcal{S}$, there are no other vertices of positive degree;
		\item[(ii)] for any leaf $l'$ of $\mathcal{T}$ located on the left of $l$, the vertices  on the path from the root of $\mathcal{T}$ to $l'$ are all of degree 0.
	\end{itemize}
It is easy to see such a vertex always exists in $\mathcal{T}$. Morally, this vertex is the ``left-upper-most" vertex of positive degree.
	Then the tree monomial  $\mathcal{T}$ can be written as
	$$(\cdots((((\cdots(X_1\circ_{i_1}X_2)\circ_{i_2}\cdots )\circ_{i_{p-1}}X_p)\circ_{i_p} \mathcal{S})\circ_{j_1}Y_1)\circ_{j_2} \cdots)\circ_{j_q}Y_q,$$
	where $X_1,\dots,X_p$ corresponds to the vertices located on the path from the root of $\mathcal{T}$ to $\mathcal{S}$ and on the left of this path in the plane.
	
	By definition,
	\begin{align*}
		&\mathfrak{H}\partial \mathcal{T}\\
		=&\mathfrak{H}\ \ {\Huge\{}\\
		&\sum_{k=1}^p(-1)^{\sum\limits_{t=1}^{k-1}|X_t|}
(\cdots(( ((\cdots((\cdots(X_1\circ_{i_1} \cdots)\circ_{i_k}\partial X_k)\circ_{i_{k+1}}\cdots )\circ_{i_{p-1}}X_p)\circ_{i_p} \mathcal{S})\circ_{j_1}Y_1)\circ_{j_2} \cdots)\circ_{j_q}Y_q\\
		&+(-1)^{\sum_{t=1}^p|X_t|}(\cdots((((\cdots(X_1\circ_{i_1}X_2)\circ\cdots )\circ_{i_{p-1}}X_p)\circ_{i_p} \partial \mathcal{S})\circ_{j_1}Y_1)\circ_{j_2} \cdots)\circ_{j_q}Y_q\\
		&\sum_{k=1}^q(-1)^{\sum\limits_{t=1}^p|X_t|+|\mathcal{S}|+\sum\limits_{t=1}^{k-1}|Y_t|} (\cdots(\cdots((((\cdots(X_1\circ_{i_1}X_2)\circ_{i_2}\cdots )\circ_{i_{p-1}}X_p)\circ_{i_p} \mathcal{S})\circ_{j_1}Y_1)\circ_{j_2} \cdots)\circ_{j_k}\partial Y_{k})\cdots\circ_{j_q}Y_q\\
		&{\Huge \}}
	\end{align*}
	
	By the assumption, the divisor consisting of the path from $\mathcal{S}$ to $l$ must be of the following forms
	
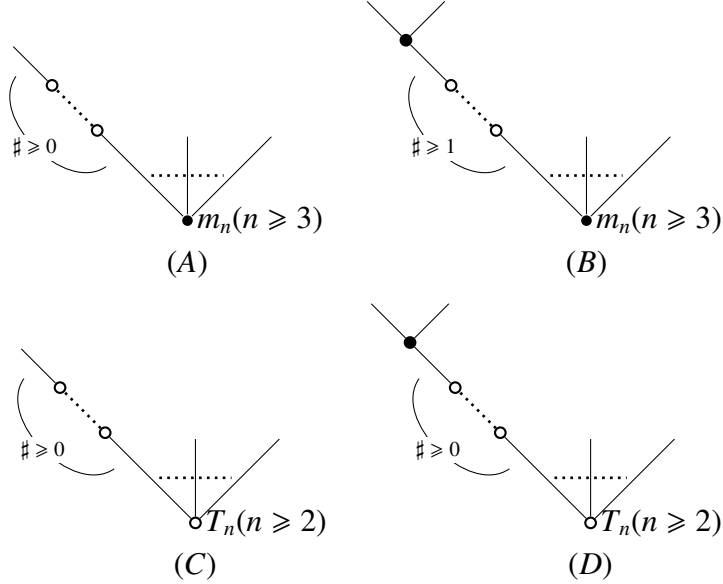
\begin{eqnarray*}
		\begin{tikzpicture}[scale=0.6,descr/.style={fill=white}]
			\tikzstyle{every node}=[thick,minimum size=4pt, inner sep=1pt]
			\node(r) at(0,-0.5)[minimum size=0pt, label=below:$(A)$]{};
			\node(0) at(0,0)[circle, fill=black, label=right:$m_n(n\geqslant 3)$]{};
			\node(1-1) at(-2,2)[circle, draw]{};
			\node(1-2) at(0,2){};
			\node(1-3) at (2,2){};
			\node(2-1) at(-3,3)[circle,draw]{};
			\node(3-1) at(-4,4){};
			\draw(0)--(1-1);
			\draw(0)--(1-2);
			\draw(0)--(1-3);
			\draw[dotted, line width=1pt](-0.8,1)--(0.8,1);
			\draw[dotted, line width=1pt](1-1)--(2-1);
			\draw(2-1)--(3-1);
			\path[-,font=\scriptsize]
			(-1.8,1.2) edge [bend left=80] node[descr]{{\tiny$\sharp\geqslant 0$}} (-3.8,3.2);
			
		\end{tikzpicture}
		\hspace{4mm}
		\begin{tikzpicture}[scale=0.6,descr/.style={fill=white}]
			\tikzstyle{every node}=[thick,minimum size=4pt, inner sep=1pt]
			\node(r) at(0,-0.5)[minimum size=0pt, label=below:$(B)$]{};
			\node(0) at(0,0)[circle, fill=black, label=right:$m_n(n\geqslant 3)$]{};
			\node(1-1) at(-2,2)[circle, draw]{};
			\node(1-2) at(0,2){};
			\node(1-3) at (2,2){};
			\node(2-1) at(-3,3)[circle,draw]{};
			\node(3-1) at(-4,4)[circle, draw, fill=black]{};
			\node(4-1) at (-5,5){};
			\node(4-2) at(-3,5){};
			\draw(0)--(1-1);
			\draw(0)--(1-2);
			\draw(0)--(1-3);
			\draw[dotted, line width=1pt](-0.8,1)--(0.8,1);
			\draw[dotted, line width=1pt](1-1)--(2-1);
			\draw(2-1)--(3-1);
			\draw(3-1)--(4-1);
			\draw(3-1)--(4-2);
			\path[-,font=\scriptsize]
			(-1.8,1.2) edge [bend left=80] node[descr]{{\tiny$\sharp\geqslant 1$}} (-3.8,3.2);
		\end{tikzpicture}\\
		\begin{tikzpicture}[scale=0.6,descr/.style={fill=white}]
			\tikzstyle{every node}=[thick,minimum size=4pt, inner sep=1pt]
			\node(r) at(0,-0.5)[minimum size=0pt, label=below:$(C)$]{};
			\node(0) at(0,0)[circle, draw, label=right:$T_n(n\geqslant 2)$]{};
			\node(1-1) at(-2,2)[circle, draw]{};
			\node(1-2) at(0,2){};
			\node(1-3) at (2,2){};
			\node(2-1) at(-3,3)[circle,draw]{};
			\node(3-1) at(-4,4){};
			\draw(0)--(1-1);
			\draw(0)--(1-2);
			\draw(0)--(1-3);
			\draw[dotted, line width=1pt](-0.8,1)--(0.8,1);
			\draw[dotted, line width=1pt](1-1)--(2-1);
			\draw(2-1)--(3-1);
			\path[-,font=\scriptsize]
			(-1.8,1.2) edge [bend left=80] node[descr]{{\tiny$\sharp\geqslant 0$}} (-3.8,3.2);
		\end{tikzpicture}
		\hspace{4mm}
		\begin{tikzpicture}[scale=0.6,descr/.style={fill=white}]
			\tikzstyle{every node}=[thick,minimum size=4pt, inner sep=1pt]
			\node(r) at(0,-0.5)[minimum size=0pt, label=below:$(D)$]{};
			\node(0) at(0,0)[circle, draw, label=right:$T_n(n\geqslant 2)$]{};
			\node(1-1) at(-2,2)[circle, draw]{};
			\node(1-2) at(0,2){};
			\node(1-3) at (2,2){};
			\node(2-1) at(-3,3)[circle,draw]{};
			\node(3-1) at(-4,4)[circle, draw, fill=black]{};
			\node(4-1) at (-5,5){};
			\node(4-2) at(-3,5){};
			\draw(0)--(1-1);
			\draw(0)--(1-2);
			\draw(0)--(1-3);
			\draw[dotted, line width=1](-0.8,1)--(0.8,1);
			\draw[dotted, line width=1pt](1-1)--(2-1);
			\draw(2-1)--(3-1);
			\draw(3-1)--(4-1);
			\draw(3-1)--(4-2);
			\path[-,font=\scriptsize]
			(-1.8,1.2) edge [bend left=80] node[descr]{{\tiny$\sharp\geqslant 0$}} (-3.8,3.2);
		\end{tikzpicture}
	\end{eqnarray*}
	By the assumption that $\mathcal{T}$ is not effective and the speciality of the position of $\mathcal{S}$, one can see that the effective tree monomials in $\partial \mathcal{T}$ will only appear in the expansion of
	$$(-1)^{\sum_{t=1}^p |X_t| }(\cdots((((\cdots(X_1\circ_{i_1}X_2)\circ_{i_2}\cdots )\circ_{i_{p-1}}X_p)\circ_{i_p} \partial \mathcal{S})\circ_{j_1}Y_1)\circ_{j_2}  \cdots)\circ_{j_q}Y_q.$$
	
	Consider the tree monomial $$(\cdots((((\cdots(X_1\circ_{i_1}X_2)\circ_{i_2}\cdots )\circ_{i_{p-1}}X_p)\circ_{i_p} \widehat{ \mathcal{S}})\circ_{j_1}Y_1)\circ_{j_2} \cdots)\circ_{j_q}Y_q$$ in $\partial \mathcal{T}$. The the path connecting root of $\widehat{\mathcal{S}}$ and $l$ must be of the following form:
	
		
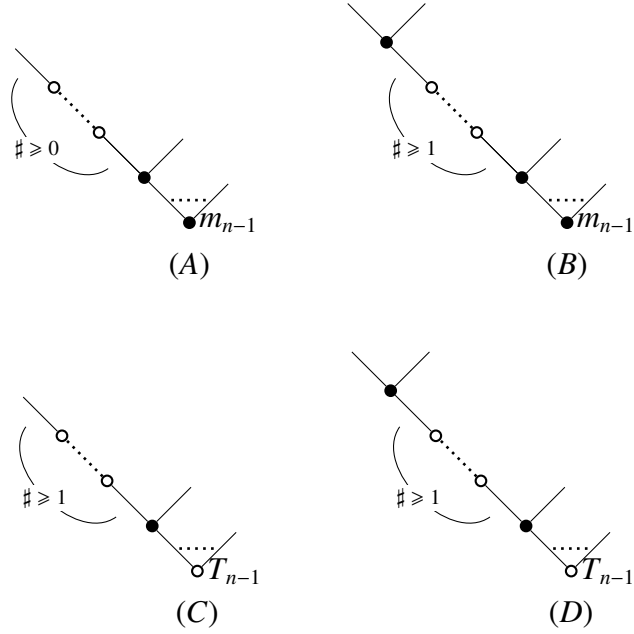
\begin{eqnarray*}
		\begin{tikzpicture}[scale=0.6,descr/.style={fill=white}]
			\tikzstyle{every node}=[thick,minimum size=4pt, inner sep=1pt]
			\node(r) at(0,-0.5)[minimum size=0pt,label=below:$(A)$]{};
			\node(1) at (0,0)[circle,draw,fill=black,label=right:$m_{n-1}$]{};
			\node(2-1) at(-1,1)[circle,draw,fill=black]{};
			\node(2-2) at (1,1) {};
			\node(3-2) at(0,2){};
			\node(3-1) at (-2,2)[circle,draw]{};
			\node(4-1) at (-3,3)[circle,draw]{};
			\node(5-1) at (-4,4){};
			\draw(1)--(2-1);
			\draw(1)--(2-2);
			\draw(2-1)--(3-1);
			\draw(2-1)--(3-2);
			\draw [dotted,line width=1pt](-0.4,0.5)--(0.4,0.5);
			\draw(2-1)--(3-1);
			\draw[dotted,line width=1pt] (3-1)--(4-1);
			\draw(4-1)--(5-1);
			\path[-,font=\scriptsize]
			(-1.8,1.2) edge [bend left=80] node[descr]{{\tiny$\sharp\geqslant 0$}} (-3.8,3.2);
		\end{tikzpicture}
		\hspace{10mm}
		\begin{tikzpicture}[scale=0.6,descr/.style={fill=white}]
			\tikzstyle{every node}=[thick,minimum size=4pt, inner sep=1pt]
			\node(r) at(0,-0.5)[minimum size=0pt,label=below:$(B)$]{};
			\node(1) at (0,0)[circle,draw,fill=black,label=right:$m_{n-1}$]{};
			\node(2-1) at(-1,1)[circle,draw,fill=black]{};
			\node(2-2) at (1,1) {};
			\node(3-2) at(0,2){};
			\node(3-1) at (-2,2)[circle,draw]{};
			\node(4-1) at (-3,3)[circle,draw]{};
			\node(5-1) at (-4,4)[circle,draw,fill=black]{};
			\node(6-1) at(-5,5){};
			\node(6-2) at (-3,5){};
			\draw(1)--(2-1);
			\draw(1)--(2-2);
			\draw(2-1)--(3-1);
			\draw(2-1)--(3-2);
			\draw [dotted,line width=1pt](-0.4,0.5)--(0.4,0.5);
			\draw(2-1)--(3-1);
			\draw[dotted,line width=1pt] (3-1)--(4-1);
			\draw(4-1)--(5-1);
			\draw(5-1)--(6-1);
			\draw(5-1)--(6-2);
			\path[-,font=\scriptsize]
			(-1.8,1.2) edge [bend left=80] node[descr]{{\tiny$\sharp\geqslant 1$}} (-3.8,3.2);
		\end{tikzpicture}\\
		\\
		\begin{tikzpicture}[scale=0.6,descr/.style={fill=white}]
			\tikzstyle{every node}=[thick,minimum size=4pt, inner sep=1pt]
			\node(r) at(0,-0.5)[minimum size=0pt,label=below:$(C)$]{};
			\node(1) at (0,0)[circle,draw,label=right:$T_{n-1}$]{};
			\node(2-1) at(-1,1)[circle,draw,fill=black]{};
			\node(2-2) at (1,1) {};
			\node(3-1) at(-2,2)[circle,draw]{};
			\node(3-2) at(0,2){};
			\node(4-1) at(-3,3)[circle,draw]{};
			\node(5-1) at (-4,4){};
			\draw(1)--(2-1);
			\draw(1)--(2-2);
			\draw(2-1)--(3-1);
			\draw(2-1)--(3-2);
			\draw[dotted,line width=1pt](3-1)--(4-1);
			\draw(4-1)--(5-1);
			\draw [dotted,line width=1pt](-0.4,0.5)--(0.4,0.5);
			\path[-,font=\scriptsize]
			(-1.8,1.2) edge [bend left=80] node[descr]{{\tiny$\sharp\geqslant 1$}} (-3.8,3.2);
		\end{tikzpicture}
		\hspace{10mm}
		\begin{tikzpicture}[scale=0.6,descr/.style={fill=white}]
			\tikzstyle{every node}=[thick,minimum size=4pt, inner sep=1pt]
			\node(r) at(0,-0.5)[minimum size=0pt,label=below:$(D)$]{};
			\node(1) at (0,0)[circle,draw,label=right:$T_{n-1}$]{};
			\node(2-1) at(-1,1)[circle,draw,fill=black]{};
			\node(2-2) at (1,1) {};
			\node(3-1) at(-2,2)[circle,draw]{};
			\node(3-2) at(0,2){};
			\node(4-1) at(-3,3)[circle,draw]{};
			\node(5-1) at (-4,4)[circle,draw,fill=black]{};
			\node(6-1) at(-5,5){};
			\node(6-2) at(-3,5){};
			\draw(1)--(2-1);
			\draw(1)--(2-2);
			\draw(2-1)--(3-1);
			\draw[dotted,line width=1pt](3-1)--(4-1);
			\draw(4-1)--(5-1);
			\draw(5-1)--(6-1);
			\draw(5-1)--(6-2);
			\draw(2-1)--(3-2);
			\draw [dotted,line width=1pt](-0.4,0.5)--(0.4,0.5);
			\path[-,font=\scriptsize]
			(-1.8,1.2) edge [bend left=80] node[descr]{{\tiny$\sharp\geqslant 1$}} (-3.8,3.2);
		\end{tikzpicture}
	\end{eqnarray*}
	So the tree monomial $$(\cdots((((\cdots(X_1\circ_{i_1}X_2)\circ_{i_2}\cdots )\circ_{i_{p-1}}X_p)\circ_{i_p} \widehat{ \mathcal{S}})\circ_{j_1}Y_1)\circ_{j_2} \cdots)\circ_{j_q}Y_q$$ is effective and its effective divisor is exactly $\widehat{\mathcal{S}}$ itself.
	Then we have
	\begin{eqnarray*}&&\mathfrak{H}\partial \mathcal{T}\\
		&=&\mathfrak{H}((-1)^{\sum_{t=1}^p |X_t|}(\cdots((((\cdots(X_1\circ_{i_1}X_2)\circ\cdots )\circ_{i_{p-1}}X_p)\circ_{i_p} \partial \mathcal{S})\circ_{j_1}Y_1)\circ_{j_2}Y_2\cdots)\circ_{j_q}Y_q)\\
		&=&l_\mathcal{S}\mathfrak{H}\Big((-1)^{\sum_{t=1}^p |X_t|}(\cdots((((\cdots(X_1\circ_{i_1}X_2)\circ_{i_2}\cdots )\circ_{i_{p-1}}X_p)\circ_{i_p} \widehat{ \mathcal{S}})\circ_{j_1}Y_1)\circ_{j_2} \cdots)\circ_{j_q}Y_q\Big)\\
		&&+\mathfrak{H}\Big((-1)^{\sum_{t=1}^p |X_t|}(\cdots((((\cdots(X_1\circ_{i_1}X_2)\circ_{i_2}\cdots )\circ_{i_{p-1}}X_p)\circ_{i_p} (\partial \mathcal{S}-l_\mathcal{S}\widehat{\mathcal{S}}))\circ_{j_1}Y_1)\circ_{j_2} \cdots)\circ_{j_q}Y_q\Big)\\
		&=& l_\mathcal{S}\overline{\mathfrak{H}}\Big((-1)^{\sum_{t=1}^p |X_t|}(\cdots((((\cdots(X_1\circ_{i_1}X_2)\circ_{i_2}\cdots )\circ_{i_{p-1}}X_p)\circ_{i_p} \widehat{ \mathcal{S}})\circ_{j_1}Y_1)\circ_{j_2} \cdots)\circ_{j_q}Y_q\Big)\\
		&&+l_\mathcal{S}\mathfrak{H}\Big((-1)^{\sum_{t=1}^p |X_t|}(\cdots((((\cdots(X_1\circ_{i_1}X_2)\circ_{i_2}\cdots )\circ_{i_{p-1}}X_p)\circ_{i_p}  (\widehat{\mathcal{S}}-\frac{1}{l_\mathcal{S}}\partial \mathcal{S}))\circ_{j_1}Y_1)\circ_{j_2} \cdots)\circ_{j_q}Y_q\Big)\\
		&&+\mathfrak{H}\Big((-1)^{\sum_{t=1}^p |X_t|}(\cdots((((\cdots(X_1\circ_{i_1}X_2)\circ_{i_2}\cdots )\circ_{i_{p-1}}X_p)\circ_{i_p} (\partial \mathcal{S}-l_\mathcal{S}\widehat{\mathcal{S}}))\circ_{j_1}Y_1)\circ_{j_2} \cdots)\circ_{j_q}Y_q\Big)\\
		&=& l_\mathcal{S}\overline{\mathfrak{H}}\Big((-1)^{\sum_{t=1}^p |X_t|}(\cdots((((\cdots(X_1\circ_{i_1}X_2)\circ_{i_2}\cdots )\circ_{i_{p-1}}X_p)\circ_{i_p} \widehat{ \mathcal{S}})\circ_{j_1}Y_1)\circ_{j_2} \cdots)\circ_{j_q}Y_q\Big)\\
		&=&(\cdots((((\cdots(X_1\circ_{i_1}X_2)\circ_{i_2}\cdots )\circ_{i_{p-1}}X_p)\circ_{i_p} \mathcal{S})\circ_{j_1}Y_1)\circ_{j_2} \cdots)\circ_{j_q}Y_q\\
		&=& \mathcal{T} .
	\end{eqnarray*}
	
	This completes the proof.
  \end{proof}

\textbf{Proof of Theorem~\ref{Thm: Minimal model}:}

We have proved that the natural morphism $\phi: \RB_\infty^\lambda\twoheadrightarrow \RB^\lambda$ is a surjective quasi-isomorphism, and  it can be easily seen that the differential $\partial$ on $\RB_\infty^\lambda$ satisfies the conditions $(1)(2)$ in Definition~\ref{Def: Minimal model of operads}.

\begin{remark}\label{rem: dotsenko}We are grateful to  Dotsenko who kindly pointed out an alternative proof of Theorem~\ref{Thm: Minimal model}.

Let $\mathfrak{Tr}$ be the free operad generated by a unary operation $T$ and a binary operation $\mu.$ Then $\RB\cong \mathfrak{Tr}/<\tilde{G}>$ where $\tilde{G}$ is the defining relations for Rota-Baxter algebras. In \cite{DK13}, the authors prove that $\tilde{G}$ is a Gr\"obner-Shirshov basis and they constructed the minimal model for the operad $\mathfrak{Tr}/<G>$ where $G$ is the set of leading monomials in $\tilde{G}$. Denote this minimal model by ${}_m\RB_\infty^\lambda$.
	
	Now, let's introducing a new ordering $\prec$ on the the set $\mathfrak{Tr}(n)$ as follows: For two tree monomials $\mathcal{T, T'}$ in $\mathfrak{Tr}(n)$,
	\begin{itemize}
		\item[(1)]If $\mbox{weight}(\mathcal{T})<\mbox{weight}( \mathcal{T'})$, then $\mathcal{T}\prec\mathcal{T'}$.
		\item[(2)] If $\mbox{weight}(\mathcal{T})=\mbox{weight}( \mathcal{T'})$, compare $\mathcal{T}$ with $\mathcal{T'}$ via the natural path-lexicographic induced by setting $T\prec \mu.$
	\end{itemize}
With respect to the ordering $\prec$, the set $\mathfrak{Tr}(n)$ becomes a totally ordered set: $\mathfrak{Tr}(n)=\{x_1\prec x_2\prec x_3\prec\dots\}\cong \mathbb{N}^+$. Now, we define a map $\omega:\RB_\infty^\lambda\rightarrow \mathfrak{Tr}$ by replacing vertices $m_n (n\geqslant 2), T_n(n\geqslant 1)$ by $\underbrace{\mu\circ_1\mu\circ_1\mu\dots\circ_1\mu}_{n-1}$ and $ \underbrace{T\circ_1\mu\circ_1T\circ_1\mu\circ\dots\circ_1T\circ_1\mu\circ_1T}_{2n-1}$ in $\mathfrak{Tr}$ respectively. Notice that different tree monomials in $\RB_\infty^\lambda$ may have same image in $\mathfrak{Tr}$ under the action $\omega$.
 Define $\mathcal{F}^n_i $  to be the subspace of $\RB_\infty^\lambda(n)$ spanned by the tree monomials $\mathcal{R}$ with $\omega(\mathcal{R})$ smaller than or equal to $x_i$ with respect to  the ordering $\prec$ on $\mathfrak{Tr}(n)$.
Then we get a bounded below and exhausitive filtration for $\RB_\infty^\lambda(n)$ \[0=\mathcal{F}^n_0\subset \mathcal{F}^n_1\subset\mathcal{F}^n_2\subset \dots .\]
It  can be easily seen that the filtration in compatible with the differential $\partial$ on $\RB_\infty^\lambda(n)$. Morever, one can prove that there is an isomorphism of complexes
$$\bigoplus_{i\geq0}\mathcal{F}^n_{i+1}/\mathcal{F}^n_i\cong {}_m\RB_\infty^\lambda(n).$$
Since all positive homologies of ${}_m\RB_\infty^\lambda(n)$ vanishes, by classical spectral sequence argument, we have that all positive homologies of $\RB_\infty^\lambda(n)$ are trvial. 

This provides another proof for   Theorem \ref{Thm: Minimal model}.
	\end{remark}

\section*{Appendix A: Proof of Theorem~\ref{Thm: rb-L-infty}}

In this appendix, we will  prove Theorem~\ref{Thm: rb-L-infty}.

%

{\noindent{{\bf{Theorem~\ref{Thm: rb-L-infty}.}}} Given a graded space $V$  and an element $\lambda\in \bfk$, the graded space ${\mathfrak{C}_{\RBA}}(V)$ endowed with operations $\{l_n\}_{n\geqslant 1}$ defined above forms an $L_\infty$-algebra.
}

\bigskip

 By the definition of $L_\infty$-algebras, we need to check that the operators $\{l_n\}_{n\geqslant 1}$ on  $\mathfrak{C}_{\RBA}(V)$ satisfy the generalised anti-symmetry and the generalised Jacobi identity in Definition \ref{Def:L-infty}.
	
	  The operators  $\{l_n\}_{n\geqslant 1}$  is automatically  anti-symmetric by construction.

\bigskip

	Now, we are going to check that $\{l_n\}_{n\geqslant 1}$ satisfies the generalised Jacobi identity, i.e, the following equation:
	\begin{eqnarray}\label{Eq:jacobi}\quad \\
 \sum\limits_{i=1}^m\sum\limits_{\sigma\in \Sh(i,m-i)}\chi(\sigma; x_1,\dots,x_m)(-1)^{i(m-i)}l_{m-i+1}(l_i(x_{\sigma(1)}\ot \cdots \ot x_{\sigma(i)})\ot x_{\sigma(i+1)}\ot \cdots \ot x_{\sigma(m)})=0\nonumber
	\end{eqnarray}
	for any $x_1,\dots,x_m\in\mathfrak{C}_{\RBA}(V), m\geqslant 1$.

\medskip

By Remark~\ref{Rem: L-infinity for small n} (i), when $m=1$, Equation~(\ref{Eq:jacobi}) holds by definition of $l_1$;
	when $m=3$ and all $ x_1,x_2,x_3 \in \mathfrak{C}_{\Alg}(V)$, the LHS of  Remark~\ref{Rem: L-infinity for small n}(iii),   is just the usual  Jabobi identity of the graded Lie algebra $\mathfrak{C}_{\Alg}(V)$ endowed with Gerstenhaber Lie bracket and the RHS of Remark~\ref{Rem: L-infinity for small n}(iii) vanishes because $l_1$   sends an element of $\mathfrak{C}_{\Alg}(V)$ to zero or an element of $\mathfrak{C}_{\RBO}(V)$ in which case  we have two elements of $\mathfrak{C}_{\Alg}(V)$ and one in $\mathfrak{C}_{\RBO}(V)$, then $l_3$ applied to these three elements would give zero.

We have seen that one only  needs to check  Equation~(\ref{Eq:jacobi}) with some $x_i\in \mathfrak{C}_{\RBO}(V)$.

	\medskip

	By definition, $l_m(x_1,\dots,x_m)=0$ when $x_1,\dots,x_m$ are all contained in $ \mathfrak{C}_{\RBO}(V)$. So we have  $$l_{m-i+1}(l_i(x_{\sigma(1)}\ot \cdots \ot x_{\sigma(i)})\ot x_{\sigma(i+1)}\ot \cdots \ot x_{\sigma(m)})=0$$ unless there are exactly two elements in $\{x_1,\dots,x_m\}$ belonging to $\mathfrak{C}_{\Alg}(V)$.
Write $m=n+2$ and assume that $$ x_1=sh_1\in \Hom((sV)^{\ot n_1},sV)\subset \mathfrak{C}_{\Alg}(V), x_2=sh_2\in \Hom((sV)^{\ot n_2},sV)\subset \mathfrak{C}_{\Alg}(V)$$ and $$x_3=g_1, \dots,  x_{n+2}=g_n\in   \mathfrak{C}_{\RBO}(V).$$
Since the expansion of the Equation~(\ref{Eq:jacobi}) depends on the integers $n_1,n_2,n$, we classify the following cases:
\begin{itemize}
	\item[(I)] $n<\min\{n_1,n_2\}$,
	\item[(II)] $\min\{n_1,n_2\}\leqslant n<\max\{n_1,n_2\}$,
	\item[(III)] $\max\{n_1,n_2\}\leqslant n<n_1+n_2-1$,
	\item[(IV)] $n=n_1+n_2-1$.
\end{itemize}

\medskip

 Now, we   check  Equation~(\ref{Eq:jacobi}) for Case $(\mathrm{I})$.

Assume first that $n_1, n_2\geqslant 1$.
Given $\sigma\in S_m$ and graded elements $g_1,\dots,g_m$, denote $\varepsilon_i^\sigma=\sum\limits_{j=1}^i|g_{\sigma(j)}|$ and $\eta_i^m=\sum\limits_{j=1}^i(|g_{\sigma(j)}|+1)$.
	The expansion of  Equation~(\ref{Eq:jacobi}) contains the following three parts:

\smallskip
	
\quad	\noindent $\mathrm{(1)}$ The terms that $sh_1,sh_2$ are both contained in $l_i(\dots)$  :
	\begin{align*}
		{{A}}=&l_{n+1}(l_2(sh_1\ot sh_2)\ot g_1\ot \cdots \ot g_n)\\
		\stackrel{(I)}{=}&l_{n+1}([sh_1,sh_2]_G\ot g_1\ot \cdots \ot g_n)\\
		\stackrel{(\ref{Eq: Gerstahaber bracket})}{=}&l_{n+1}(sh_1\{ sh_2\}\ot g_1\ot \cdots \ot g_n)+(-1)^{1+(|h_1|+1)(|h_2|+1)}l_{n+1}(sh_2\{ sh_1\}\ot g_1\ot \cdots \ot g_n)\\
		\stackrel{(III)(ii)}{=}&\ \ \  \sum_{\delta\in S_n}(-1)^{\alpha_1}\lambda^{n_1+n_2-1-n}g_{\delta(1)}\Big\{\big(sh_1\{ sh_2\}\big)\big\{ sg_{\delta(2)},\dots,sg_{\delta(n)}\big\}\Big\}\\
		&+\sum_{\delta\in S_n}(-1)^{\alpha_1}(-1)^{1+(|h_1|+1)(|h_2|+1)}\lambda^{n_1+n_2-1-n}g_{\delta(1)}\Big\{\big(sh_2\{ sh_1\}\big) \big\{sg_{\delta(2)},\dots,sg_{\delta(n)}\big\}\Big\}\\
		\stackrel{(\ref{Eq: pre-jacobi})}{=}&\ \ \ \sum_{k=0}^{n-i}\sum_{i=1}^n\sum_{\delta\in S_n}(-1)^{\alpha_2}\lambda^{n_1+n_2-1-n}\cdot \\
		&\ \ \ \ \  g_{\delta(1)}\Big\{sh_1 \Big\{sg_{\delta(2)},\dots, sg_{\delta(i)},sh_2\{sg_{\delta(i+1)},\dots,sg_{\delta(i+k)}\},sg_{\delta(i+k+1)}, \dots, sg_{\delta(n)}\Big\}\Big\}	\\
		&+\sum_{k=0}^{n-i}\sum_{i=1}^n\sum_{\delta\in S_n}(-1)^{\alpha_3}\lambda^{n_1+n_2-1-n}\cdot \\
		&\ \ \ \ \ g_{\delta(1)}\Big\{sh_2 \Big\{sg_{\delta(2)},\dots, sg_{\delta(i)},sh_1\big\{sg_{\delta(i+1)}, \dots,sg_{\delta(i+k)}\big\},sg_{\delta(i+k+1)},\dots,sg_{\delta(n)}\Big\}\Big\},
	\end{align*}
	where the signs $(-1)^{\alpha_j}, j=1,2,3$ are the following:
	\begin{eqnarray*}(-1)^{\alpha_1}&=&\chi(\delta,g_1,\dots g_n)(-1)^{\sum\limits_{j=1}^{n-1}\varepsilon_j^\delta+1+n(|h_1|+|h_2|)+ (|h_1|+|h_2|)(|g_{\delta(1)}|+1)},\\
(-1)^{\alpha_2}&=&\chi(\delta; g_1,\dots,g_n)(-1)^{\sum\limits_{j=1}^{n-1}\varepsilon_j^\delta+1+n(|h_1|+|h_2|)+
(|h_1|+1)(|g_{\delta(1)}|+1)+(|h_2|+1)\eta_i^\delta},\\
		(-1)^{\alpha_3}		&=&\chi(\delta;g_1,\dots,g_n)(-1)^{\sum\limits_{j=1}^{n-1}\varepsilon_j^\delta+n(|h_1|+|h_2|)+(|h_1|+1)(|h_2|+1)+(|h_2|+1)(|g_{\delta(1)}|+1)
+(|h_1|+1)  \eta_i^\delta }.
	\end{eqnarray*}

	\medskip
	
	\quad \noindent $\mathrm{(2)}$    The terms with $sh_1$ contained in $l_i(\dots)$ :
	\begin{align*}
		{ B}=&\sum_{i=1}^n\sum_{\sigma\in \Sh(i,n-i)}(-1)^{\beta_1}l_{n-i+2}\Big(l_{i+1}(sh_1,g_{\sigma(1)},\dots,g_{\sigma(i)}),sh_2,g_{\sigma(i+1)},\dots,g_{\sigma(n)}\Big)\\
		=&\sum_{i=1}^n\sum_{\sigma\in \Sh(i,n-i)}(-1)^{\beta_2} l_{n-i+2}\Big(sh_2,\underbrace{l_{i+1}(sh_1,g_{\sigma(1)},\dots,g_{\sigma(i)})}_{\widehat{h_1}},g_{\sigma(i+1)},\dots,g_{\sigma(n)}\Big)\\
		=&\ \ \ \ \sum_{i=1}^n\sum_{\sigma\in \Sh(i,n-i)}\sum_{\pi\in S_{n-i}}(-1)^{\beta_3}\lambda^{n_2-(n-i+1)}\widehat{h_1} \Big\{sh_2 \big\{sg_{\sigma(i+\pi(1))},\dots,sg_{\sigma(i+\pi(n-i))}\big\}\Big\}\\
		&+	\sum_{i=1}^{n}\sum_{k=1}^{n-i}\sum_{\sigma\in \Sh(i,n-i)}\sum_{\pi\in S_{n-i}}(-1)^{\beta_4}\lambda^{n_2-(n-i+1)}\cdot \\
		&\ \ \ \ g_{\sigma(i+\pi(1))} \Big\{sh_2\big\{sg_{\sigma(i+\pi(2))},\dots,sg_{\sigma(i+\pi(k))},s\widehat{h_1},sg_{\sigma(i+\pi(k+1))},\dots,sg_{\sigma(i+\pi(n-i))}\big\}\Big\}
	\end{align*}

	\medskip
	
	By definition,
	\begin{align*}
		\widehat{h_1}&=l_{i+1}(sh_1,g_{\sigma(1)},\dots,g_{\sigma(i)})\\
		&=\sum_{\tau\in S_i}\chi(\tau,g_{\sigma(1)},\dots,g_{\sigma(i)})(-1)^{1+i(|h_1|+1)+\sum\limits_{k=1}^{i-1}\varepsilon_k^{\sigma\tau}+(|h_1|+1)(|g_{\sigma\tau(1)}|+1)}\lambda^{n_1-i}g_{\sigma\tau(1)}\Big\{sh_1\big\{sg_{\sigma\tau(2)},\dots,sg_{\sigma\tau(i)}\big\}\Big\}.
	\end{align*}
	Then we have
	\begin{align*}
		B=&\ \ \ \ \sum_{i=1}^n\sum_{\sigma\in \Sh(i,n-i)}\sum_{\pi\in S_{n-i}}\sum_{\tau\in S_i}(-1)^{\beta_5}\lambda^{n_1+n_2-n-1}\cdot\\
		&\ \ \ \ \ \Big(g_{\sigma\tau(1)}\Big\{sh_1\{sg_{\sigma\tau(2)},\dots,sg_{\sigma\tau(i)}\}\Big\}\Big)\Big\{sh_2 \big\{sg_{\sigma(i+\pi(1))},\dots,sg_{\sigma(i+\pi(n-i))}\big\}\Big\}\\
		&+\sum_{k=1}^{n-i}\sum_{i=1}^n\sum_{\sigma\in \Sh(i,n-i)}\sum_{\tau\in S_i}\sum_{\pi\in S_{n-i}}(-1)^{\beta_6}\cdot\\
		&g_{\sigma(i+\pi(1))}\Big\{sh_2\big\{sg_{\sigma(i+\pi(2))},\dots,sg_{\sigma(i+\pi(k))},sg_{\sigma\tau(1)}\big\{sh_1 \{sg_{\sigma\tau(2)},\dots,sg_{\sigma\tau(i)} \}\big\}, sg_{\sigma(i+\pi(k+1))},\dots,sg_{\sigma(i+\pi(n-i))}\big\}\Big\}\\
		\stackrel{(\ref{Eq: pre-jacobi})}{=}&\ \ \ \ \sum_{i=1}^n\sum_{\sigma\in \Sh(i,n-i)}\sum_{\pi\in S_{n-i}}\sum_{\tau\in S_i}(-1)^{\beta_5}\lambda^{n_1+n_2-n-1}g_{\sigma\tau(1)}\Big\{sh_1\big\{sg_{\sigma\tau(2)},\dots,sg_{\sigma\tau(i)}\big\},sh_2 \big\{sg_{\sigma(i+\pi(1))},\dots,sg_{\sigma(i+\pi(n-i))}\big\}\Big\}\\
		&+\sum_{i=1}^n\sum_{\sigma\in \Sh(i,n-i)}\sum_{\pi\in S_{n-i}}\sum_{\tau\in S_i}(-1)^{\beta_7}\lambda^{n_1+n_2-n-1}g_{\sigma\tau(1)}\Big\{sh_2 \big\{sg_{\sigma(i+\pi(1))},\dots,sg_{\sigma(i+\pi(n-i))}\big\},sh_1\big\{sg_{\sigma\tau(2)},\dots,sg_{\sigma\tau(i)}\big\}\Big\}\\
		&+\sum_{k=1}^i\sum_{i=1}^n\sum_{\sigma\in \Sh(i,n-i)}\sum_{\pi\in S_{n-i}}\sum_{\tau\in S_i}(-1)^{\beta_8}\lambda^{n_1+n_2-n-1}\cdot \\
		&\ \ \ \ g_{\sigma\tau(1)}\Big\{ sh_1\big\{sg_{\sigma\tau(2)},\dots,sg_{\sigma\tau(k)},sh_2 \big\{sg_{\sigma(i+\pi(1))},\dots,sg_{\sigma(i+\pi(n-i))}\big\},sg_{\sigma\tau(k+1)},\dots,sg_{\sigma\tau(i)}\big\}\Big\}\\
		&+\sum_{k=2}^i\sum_{i=1}^n\sum_{\sigma\in \Sh(i,n-i)}\sum_{\pi\in S_{n-i}}\sum_{\tau\in S_i}(-1)^{\beta_8}\lambda^{n_1+n_2-n-1}\cdot\\
		&\ \ \ \ \ \ g_{\sigma\tau(1)}\Big\{ sh_1\big\{sg_{\sigma\tau(2)},\dots,sg_{\sigma\tau(k)}\big\{sh_2 \big\{sg_{\sigma(i+\pi(1))},\dots,sg_{\sigma(i+\pi(n-i))}\big\}\big\},\dots,sg_{\sigma\tau(i)}\big\}\Big\}\\
		&+\sum_{k=1}^{n-i}\sum_{i=1}^n\sum_{\sigma\in \Sh(i,n-i)}\sum_{\tau\in S_i}\sum_{\pi\in S_{n-i}}(-1)^{\beta_6}\lambda^{n_1+n_2-n-1}\cdot\\
		&\ \ \ \ g_{\sigma(i+\pi(1))}\Big\{sh_2\big\{sg_{\sigma(i+\pi(2))},\dots,sg_{\sigma(i+\pi(k))},sg_{\sigma\tau(1)}\big\{sh_1 \{sg_{\sigma\tau(2)},\dots,sg_{\sigma\tau(i)} \}\big\}, sg_{\sigma(i+\pi(k+1))},\dots,sg_{\sigma(i+\pi(n-i))}\big\}\Big\}\\
		\overset{\star}{=}&\ \ \ \ \sum_{i=1}^n\sum_{\delta\in S_m}(-1)^{\beta_9}\lambda^{n_1+n_2-n-1}g_{\delta(1)}\Big\{sh_1\big\{sg_{\delta(2)},\dots,sg_{\delta(i)}\big\},sh_2\big\{sg_{\delta(i+1)},\dots,sg_{\delta(n)}\big\}\Big\}\\
	&+\sum_{i=1}^n\sum_{\delta\in S_n}(-1)^{\beta_{10}}\lambda^{n_1+n_2-n-1}g_{\delta(1)}\Big\{sh_2\big\{sg_{\delta(i+1)},\dots,sg_{\delta(n)}\big\}, sh_1\big\{sg_{\delta(2)},\dots,sg_{\delta(i)}\big\}\Big\}\\
	&+\sum_{k=1}^i\sum_{i=1}^n\sum_{\delta\in S_n}(-1)^{\beta_{11}}\lambda^{n_1+n_2-n-1} g_{\delta(1)}\Big\{sh_1\big\{sg_{\delta(2)},\dots,sg_{\delta(k)},sh_2\{sg_{\delta(i+1)},\dots,sg_{\delta(n)}\},sg_{\delta(k+1)},\dots,sg_{\delta(i)}\big\}\Big\}\\
	&+\sum_{k=2}^i\sum_{i=1}^n\sum_{\delta\in S_n}(-1)^{\beta_{11}}\lambda^{n_1+n_2-n-1} g_{\delta(1)}\Big\{sh_1\big\{sg_{\delta(2)},\dots,sg_{\delta(k)}\big\{sh_2\{sg_{\delta(i+1)},\dots,sg_{\delta(n)}\}\big\},sg_{\delta(k+1)},\dots,sg_{\delta(i)}\big\}\Big\}\\
	&+\sum_{k=1}^{n-i}\sum_{i=1}^n\sum_{\delta\in S_n}(-1)^{\beta_{12}}\lambda^{n_1+n_2-n-1}\cdot\\
	&\ \ \ \ \  g_{\delta(i+1)}\Big\{sh_2\big\{sg_{\delta(i+2)},\dots,sg_{\delta(i+k)},sg_{\delta(1)}\big\{sh_1\big\{sg_{\delta(2)},\dots,sg_{\delta(i)}\big\}\big\},sg_{\delta(i+k+1)},\dots,sg_{\delta(n)}\big\}\Big\}
	&\\
	=&\ \ \ \ \sum_{i=1}^n\sum_{\delta\in S_m}(-1)^{\zeta_1}\lambda^{n_1+n_2-n-1}g_{\delta(1)}\Big\{sh_1\big\{sg_{\delta(2)},\dots,sg_{\delta(i)}\big\},sh_2\big\{sg_{\delta(i+1)},\dots,sg_{\delta(n)}\big\}\Big\}\\
	&+\sum_{i=1}^n\sum_{\delta\in S_n}(-1)^{\zeta_2}\lambda^{n_1+n_2-n-1}g_{\delta(1)}\Big\{sh_2\big\{sg_{\delta(2)},\dots,sg_{\delta(i)}\big\}, sh_1\big\{sg_{\delta(i+1)},\dots,sg_{\delta(n)}\big\}\Big\}\\
	&+\sum_{k=0}^{n-i}\sum_{i=1}^n\sum_{\delta\in S_n}(-1)^{\zeta_3}\lambda^{n_1+n_2-n-1} g_{\delta(1)}\Big\{sh_1\big\{sg_{\delta(2)},\dots,sg_{\delta(i)},sh_2\{sg_{\delta(i+1)},\dots,sg_{\delta(i+k)}\},sg_{\delta(i+k+1)},\dots,sg_{\delta(n)}\big\}\Big\}\\
	&+\sum_{k=0}^{n-i}\sum_{i=2}^n\sum_{\delta\in S_n}(-1)^{\zeta_3}\lambda^{n_1+n_2-n-1} g_{\delta(1)}\Big\{sh_1\big\{sg_{\delta(2)},\dots,sg_{\delta(i)}\big\{sh_2\{sg_{\delta(i+1)},\dots,sg_{\delta(i+k)}\}\big\},sg_{\delta(i+k+1)},\dots,sg_{\delta(n)}\big\}\Big\}\\
	&+\sum_{k=0}^{n-i}\sum_{i=2}^n\sum_{\delta\in S_n} (-1)^{\zeta_4}\lambda^{n_1+n_2-n-1}g_{\delta(1)}\Big\{sh_2\big\{sg_{{\delta(2)}},\dots,sg_{\delta(i-1)}, sg_{\delta(i)}\big\{sh_1\{sg_{\delta(i+1)},\dots,sg_{\delta(i+k)}\}\big\},sg_{\delta(i+k+1)},\dots, sg_{\delta(n)}\big\}\Big\},
\end{align*}
where
\begin{eqnarray*}(-1)^{ \beta_1}&=&\chi(\sigma;g_1,\dots,g_n)(-1)^{(i+1)(n-i+1)+i+(|h_2|+1)\varepsilon_i^\sigma },\\
	(-1)^{\beta_2}&=&\chi(\sigma;g_1,\dots,g_n)(-1)^{(i+1)(n-i+1)+i+(|h_2|+1)\varepsilon_i^\sigma+1+(|h_2|+1)(i+|h_1|+\varepsilon_i^\sigma)}
=\chi(\sigma;g_1,\dots,g_n)(-1)^{(i+1)(n-i)+(|h_2|+1)(i+|h_1|)},\\
	(-1)^{\beta_3}&=&(-1)^{\beta_2}\chi(\pi;g_{\sigma(i+1)},\dots,g_{\sigma(i+n-i)})(-1)^{1+(n-i+1)(|h_2|+1)+(n-i)|\widehat{h_1}|+
\sum\limits_{k=1}^{n-i-1}\sum\limits_{j=1}^k|g_{\sigma(i+\pi(j))}|+(|\widehat{h_1}|+1)(|h_2|+1)},\\
(-1)^{\beta_4}&=&(-1)^{\beta_2}\chi(\pi;g_{\sigma(i+1)},\dots,g_{\sigma(i+n-i)})(-1)^{k+|\widehat{h_1}| (\sum\limits_{j=1}^kg_{\sigma(i+\pi(j))} )+1+(n-i+1)(|h_2|+1)+
\sum\limits_{t=1}^{n-i-1}\sum\limits_{j=1}^t|g_{\sigma(i+\pi(j))}|+(n-i-k)|\widehat{h_1}| }\\ &&
\cdot (-1)^{\sum_{j=1}^k |g_{\sigma(i+\pi(j))}|+ (|h_2|+1)(|g_{\sigma(i+\pi(1))}|+1)},\\
	(-1)^{\beta_5}&=&(-1)^{\beta_3}\chi(\tau;g_{\sigma(1)},\dots,g_{\sigma(n)})(-1)^{1+i(|h_1|+1)+\sum\limits_{j=1}^{i-1}\varepsilon_j^{\sigma\tau}+
(|h_1|+1)(|g_{\sigma\tau(1)}|+1)},\\
	(-1)^{\beta_6}&=&(-1)^{\beta_4}\chi(\tau;g_{\sigma(1)},\dots,g_{\sigma(i)})(-1)^{1+i(|h_1|+1)+\sum\limits_{j=1}^{i-1}
\varepsilon_j^{\sigma\tau}+(|h_1|+1)(|g_{\sigma\tau(1)}|+1)},\\
	(-1)^{\beta_7}&=&(-1)^{\beta_5}(-1)^{\big(|h_1|+1+\eta_i^{\sigma\tau}-(|g_{\sigma\tau(1)}|+1)\big)\big(|h_2|+1+\sum\limits_{j=1}^{n-i}(|g_{\sigma(i+\pi(j))}|+1)\big)},\\
	(-1)^{\beta_8}&=&(-1)^{\beta_5}(-1)^{\big(\eta_i^{\sigma\tau}-\eta_k^{\sigma\tau}\big)\big(|h_2|+1+\sum\limits_{j=1}^{n-i}(|g_{\sigma(i+\pi(j))}|+1)\big)},\\
	(-1)^{\beta_9}&=&\chi(\delta;g_1,\dots,g_n)(-1)^{n(|h_1|+|h_2|)+(|h_1|+1)(|g_{\delta(1)}|+1)+\sum\limits_{j=1}^{n-1}\varepsilon_{j}^\delta+(|h_2|+1)\eta^\delta_i},\\
	(-1)^{\beta_{10}}&=&(-1)^{\beta_9}(-1)^{\big(|h_2|+1+\eta^\delta_n-\eta_i^\delta\big)\big(|h_1|+|g_{\delta(1)}|+\eta_i^\delta\big)},\\
	(-1)^{\beta_{11}}&=&\chi(\delta;g_1,\dots,g_n)(-1)^{\sum\limits_{j=1}^{n-1}\varepsilon_j^\delta+n(|h_1|+|h_2|)+(|h_1|+1)(|g_{\delta(1)}|+1)+(|h_2|+1)\eta^\delta_k+(\eta^\delta_i-\eta^\delta_k)(\eta_n^\delta-\eta^\delta_i)},
	\\
	(-1)^{\beta_{12}}&=&\chi(\delta;g_1,\dots,g_n)(-1)^{\sum\limits_{j=1}^{n-1}\varepsilon^\delta_j+n(|h_1|+|h_2|)+(|h_1|+1)(|h_2|+1)+(|h_1|+1+\eta_{i}^\delta)(\eta_{i+k}^\delta-\eta^\delta_i)+(|h_1|+1)(|g_{\delta(1)}|+1)+(|h_2|+1)(|g_{\delta(i+1)}|+1)},\\
	(-1)^{\zeta_1}&=&(-1)^{\beta_9},\\
	(-1)^{\zeta_2}&=&\chi(\delta;g_1,\dots, g_n)(-1)^{\sum\limits_{j=1}^{n-1}\varepsilon_j^\delta+n(|h_1|+h_2)+(|h_1|+1)(|h_2|+1)+(|h_2|+1)(|g_{\delta(1)}|+1)+(|h_1|+1)\eta_i^\delta},\\
	(-1)^{\zeta_3}&=&\chi(\delta;g_1,\dots,g_n)(-1)^{\sum\limits_{j=1}^{n-1}\varepsilon_j^\delta+n(|h_1|+h_2)+(|h_1|+1)(|g_{\delta(1)}|+1)+(|h_2|+1)\eta^\delta_i},\\
	(-1)^{\zeta_4}&=&\chi(\delta;g_1,\dots,g_n)(-1)^{\sum\limits_{j=1}^{n-1}\varepsilon_j^\delta+n(|h_1|+|h_2|)+(|h_1|+1)(|h_2|+1)+(|h_1|+1)\eta_i^\delta+(|h_2|+1)(|g_{\delta(1)}|+1)}.
	\end{eqnarray*}

	Notice that in the step $\overset{\star}{=}$ above, we replace the triple $(\tau,\sigma,\pi)$ by its corresponding permutation $\delta\in S_n$ as in { Lemma \ref{Lem: permutation}},  and  we use Equation~(\ref{Eq: sign formulae for shuffles}).

	\medskip
	
	\quad \noindent $\mathrm{(3)}$ The computation of the terms with $sh_2$ contained in $l_i(\dots)$ is almost the same as $\mathrm{(II)}$
	
	\begin{align*}
		C=&\sum_{i=1}^n\sum_{\sigma\in S(i,n-i)}(-1)^{\gamma_1} l_{n-i+2}\Big(l_{i+1}(sh_2,g_{\sigma(1)},\dots,g_{\sigma(i)}), sh_1,g_{\sigma(i+1)},\dots,g_{\sigma(n)}\Big)\\
		=&\sum_{i=1}^n\sum_{\sigma\in \Sh(i,n-i)}(-1)^{\gamma_2}l_{n-i+2}\big(sh_1,l_{i+1}(sh_2,g_{\sigma(1)},\dots,g_{\sigma(i)}), g_{\sigma(i+1)},\dots,g_{\sigma(n)}\big)\\
		=&\ \ \ \  \sum_{i=1}^n\sum_{\delta\in S_n}(-1)^{\gamma_3}\lambda^{n_1+n_2-n-1}g_{\delta(1)}\Big\{ sh_2\big\{sg_{\delta(2)},\dots,sg_{\delta(i)}\big\}, sh_1\big\{sg_{\delta(i+1)},\dots,sg_{\delta(n)}\big\}\Big\}\\
		&+\sum_{i=1}^n\sum_{\delta\in S_n}(-1)^{\gamma_4}\lambda^{n_1+n_2-n-1}g_{\delta(1)}\Big\{sh_1\big\{sg_{\delta(2)},\dots,sg_{\delta(i)}\big\}, sh_2\big\{sg_{\delta(i+1)},\dots,sg_{\delta(n)}\big\}\Big\}\\
		&+\sum_{k=0}^{n-i}\sum_{i=1}^n\sum_{\delta\in S_n}(-1)^{\gamma_5}\lambda^{n_1+n_2-n-1}\cdot\\
		&\ \ \ \ \  g_{\delta(1)}\Big\{sh_2\big\{sg_{\delta(2)},\dots,sg_{\delta(i)},sh_1\{sg_{\delta(i+1)},\dots,sg_{\delta(i+k)}\},sg_{\delta(i+k+1)},\dots,sg_{\delta(n)}\big\}\Big\}\\
		&+\sum_{k=0}^{n-i}\sum_{i=2}^n\sum_{\delta\in S_n}(-1)^{\gamma_5}\lambda^{n_1+n_2-n-1}\cdot\\
		&\ \ \ \ \  g_{\delta(1)}\Big\{sh_2\Big\{sg_{\delta(2)},\dots,sg_{\delta(i)}\big\{sh_1\{sg_{\delta(i+1)},\dots,sg_{\delta(i+k)}\}\big\},sg_{\delta(i+k+1)},\dots,sg_{\delta(n)}\Big\}\Big\}\\
		&+\sum_{k=0}^{n-i}\sum_{i=2}^n\sum_{\delta\in S_n}(-1)^{\gamma_6}\lambda^{n_1+n_2-n-1}\cdot\\
		&\ \ \ \ \ g_{\delta(1)}\Big\{sh_1\Big\{sg_{\delta(2)},\dots,sg_{\delta(i-1)},sg_{\delta(i)}\big\{sh_2\{sg_{\delta(i+1)},\dots,sg_{\delta(i+k)}\}\big\},sg_{\delta(i+k+1)},\dots,sg_{\delta(n)}\Big\}\Big\}
	\end{align*}
where \begin{eqnarray*}
	(-1)^{\gamma_1}&=&\chi(\sigma,g_1,\dots,g_n)(-1)^{i+1+(|h_1|+1)(|h_2|+1+\varepsilon_i^\sigma)+(i+1)(n-i+1)},\\
	(-1)^{\gamma_2}&=&(-1)^{\gamma_1}(-1)^{\big(|h_1|+1\big)\big(i+|h_2|+\varepsilon_{i}^\sigma\big)+1}\\
	(-1)^{\gamma_3}&=&\chi(\delta;g_1,\dots,g_n)(-1)^{\sum\limits_{j=1}^{n-1}\varepsilon_j^\delta+1+n(|h_1|+|h_2|)+(|h_2|+1)(|g_{\delta(1)}|+1)+(|h_1|+1)(|h_2|+1)+(|h_1|+1)\eta_i^\delta}\\
	(-1)^{\gamma_4}&=&\chi(\delta;g_1,\dots,g_n)(-1)^{\sum\limits_{j=1}^{n-1}\varepsilon_j^\delta+1+n(|h_1|+|h_2|)+(|h_1|+1)(|g_{\delta(1)}|+1)+(|h_2|+1)\eta_i^\delta}\\
	(-1)^{\gamma_5}&=&\chi(\delta;g_1,\dots,g_n)(-1)^{\sum\limits_{j=1}^{n-1}\varepsilon_j^\delta+1+n(|h_1|+|h_2|)+(|h_2|+1)(|g_{\delta}|+1)+(|h_1|+1)(|h_2|+1)+(|h_1|+1)\eta_i^\delta},\\
	(-1)^{\gamma_6}&=&\chi(\delta;g_1,\dots,g_n)(-1)^{\sum\limits_{j=1}^{n-1}\varepsilon_j^\delta+1+n(|h_1|+|h_2|)+(|h_2|+1)\eta_i^\delta+(|h_1|+1)(|g_{\delta(1)}|+1)}
	\end{eqnarray*}

	Then the expansion of  Equation~(\ref{Eq:jacobi}) is just ${A+B+C}$. And one can see that the same term appears exactly twice in ${A+B+C}$ with opposite signs. Thus  we have $$ A+B+C=0.$$
	
	For the situation that $n_1 =0\ \mbox{or}\ n_2=0$, i.e., $sh_1$ or $sh_2$ belongs to $\Hom(k,sV)$, the computation for Equation~(\ref{Eq:jacobi}) is similar, but notice that $l_1(sh_1)$ or $l_1(sh_2)$ may be nonzero in this situation.

\medskip

	For the cases $\mathrm{(II)\ (III)\ (IV)}$, the computation is also similar, but there may be more terms in the expansion of Equation~(\ref{Eq:jacobi}). For example, in case $(\mathrm{(II)}$, assuming $n_1\leqslant n<n_2$, there will be terms of the following forms:
	\begin{eqnarray*}
		h_1\circ\Big(sg_{\delta(1)}\ot \cdots \ot sg_{\delta(i)}\big\{sh_2\{sg_{\delta(i+1)},\dots,sg_{\delta(i+n-n_1)}\}\big\}\ot \cdots \ot sg_{\delta(n)}\Big)\\
		g_{\delta(1)}\Big\{sh_2\big\{sg_{\delta(2)},\dots,sg_{\delta(i)},sh_1\circ(sg_{\delta(i+1)}\ot \cdots\ot sg_{\delta(i+n_1)}),\dots,sg_{\delta(n)}\big\}\Big\}
		\end{eqnarray*}
	in both $B$ and $C$. Tracking their signs, one can find that these terms will be  eliminated.

 We are done!

\section*{Appendix B: Proof of Proposition~\ref{Prop: homotopy RB-arising-from-homotopy-Rota-Baxter}}
In this appendix, we will  prove Proposition~\ref{Prop: homotopy RB-arising-from-homotopy-Rota-Baxter}.

{\noindent{{\bf{Proposition~\ref{Prop: homotopy RB-arising-from-homotopy-Rota-Baxter}}}
		\begin{itemize}
			\item[(i)]
			The pair $(V,\{\widetilde{m}_n\}_{n\geqslant 1})$     forms an $A_\infty$-algebra.  And the family of operators $\{T_n\}_{n\geqslant 1}$ defines an $A_\infty$-morphism from $(V,\{\widetilde{m}_n\}_{n\geqslant 1})$ to $(V,\{m_n\}_{n\geqslant 1})$.
			
			\item[(ii)]  These two family of operators $\{\widetilde{b}_n\}_{n\geqslant 1}\bigcup\{\widetilde{R}_n\}_{n\geqslant 1}$ is also a Maurer-Cartan element in $\mathfrak{C}_\RBA(V)$, thus a homotopy Rota-Baxter algebra structure of weight $\lambda$ on $V$.
			%
			
		\end{itemize}

		\bigskip
		
		For (i), we show that    operators $\{\widetilde{b}_n\}_{n\geqslant 1}$ satisfy the following equation:
		\begin{eqnarray*}
			\sum_{j=1}^n\widetilde{b}_{n-j+1}\{ \widetilde{b}_j\}=\sum_{  i+j+k= n,\atop  i,   k\geq 0, j\geq 1 }\widetilde{b}_{i+1+k}\circ(\id^{\ot i}\ot \widetilde{b}_j\ot \id^{\ot k})=0,\end{eqnarray*}
		which says that $(V,\{\widetilde{m_n}\}_{n\geqslant 1})$ is an $A_\infty$-algebra.

		In fact,
		\begin{align*}
			&\sum_{j=1}^n\widetilde{b}_{n-j+1}\{ \widetilde{b}_j\}\\
			=&\sum_{j=1}^n \sum_{p=1}^{n-j+1}\sum_{k=0}^{p-1}\sum_{l_1+\cdots +l_k+p-k=n-j+1,\atop l_1, \dots, l_k\geqslant 1}  \lambda^{p-k-1}\big(b_p\{ sR_{l_1},\dots,sR_{l_k}\}\big)\{\widetilde{b}_j\}\\
			=&\sum_{j=1}^n\sum_{p=1}^{n-j+1}\sum_{k=0}^{p-1}\sum_{l_1+\cdots +l_k+p-k=n-j+1,\atop l_1, \dots, l_k\geqslant 1} \sum_{r=0}^k\lambda^{p-k-1}b_p\big\{sR_{l_1},\dots,sR_{l_r},\widetilde{b}_j,sR_{l_{r+1}},\dots, sR_{l_k}\big\}\\
			&+\sum_{j=1}^n\sum_{p=1}^{n-j-1}\sum_{k=0}^{p-1}\sum_{l_1+\cdots +l_k+p-k=n-j+1,\atop l_1, \dots, l_k\geqslant 1}\sum_{r=1}^k\lambda^{p-k-1}b_p\big\{sR_{l_1},\dots,sR_{l_r}\{\widetilde{b}_j\},sR_{l_{r+1}}, \dots,sR_{l_k}\big\}\\
			=&\sum_{j=1}^n\sum_{p=1}^{n-j+1}\sum_{k=0}^{p-1}\sum_{l_1+\cdots +l_k+p-k=n-j+1,\atop l_1, \dots, l_k\geqslant 1} \sum_{r=0}^k \sum_{q=1}^j\sum_{s=0}^{q-1}  \sum_{t_1+\cdots +t_s+q-s=j,\atop t_1, \dots, t_s\geqslant 1}  \\
&\quad\quad\quad\quad \lambda^{p+q-k-s-2}b_p\Big\{sR_{l_1},\dots,sR_{l_r},b_q\big\{sR_{t_1},\dots,sR_{t_s}\big\}, sR_{l_{r+1}}, \dots,sR_{l_k}\Big\}\\
			&+\sum_{j=1}^n\sum_{p=1}^{n-j+1}\sum_{k=0}^{p-1}\sum_{l_1+\cdots +l_k+p-k=n-j+1,\atop l_1, \dots, l_k\geqslant 1}  \sum_{r=1}^k \sum_{q=1}^j\sum_{s=0}^{q-1}  \sum_{t_1+\cdots +t_s+q-s=j,\atop t_1, \dots, t_s\geqslant 1}  \\
&\quad\quad\quad\quad \lambda^{p+q-k-s-2}b_p\Big\{sR_{l_1},\dots,sR_{l_r}\big\{b_q\big\{sR_{t_1},\dots,sR_{t_s}\big\}\big\}, sR_{l_{r+1}},\dots,sR_{l_k}\Big\}\\
			=& \sum_{1\leqslant p+q\leqslant n+1, \atop p, q\geqslant 1}  \sum_{m=0}^{p+q-2}\sum_{s=0}^{q-1}\sum_{r=0}^{m-s}\sum_{l_1+\cdots +l_m=n+1-p-q+m,\atop l_1, \dots, l_m\geqslant 1}  \lambda^{p+q-m-2} b_p\Big\{sR_{l_1},\dots,sR_{l_r},b_q\big\{sR_{l_{r+1}},\dots,sR_{l_{r+s}}\big\},\dots,sR_{l_m}\Big\}\\ &  +\sum_{j=1}^n\sum_{p=1}^{n-j+1}
			   \sum_{k=0}^{p-2}  \sum_{r=0}^{k}\sum_{j_1+\cdots +j_k+p-k =n+1-j, \atop  j_1, \dots, j_k   \geqslant 1} \\
&\quad\quad\quad\quad \lambda^{p-k-1}b_p\Big\{sR_{j_1},\dots,sR_{j_r},
			\sum_{1\leqslant s\leqslant  q\leqslant  j} \sum_{t_1+\cdots+t_s+q-s=j,\atop t_1, \dots, t_s\geqslant 1} \lambda^{q-s}sR_{t_1}\Big\{b_q\big\{sR_{t_2},\dots,
			sR_{t_s}\big\}\Big\},sR_{j_{r+1}},\dots,sR_{j_k}\Big\} \\
			=& \sum_{1\leqslant p+q\leqslant n+1, \atop p, q\geqslant 1}  \sum_{k=0}^{q-1}\sum_{m-q=0}^{p-1}\sum_{r=0}^m\sum_{l_1+\cdots +l_m=n+1-p-q+m,\atop l_1, \dots, l_m\geqslant 1}  \lambda^{p+q-m-2}b_p\Big\{sR_{l_1},\dots,sR_{l_r},b_q\big\{sR_{l_{r+1}},\dots,sR_{l_{r+k}}\big\},\dots,sR_{l_m}\Big\} \\
			 &+\sum_{1\leqslant p+q\leqslant n+1, \atop p, q\geqslant 1}  \sum_{r=0}^{m-q} \sum_{l_1+\cdots +l_m=n+1-p-q-m\atop l_1, \dots, l_m\geqslant 1}\lambda^{p+q-m-2}b_p\Big\{sR_{l_1},\dots,sR_{l_r},b_q\circ(sR_{l_{r+1}},\dots,sR_{l_{r+q}}),\dots,sR_{l_m}\Big\}\\
			=&\sum_{t=2}^{n+1}  \sum_{l_1+\cdots +l_m=n+1-t-m, \atop l_1, \dots, l_m\geqslant 1}\lambda^{t-m-2}   \Big(\sum_{  p+q=t, p, q\geqslant 1}b_p\{ b_q\}\Big)\Big\{sR_{l_1},\dots,sR_{l_m}\Big\}\\
			=&0,
		\end{align*}
		where   we get  the fourth equality by reindexing,   the fifth equality is obtained from  Equation~(\ref{Eq: homotopy-rb-operator}) and the last equality uses Equation~(\ref{Eq: A infinity}).
		
			By the definition of homotopy Rota-Baxter algebra of weight $\lambda$, the three family of operators $\{b_n\}_{n\geqslant1}$, $\{\widetilde{b}_n\}_{n\geqslant 1}$, $\{R_n\}_{n\geqslant 1}$ fulfill the equation:
		\[\sum_{k=1}^n\sum_{l_1+\dots+l_k=n,\atop l_1,\dots,l_k\geqslant 1}b_k(sR_{l_1}\ot \dots\ot sR_{l_k})=\sum_{p=1}^nsR_{p}\{\widetilde{b}_{n-p+1}\}.\]
		Thus $\{T_n\}_{n\geqslant 1}$ is an $A_\infty$-morphism from $(V,\{\widetilde{m}_n\}_{n\geqslant 1})$ to $(V, \{m_n\}_{n\geqslant 1})$.
		
		\bigskip

		For (ii), we just need to check that $\{\widetilde{b}_n\}_{n\geqslant 1}\bigcup\{\widetilde{R}_n\}_{n\geqslant 1}$ fulfill Equation ~(\ref{Eq: homotopy RB-operator-version-2}).
		
		By the definition of $\{\widetilde{R}_n\}_{n\geqslant 1}$, one can check that the following equation holds:
		\begin{eqnarray}\label{Eq:expansion-of-new-rb-operator}\widetilde{R}_n&=&\sum_{k=1}^\infty R_n^k\\
			\notag	&=&\sum_{k=1}^\infty\sum\limits_{0\leqslant q\leqslant p-1,\atop t_1,\dots,t_q\leqslant k-1}\sum\limits_{l_1+\cdots +l_q+p-q=n,\atop l_1,\dots,l_q\geqslant 1}\lambda^{p-q-1}s^{-1}(sR_p)\{sR^{t_1}_{l_1},\dots,sR^{t_{q}}_{l_q}\}\\
			\notag	&=&\sum_{l_1+\cdots +l_q+p-q=n,\atop l_1,\dots,l_q\geqslant 1}\sum_{0\leqslant q\leqslant p-1}\lambda^{p-q-1}s^{-1}(sR_p)\{s\widetilde{R}_{l_1},\dots,s\widetilde{R}_{l_q}\}
		\end{eqnarray}

		Now, let's prove Equation \eqref{Eq: homotopy-rb-operator} holds for $\{\widetilde{b}_{m}\}_{m\geqslant 1}\cup\{\widetilde{R}_m\}_{m\geqslant 1}$, i.e., the following equation holds for any $n\geqslant 1$:
		\begin{eqnarray} \label{Eq: homotopy-new-rb-operator}
			&&    \sum_{k=1}^n\sum_{l_1+\cdots +l_k=n}s^{-1}\widetilde{b}_k\circ(s\widetilde{R}_{l_1}\ot \cdots\ot s\widetilde{R}_{l_k})\\
			&&\notag  \ \ \ \ \ \ \ \ \ \ \ \ \ =\sum_{p=1}^n\sum_{r_1+\cdots +r_q+p-q=n,\atop 1\leqslant q\leqslant p}\lambda^{p-q}s^{-1}(s\widetilde{R}_{r_1})\big\{\widetilde{b}_p\{s\widetilde{R}_{r_2},\dots,s\widetilde{R}_{r_q}\}\big\}.
		\end{eqnarray}
		We prove this by taking induction on $n$. When $n=1$, it is easy to see that $\widetilde{R}_1=R_1$ and $\widetilde{b}_1=b_1$. The Equation \eqref{Eq: homotopy-new-rb-operator} holds naturally for $n=1$. Now, assume that Equation \eqref{Eq: homotopy-new-rb-operator} holds for all integers $\leqslant n-1$.
		Firstly, we have the following equation holds:
		
		\begin{eqnarray*}&&\sum_{m=1}^n\sum_{l_1+\cdots +l_m=n}s^{-1}\widetilde{b}_m\circ(s\widetilde{R}_{l_1}\ot \cdots\ot s\widetilde{R}_{l_m})\\
			&=&\sum_{m=1}^n\sum_{l_1+\cdots +l_m=n}\sum_{0\leqslant k\leqslant p-1,\atop i_1+\cdots +i_k+p-k=m}\lambda^{p-k-1}\Big(b_p\{sR_{i_1},\dots,sR_{i_k}\}\Big)\circ\Big(s\widetilde{R}_{l_1}\ot \cdots \ot s\widetilde{R}_{l_m}\Big)\\
			&=&\sum_{m=1}^n\sum_{l_1+\cdots +l_m=n}\sum_{0\leqslant k\leqslant p-1,\atop i_1+\cdots +i_k+p-k=m}\sum_{0\leqslant t_1\leqslant t_2\leqslant \dots\leqslant t_k\leqslant l_m}\lambda^{p-k-1}\cdot \\
			&&\  b_p\circ\Big(s\widetilde{R}_{l_1}\ot\cdots \ot s\widetilde{R}_{l_{t_1}}\ot sR_{i_1}\circ(s\widetilde{R}_{l_{t_1+1}}\ot \cdots \ot s\widetilde{R}_{l_{t_1+i_1}})\ot \cdots \ot sR_{i_k}\circ\big( s\widetilde{R}_{l_{t_k+1}}\ot \cdots \ot s\widetilde{R}_{l_{t_k+i_k}}\big)\ot \cdots \ot s\widetilde{R}_{l_m}\Big)\\
			&\overset{\star}{=}&\sum_{m=1}^n\sum_{l_1+\cdots +l_m=n}\sum_{0\leqslant k\leqslant p-1,\atop i_1+\cdots +i_k+p-k=m}\sum_{0\leqslant t_1\leqslant t_2\leqslant \dots\leqslant t_k\leqslant l_m}\lambda^{p-k-1}\cdot \\
			&&\  b_p\circ\Bigg(\Big(\sum_{0\leqslant j_1\leqslant k_1-1,\atop r_1^1+\cdots +r^1_{j_1}+k_1-j_1=l_1} \lambda^{k_1-j_1-1}sR_{k_1}\{s\widetilde{R}_{r_{1}^1},\dots,s\widetilde{R}_{r^{1}_{j_1}}\}\Big)\ot\cdots \\
			&&\cdots \ot \Big(\sum_{0\leqslant j_{t_1}\leqslant k_{t_1}-1,\atop r_1^{t_1}+\cdots +r^{t_1}_{j_{t_1}}+k_{t_1}-j_{t_1}=l_{t_1}} \lambda^{k_{t_1}-j_{t_1}-1}sR_{k_{t_1}}\{s\widetilde{R}_{r_{1}^{t_1}},\dots,s\widetilde{R}_{r^{t_1}_{j_{t_1}}}\}\Big)\ot  sR_{i_1}\circ(s\widetilde{R}_{l_{t_1+1}}\ot \cdots \ot s\widetilde{R}_{l_{t_1+i_1}})\ot \cdots \\
			&&\cdots \ot sR_{i_k}\circ\big( s\widetilde{R}_{l_{t_k}+1}\ot \cdots \ot s\widetilde{R}_{l_{t_k+i_k}}\big)\ot \cdots \ot \Big(\sum_{0\leqslant j_m\leqslant k_m-1,\atop r_1^m+\cdots +r^m_{j_m}+k_m-j_m=l_m} \lambda^{k_m-j_m-1}sR_{k_m}\{s\widetilde{R}_{r_{1}^m},\dots,s\widetilde{R}_{r^{m}_{j_m}}\}\Big)\Bigg)\\
			&=&\sum_{p+r_1+\cdots +r_q-q=n,\atop0\leqslant q\leqslant p-1}\lambda^{p-q-1}\sum_{j_1+\cdots +j_k=p}s^{-1}\Big(sb_k\circ( sR_{j_1}\ot \cdots \ot sR_{j_k})\Big)\{s\widetilde{R}_{r_1},\dots,s\widetilde{R}_{r_q}\}.
		\end{eqnarray*}
		In the Equality $\overset{\star}{=}$ above, we replace all $s\widetilde{R}_{l_j}, j\notin \bigcup\limits_{r=1}^k\{t_r+1,\dots,t_r+i_r\}$ by their expansions in the last line  of Equation \eqref{Eq:expansion-of-new-rb-operator}.

	Now, let's compute the RHS of Equation \eqref{Eq: homotopy-new-rb-operator}.  We have:
		
		\begin{eqnarray*}&&\sum_{p=1}^n\sum_{1\leqslant q\leqslant p}\sum_{r_1+\cdots +r_q+p-q=n}\lambda^{p-q}s^{-1}(s\widetilde{R}_{r_1})\big\{\widetilde{b}_p\{s\widetilde{R}_{r_2},\dots,s\widetilde{R}_{r_q}\}\big\}\\&=&\sum_{p=1}^n\sum_{0\leqslant q\leqslant p-1}\sum_{m+r_1+\cdots +r_q+p-q-1=n}\lambda^{p-q-1}s^{-1}s\widetilde{R}_m\big\{\widetilde{b}_p\{s\widetilde{R}_{r_1},\dots,s\widetilde{R}_{r_q}\}\big\}	\\
			&=&\sum_{p=1}^n\sum_{0\leqslant q\leqslant p-1}\sum_{m+r_1+\cdots +r_q+p-q-1=n}\lambda^{p-q-1}\sum_{0\leqslant j\leqslant k-1,\atop i_1+\cdots +i_j+k-j=m}\lambda^{k-j-1}s^{-1}\Big(sR_k\big\{s\widetilde{R}_{i_1},\dots,s\widetilde{R}_{i_j}\big\}\Big)\Big\{\widetilde{b}_p\big\{s\widetilde{R}_{r_1},\dots,s\widetilde{R}_{r_q}\big\}\Big\}\\
			&=&\sum_{i_1+\cdots +i_r+p+k-r-1=n}\sum_{0\leqslant q\leqslant p-1,\atop 0\leqslant j\leqslant  r-q\leqslant k-1}\lambda^{k-(r-q)-1+p-q-1}s^{-1} (sR_{k})\Big\{s\widetilde{R}_{i_1},\dots,s\widetilde{R}_{i_j},\widetilde{b}_p\big\{s\widetilde{R}_{i_{j+1}},\dots,s\widetilde{R}_{i_{j+q}}\big\},s\widetilde{R}_{i_{j+q+1}},\dots,s\widetilde{R}_{i_r}\Big\}\\
			&&+\sum_{i_1+\cdots +i_r+p+k-r-1=n}\sum_{0\leqslant q\leqslant p-1,\atop 1\leqslant j\leqslant r-q\leqslant k-1}\lambda^{k-(r-q)-1+p-q-1}s^{-1} (sR_{k})\Big\{s\widetilde{R}_{i_1},\dots,s\widetilde{R}_{i_j}\big\{\widetilde{b}_p\{s\widetilde{R}_{i_{j+1}},\dots,s\widetilde{R}_{i_{j+q}}\}\big\},s\widetilde{R}_{i_{j+q+1}},\dots,s\widetilde{R}_{i_r}\Big\}\\
			&=&\sum_{i_1+\cdots +i_r+p+k-r-1=n}\sum_{0\leqslant q\leqslant p-1,\atop 0\leqslant j\leqslant r-q\leqslant k-1}\lambda^{k-(r-q)-1+p-q-1}s^{-1} (sR_{k})\Big\{s\widetilde{R}_{i_1},\dots,s\widetilde{R}_{i_j},\widetilde{b}_p\big\{s\widetilde{R}_{i_{j+1}},\dots,s\widetilde{R}_{i_{j+q}}\big\},s\widetilde{R}_{i_{j+q+1}},\dots,s\widetilde{R}_{i_r}\Big\}\\
			&&+\sum_{i_1+\cdots +i_{j-1}+m+i_{j+q+1}+\cdots +i_r+k-r+q=n}\sum_{1\leqslant j\leqslant r-q\leqslant k-1}\lambda^{k-(r-q)-1}s^{-1} (sR_{k})\Big\{s\widetilde{R}_{i_1},\dots,\\
			&&\ \ \ \ \ \ \ \ \ \ \ \ \ \dots,\underline{ \sum_{i_j+i_{j+1}+\cdots +i_{j+q}+p-q-1=m,\atop 0\leqslant q\leqslant p-1}\lambda^{p-q-1}s\widetilde{R}_{i_j}\big\{\widetilde{b}_p\{s\widetilde{R}_{i_{j+1}},\dots,s\widetilde{R}_{i_{j+q}}\}\big\}},s\widetilde{R}_{i_{j+q+1}},\dots,s\widetilde{R}_{i_r}\Big\}	
		\end{eqnarray*}
		Notice that in the last step of the above expansion, $m=i_j+i_{j+1}+\cdots +i_{j+q}+p-q-1\leqslant n-1$. By assumption, Equation \eqref{Eq: homotopy-new-rb-operator} holds for all integers $\leqslant n-1$, so  we  have:
		
		\begin{eqnarray*}
			\sum_{i_j+i_{j+1}+\cdots +i_{j+q}+p-q-1=m,\atop 0\leqslant q\leqslant p-1}\lambda^{p-q-1}s\widetilde{R}_{i_j}\big\{\widetilde{b}_p\{s\widetilde{R}_{i_{j+1}},\dots,s\widetilde{R}_{i_{j+q}}\}\big\}=\sum_{p=1}^m\sum_{l_1+\cdots +l_p=m}\widetilde{b}_p\circ(s\widetilde{R}_{l_1}\ot \cdots \ot s\widetilde{R}_{l_p})
		\end{eqnarray*}
		
		Replacing the underlined part in the expansion by the RHS above and reindexing,  we have
		\begin{eqnarray*}
			&&\sum_{p=1}^n\sum_{1\leqslant q\leqslant p}\sum_{r_1+\cdots +r_q+p-q=n}\lambda^{p-q}s^{-1}(s\widetilde{R}_{r_1})\big\{\widetilde{b}_p\{s\widetilde{R}_{r_2},\dots,s\widetilde{R}_{r_q}\}\big\}\\
			&=&	\sum_{i_1+\cdots +i_r+p+k-r-1=n}\sum_{0\leqslant j\leqslant r-q\leqslant k-1,\atop 0\leqslant q\leqslant p-1}\lambda^{k-(r-q)-1+p-q-1}s^{-1} (sR_{k})\Big\{s\widetilde{R}_{i_1},\dots,s\widetilde{R}_{i_j},\widetilde{b}_p\{s\widetilde{R}_{i_{j+1}},\dots,s\widetilde{R}_{i_{j+q}}\},s\widetilde{R}_{i_{j+q+1}},\dots,s\widetilde{R}_{i_r}\Big\}\\
			&&+	\sum_{i_1+\cdots +i_r+p+k-r-1=n}\sum_{0\leqslant j\leqslant r-p\leqslant k-1}\lambda^{k-(r-p+1)-1}s^{-1}sR_{k}\Big\{s\widetilde{R}_{i_1},\dots, s\widetilde{R}_{i_j},\widetilde{b}_p\circ\big(s\widetilde{R}_{i_{j+1}}\ot \cdots \ot s\widetilde{R}_{i_{j+p}}\big), \dots,s\widetilde{R}_{i_r}\Big\}\\
			&=&\sum_{r_1+\cdots +r_q+p-q=n}\sum_{0\leqslant q\leqslant p-1}\lambda^{p-q-1}s^{-1}\sum_{k=1}^p\Big(sR_{k}\{\widetilde{b}_{p-k+1}\}\Big)\Big\{s\widetilde{R}_{r_1},\dots,s\widetilde{R}_{r_q}\Big\}
		\end{eqnarray*}
		Since $\{b_k\}_{k\geqslant 1}\cup\{R_k\}_{k\geqslant 1}$ fulfill Equation \eqref{Eq: homotopy-rb-operator}, we have the equation
		\[\sum_{j_1+\dots+j_k=p}s^{-1}\Big(sb_k\circ( sR_{j_1}\ot \cdots\ot sR_{j_k})\Big)=\sum_{k=1}^ps^{-1}\Big(sR_{k}\{\widetilde{b}_{p-k+1}\}\Big)\]
		holds for all positive integer $p$.
		Then Equation \eqref{Eq: homotopy-new-rb-operator} holds for integer $n$. Thus $\{\widetilde{b}_k\}_{k\geqslant 1}\cup\{\widetilde{R}_k\}_{k\geqslant 1}$ gives a homotopy Rota-Baxter algebra structure of weight $\lambda$ on $V$.

		\bigskip

		%

 \textbf{Acknowledgements:}   The authors were  supported    by NSFC (No.   11971460, 12071137)   and  by  STCSM  (No. 18dz2271000).

 The authors are grateful to Jun Chen, Xiaojun Chen, Li Guo, Yunnan Li,  Zihao Qi,  Yunhe Sheng, Rong Tang etc       for many useful comments.
  Li Guo  read carefully  part of this paper and gave  very detailed suggestions and  kind encouragements.
    Xiaojun Chen imposed  a question which led to Subsection 6.3  and Zihao Qi added a remark which  simplified the proof of Lemma~\ref{Lem: homotopy well defined}. Dotensko kindly gave an alternative proof of Theorem~\ref{Thm: Minimal model} which is reproduced in Remark~\ref{rem: dotsenko}.  Das draw our attention to his papers \cite{Das21, Das21b}. We are very grateful to these researchers for their interests and comments.

 The authors lectured about this paper in various occasions, in particular,    at ICCM meeting in December 2020, at Capital Normal University in January 2021,  at Southest University in May 2021,   at Beijing Normal Univeristy and at Northest Normal University   in June 2021 etc. We would like to express our sincere gratitude to the organisers  for the invitations and their useful remarks.

\end{document}